\documentclass[11pt,reqno]{amsart}
\usepackage{amssymb,amsmath,amsthm,amsfonts,mathrsfs,graphicx}
\usepackage{mathtools}
\usepackage{enumerate,enumitem}
\usepackage[margin=2.5cm]{geometry}
\usepackage[colorlinks=true,linkcolor=blue]{hyperref}
\usepackage{tikz}
\usepackage{comment}

\title[Patch solutions for the 2D Loglog-Euler type equation]{A revisit of patch solutions for the 2D Loglog-Euler type equation}

\author[Changhui Tan]{Changhui Tan}
\address[Changhui Tan]{\newline Department of Mathematics, University of South Carolina, Columbia SC 29208, USA}
\email{tan@math.sc.edu}

\author[Liutang Xue]{Liutang Xue}
\address[Liutang Xue]{\newline School of Mathematical Sciences, Laboratory of Mathematics and Complex Systems (MOE), Beijing Normal University, Beijing 100875, P.R. China}
\email{xuelt@bnu.edu.cn}

\author[Zhilong Xue]{Zhilong Xue}
\address[Zhilong Xue]{\newline Academy of Mathematics and Systems Science, Chinese Academy of Sciences, Beijing 100190, P.R. China}
\email{zhilongxue@amss.ac.cn}

\thanks{\textit{Acknowledgment.} 
C. Tan is partially supported by NSF grant DMS-2238219.
L. Xue and Z. Xue are partially supported by National Key Research and Development Program of China (No. 2020YFA0712900) and National Natural Science Foundation of China (No. 12271045).
}

\makeatletter
\@namedef{subjclassname@2020}{\textup{2020} Mathematics Subject Classification}
\makeatother


\newtheorem{theorem}{Theorem}[section]
\newtheorem{lemma}[theorem]{Lemma}
\newtheorem{coro}[theorem]{Corollary}
\newtheorem{proposition}[theorem]{Proposition}

\theoremstyle{definition}
\newtheorem{definition}{Definition}[section]

\theoremstyle{remark}
\newtheorem{remark}{Remark}

\def\RR{\mathbb{R}}
\def\dd{\mathrm{d}}

\begin{document}
\allowdisplaybreaks

\begin{abstract}
In this paper, we revisit the patch solutions for a class of inviscid whole-space active scalar equations that interpolate between the 2D Euler equation and the $\alpha$-SQG equation. Compared with the 2D Euler equation in vorticity form, there is an additional Fourier multiplier $m(\Lambda)$ ($\Lambda = (-\Delta)^{1/2}$) in the Biot-Savart law. If the symbol $m$ satisfies the Osgood-type condition
$$\int_2^{+\infty} \frac{1}{r (\log r) m(r)}\dd r = +\infty$$
and certain mild assumptions, the system is referred to as the 2D Loglog-Euler type equation.

First, we prove a Yudovich-type theorem establishing the existence and uniqueness of a global weak solution for the Loglog-Euler type equation associated with bounded and integrable initial data. This result directly applies to patch solutions, which are weak solutions corresponding to patch initial data given by characteristic functions of disjoint, regular, bounded domains.

Next, we revisit the seminal result by Elgindi~\cite{Elgindi14} and provide a different proof under explicit assumptions on $m$, showing that for the 2D Loglog-Euler type equation with $C^{1,\mu}$ ($0<\mu<1$) single-patch initial data, the evolved patch boundary globally preserves the $C^{1,\mu-\varepsilon}$ regularity for any $\varepsilon \in (0,\mu)$. In contrast to the frequency-space argument in~\cite{Elgindi14}, we develop an entirely physical-space-based approach that avoids the Littlewood-Paley theory and offers advantages for potential extensions to the half-plane or bounded smooth domains.

Furthermore, we investigate the global propagation of higher-order $C^{n,\mu}$ boundary regularity for patch solutions with any $n \in \mathbb{N}^\star$, and analyze the evolution of multiple patches.
\end{abstract}

\maketitle 

\tableofcontents

\section{Introduction}
In this paper we study a family of two-dimensional (2D) active scalar equations 
\begin{equation}\label{m-SQG}
\begin{cases}
  \partial_t \omega + u\cdot \nabla \omega = 0, \qquad & (x,t)\in \mathbb{R}^2\times\RR_+,\\
  u = \nabla^\perp (-\Delta)^{-1} m(\Lambda) \omega,\qquad & (x,t)\in \mathbb{R}^2\times\RR_+,\\
  \omega|_{t=0}(x) = \omega_0(x), \qquad & x\in \mathbb{R}^2,
\end{cases}
\end{equation}
where $\nabla^\perp \triangleq (\partial_{x_2},-\partial_{x_1})$ and $\Lambda\triangleq(-\Delta)^{\frac{1}{2}}$, 
the vector $u=(u_1,u_2)$ is the velocity field and the scalar field $\omega$ can be interpreted as the vorticity 
(or density, or temperature) of the fluid. 
The operator $m(\Lambda)$ is a Fourier multiplier with the symbol $m(\xi)=m(|\xi|)$, 
which is a radial function on $\mathbb{R}^2$ satisfying the following hypotheses:
\begin{enumerate}[label=$(\mathbf{H}\theenumi)$,ref=$\mathbf{H}\theenumi$]
\item\label{H1} 
$m(r)\in C^{n+4}(\mathbb{R}_+)$, $n\in\mathbb{N}^\star$ and
\begin{align*}
  \forall\; r> 0,\quad m(r)> 0,\;\; m'(r)\geq0;\quad
  \textrm{and}\quad
  \lim\limits_{r\to 0^+} r\,m'(r) \,\,\textrm{exists};
\end{align*}
and $m'(r)$ satisfies the Mikhlin-H\"ormander condition, that is, there exists a constant $C>0$ such that
\begin{align*}
  \big|\tfrac{\dd^k}{\dd r^k} m'(r)\big| \leq Cr^{-k}\, m'(r),\quad \forall\, k=1,\cdots,n+3,\quad
  \forall\, r > 0.
\end{align*}
\item\label{H2}
Denote by
$  \widetilde{m}(r) \triangleq m( e^r).	$
Either one of the statements holds: 
\begin{enumerate}[label=$(\mathbf{H}2a)$,ref=$\mathbf{H}2a$]
\item\label{H2a} 
there exist  constants $\beta\in[0,+\infty]$, $\beta_1\in[0,+\infty)$, and $\beta_2\in(-2,+\infty)$ such that the following limits exist:
\begin{align*}
  \lim\limits_{r\rightarrow +\infty} m(r)= +\infty,\quad 
  \lim_{r\rightarrow +\infty} \frac{r(\log r)\widetilde{m}'(r) }{\widetilde{m}(r)} = \beta, \quad 
  \lim_{r\rightarrow +\infty} \frac{r\widetilde{m}'(r) }{\widetilde{m}(r)} =\beta_1,\quad 
  \lim_{r\rightarrow +\infty} \frac{r\widetilde{m}''(r) }{\widetilde{m}'(r)} = \beta_2; 
\end{align*}
\end{enumerate}
\begin{enumerate}[label=$(\mathbf{H}2b)$,ref=$\mathbf{H}2b$]
\item\label{H2b} 
there exists a constant $\alpha\in (0,2)$ such that
\begin{align*}
  \lim_{r\rightarrow +\infty} \frac{r\, m'(r)}{m(r)} = \alpha;
\end{align*}
\end{enumerate}
\begin{enumerate}[label=$(\mathbf{H}2c)$,ref=$\mathbf{H}2c$]
\item\label{H2c} 
there exists a constant $C>0$ such that 
\begin{align*}
  \lim_{r\rightarrow +\infty} m(r) = C < +\infty.
\end{align*}
\end{enumerate}
\end{enumerate}
We emphasize that the technical assumptions in \eqref{H2} can be viewed as a non-oscillatory condition on $m(r)$ near infinity.
The three cases describe different types of asymptotic growth of $m(r)$ as $r \to +\infty$: bounded $m$ in \eqref{H2c}, power-law growth in \eqref{H2b}, and intermediate growth in \eqref{H2a}.
The existence of the limits in \eqref{H2} excludes highly oscillatory behaviors of $m$ near infinity, for instance, $m(r) = \log\log(e^2 + 2r + \sin r)$.
To keep \eqref{H2a} as general as possible, we allow $\beta$ to take the value $+\infty$.

Active scalar equations \eqref{m-SQG} under assumptions \eqref{H1}-\eqref{H2} arise frequently in hydrodynamic models and have attracted considerable attention. Typical examples include the following:
\begin{itemize}
\setlength\itemsep{.5em}
\item \textit{2D Euler equation}: $m(r)\equiv 1$. In this case, $m(\Lambda) \equiv \mathrm{Id}$, and equation \eqref{m-SQG} reduces to the two-dimensional Euler equation in vorticity form, which describes the motion of an inviscid incompressible fluid in two dimensions and is a fundamental model in fluid dynamics. It is a typical example where \eqref{H2c} holds.

\item \textit{Inviscid $\alpha$-SQG equation}: $m(r) = r^\alpha$, $\alpha \in (0,2)$. Here,
\begin{align*}
  m(\Lambda) = (-\Delta)^{\frac{\alpha}{2}} = \Lambda^\alpha.
\end{align*}
For $\alpha=1$, equation \eqref{m-SQG} becomes the well-known SQG equation, a simplified model for atmospheric circulation near the tropopause \cite{HPGS95} and for ocean dynamics in the upper layers \cite{LK06}. The case $0<\alpha<1$ was introduced by C\'ordoba, Fontelos, Mancho, and Rodrigo \cite{CFMR05} as a class of models interpolating between the 2D Euler and SQG equations. This family of equations satisfy \eqref{H2b}.  

\item \textit{2D Loglog-Euler equation}: $m(r) = \log^\beta\big(1+\log(1+r^2)\big)$, $\beta \in (0,1]$. In this case, the multiplier is
\begin{equation}\label{eq:loglogEuler}
m(\Lambda) = \log^\beta(1+\log(1-\Delta)).
\end{equation}
This equation was introduced by Chae, Constantin, and Wu \cite{CCW11} as a more general framework connecting the 2D Euler and $\alpha$-SQG equations.
It was later studied independently by Dabkowski, Kiselev, Silvestre, and Vicol \cite{DKSV14}, and by Elgindi \cite{Elgindi14}.

The 2D Loglog-Euler equation is a prototypical example satisfying \eqref{H2a}. Other examples obeying \eqref{H2a} include, for instance, 
\begin{align*}
  m(r)=\log^{\beta_1}(1+r)\quad\text{and}\quad  m(r)=\log(1+\log(1+(\log(1+r)))).
\end{align*}
We refer to such models collectively as the \emph{2D Loglog-Euler type equation}; these equations are the primary focus of this paper.
\end{itemize}

The study of well-posedness issues for the 2D Euler equation originated with the classical works of Wolibner \cite{Wolibner33} and H\"older \cite{Holder33} in the 1930s, who established global well-posedness for smooth solutions with H\"older continuous vorticity. 
A significant advance was made by Yudovich \cite{Yud63}, who proved the global existence and uniqueness of weak solutions to the 2D Euler equations with bounded and decaying vorticity; see also \cite{BCD11,Chemin98,CS24,MP94,MB02} for other accessible proofs. 
The velocity field associated with a Yudovich solution is generally not Lipschitz continuous in the spatial variable, but instead log-Lipschitz in the sense that
\begin{align}\label{eq:logLip}
  |u(x,t) - u(y,t)| \leq C |x-y| \Big(1+ \log\frac{1}{|x-y|} \Big),\quad \forall\,t>0,\,\, 0<|x-y|\leq 1.
\end{align}
As observed by Bahouri and Chemin \cite{BC94}, such a log-Lipschitz velocity field
will cause the regularity of the free transported quantity to deteriorate during evolution;
one can see recent advances in \cite{Dan05,BN21b,Jeong21,DMS24}.
The global existence of a weak solution (possibly without uniqueness) 
was proved for vorticity in a wider class $L^\infty_t L^p_x$, 
$1<p<\infty$, \cite{Yud63,DM87}; 
and the uniqueness of weak solutions can also be slightly improved: 
Yudovich \cite{Yud95} extended his uniqueness result for unbounded vorticity $\omega \in \cap_{p_0\leq p<\infty} L^\infty_t L^p_x $ 
so that $\|\omega(t)\|_{L^p} \leq C \theta (p)$ and $\theta(p)$ grows moderately in $p$ (e.g. $\theta(p) =\log p$); 
Vishik \cite{Vishik99} provided a different uniqueness class that
$\omega \in L^\infty_t ( L^{p_0}\cap B_\Pi)$, $p_0\in(1,2)$ with 
$\int_1^\infty \Pi(\tau)^{-1}\dd \tau =\infty$ and
$B_\Pi \triangleq \big\{f\in \mathcal{S}'(\mathbb{R}^2)\,: \, \sum_{j=-1}^N \|\Delta_j f\|_{L^\infty} = O(\Pi(N))\big\}$;
Taniuchi \cite{Tan04} generalized \cite{Yud95} to be the following result that $\omega\in L^\infty_t Y_\Theta$ 
with $\Theta: [1,\infty)\mapsto [1,\infty)$ an increasing function satisfying the Osgood type condition 
$\int_2^\infty \frac{\dd p}{p \Theta(p)} =\infty$ and 
$Y_\Theta \triangleq \big\{f\in \mathcal{S}'(\mathbb{R}^2)\,: \, \sup_{p\in [1,\infty)}\frac{\|f\|_{L^p}}{\Theta(p)} < \infty \big\}$;
one can see \cite{BK14,BT15} for some related improvement.
Elgindi et al. \cite{EMS23} developed a novel class of solutions beyond the 
Yudovich class that admits the local well-posedness and finite-time singularity result.
In the direction of non-uniqueness for the 2D Euler equation (with or without forcing)
with vorticity in $L^\infty_t L^p_x$ ($p<\infty$), 
one can see \cite{Vishik18a,Vishik18b,ABCD+,CFMS25,BC23,Mengual24,BCK24} for the recent developments.
Except for the study related to the Yudovich theory, 
the 2D Euler equation has been actively investigated in other key aspects, 
such as ill-posedness in critical or supercritical function spaces \cite{BL15,EJ17,EM20,JK22,CMO24}, 
well-posedness in singular planar domains \cite{GL13,HZ21,HZ23}, 
(optimal) bounds for vorticity gradient growth \cite{Den09,KS14,Zlat15,DEJ24,Zlat25,JYZ25+}, etc. 

For the $\alpha$-SQG equation, the local well-posedness and ill-posedness results 
in regular Sobolev and H\"older spaces have been well established; see, e.g. \cite{Wu05,CCCG12,ConN18a} 
and \cite{Cor-Zo21,Cor-Zo22,JeoK24,CJK24}, respectively.
The global existence of weak solutions was shown in
\cite{Resnick,Mar08,CCCG12,ConN18b,NHQ18,LX19}, and the non-uniqueness result for the (forcing) 
$\alpha$-SQG equation can be referred to \cite{BSV19,CKL21,Isett21,DGR24,CFMS25b} 
and references therein. 
He and Kiselev \cite{HK21} proved that there exist solutions of 
$\alpha$-SQG equation with $\alpha \in (0,2)$ that exhibit either infinite-in-time growth of derivatives or finite-time blowup.
However, the global well-posedness of smooth solutions for the $\alpha$-SQG equation with any $\alpha\in (0,2)$ 
in the whole space $\mathbb{R}^2$ or the torus $\mathbb{T}^2$ still remains a remarkable open problem. 
We note that this problem is solved for the $\alpha$-SQG equation in some domains with boundary:
indeed, for the $\alpha$-SQG equation with $\alpha\in (0,\frac{1}{2}]$ in the half-plane, 
Zlato\v s \cite{Zlat23} proved the local well-posedness in the anisotropic Lipschitz type spaces 
and established the finite-time singularity formation in this class associated with smooth initial data 
(see also \cite{JKY25}), where the rigid boundary plays an essential role in the blowup mechanism.

Based on the study of the 2D Euler and $\alpha$-SQG equations, an intriguing question arises:
\begin{quote}
  \textit{How far can we deviate from the 2D Euler equation towards the $\alpha$-SQG equation while still maintaining global well-posedness of smooth (or non-smooth) solutions?}
\end{quote}

This question strongly motivated Chae, Constantin, and Wu \cite{CCW11} to introduce the Loglog-Euler equation \eqref{m-SQG} with multiplier \eqref{eq:loglogEuler}, for which they proved global well-posedness of smooth solutions. Later, the Osgood-type condition
\begin{align}\tag{Osg}\label{eq:Osgood-cond}
  \int_2^{+\infty}\frac{1}{r(\log r)m(r)}\dd r=\int_{\log 2}^{+\infty} \frac{\dd r}{r \widetilde{m}(r)}=+\infty.
\end{align}
was proposed as a possible criterion distinguishing well-posedness from ill-posedness. 
Elgindi \cite{Elgindi14} and Dabkowski et al. \cite{DKSV14} 
independently established the global regularity of smooth solutions to \eqref{m-SQG} when the multiplier 
$m$ satisfies the Osgood condition \eqref{eq:Osgood-cond}, together with certain mild assumptions. 
So far, it remains open whether \eqref{eq:Osgood-cond} is a critical condition in deciding the global regularity of smooth solutions for equation \eqref{m-SQG}. 

We note that if $m$ satisfies \eqref{H2c} (e.g., the 2D Euler case), \eqref{eq:Osgood-cond} holds automatically, 
whereas if $m$ satisfies \eqref{H2b} (e.g., the $\alpha$-SQG equation), \eqref{eq:Osgood-cond} fails.
Our main interest lies in the intermediate case \eqref{H2a}, for which \eqref{eq:Osgood-cond} holds if and only if $\beta \leq 1$.
See Lemma~\ref{lem:H2H3} for more details.
\vskip 1em

In this paper, we focus primarily on patch solutions, which are weak solutions of the active scalar equation \eqref{m-SQG} associated with patch-like initial data, that is, initial data consisting of either a single patch ($N = 1$) or multiple patches ($N > 1$):
\begin{equation}\label{eq:patch-data}
  \omega_0(x) = \sum_{j=1}^N a_j \mathbf{1}_{D_j(0)}(x),\quad a_j\in \mathbb{R},
  \quad \mathbf{1}_{D_j(0)}(x)
  \triangleq \begin{cases}
	1,\quad x\in D_j(0), \\
	0,\quad x\in \mathbb{R}^2\setminus D_j(0),
  \end{cases}
\end{equation}
where $D_j(0)\subset\RR^2$, $j=1,\cdots,N$ are simply connected disjoint bounded domains with regular boundaries
$\partial D_j(0)$. Patch solutions offer an effective mathematical approach for modeling the evolution of bounded domains
with concentrated scalar fields 
in fluid systems, 
particularly excelling in capturing sharp interfacial dynamics. 
For the 2D Euler equation, Yudovich's result \cite{Yud63} 
guarantees that patch solutions with initial data \eqref{eq:patch-data} 
exist uniquely and globally in time and keep the patch structure during the evolution. 
But the flow map (defined by \eqref{eq:flow_map_def}) provided in the Yudovich theory is a hemoemorphism of H\"older class $C^{\exp (-C t)}$,
thus the patch boundary is \textit{a priori} of class $C^{\exp(-Ct)}$.
Consequently, the so-called \textit{vorticity patch problem} for the 2D Euler equation was initiated in the 1980s 
which asks whether the initial smoothness of patch boundaries $\partial D_j(t)$ 
can be maintained for all time under evolution. 
This fundamental question was resolved by Chemin \cite{Chemin93}, 
who established the global persistence of $C^{k,\gamma}$-regularity with $k\in \mathbb{N}^{\star}$ 
and $\gamma\in (0,1)$ for the vortex patch boundaries. 
Bertozzi and Constantin \cite{BC93} proved the same result using a more elementary geometric argument.
For other proofs, we refer to \cite{Serf94,Radu22}. 
This global persistence of $C^{1,\gamma}$-regularity of vortex patches was extended to the 2D Euler equation in the half-plane \cite{KRYZ16} 
and the smooth bounded domain \cite{KLi19}, where the patch boundary allows to touch the rigid domain boundary 
(see \cite{Dep99,Dut03} for the previous results). 
Recently, Kiselev and Luo \cite{KL23a} established the strong ill-posedness of $C^{2}$ vortex patches. 
One can refer to \cite{Dan97,CD00,EJ23,EJo25} for the study of vortex patches with boundary singularities (e.g. corners).

For the $\alpha$-SQG equation, there is no counterpart of Yudovich's theory 
that directly implies the existence and uniqueness of patch solutions. 
Nevertheless, the evolution of patch boundaries can still be effectively analyzed through the contour dynamics equations. 
The local existence and uniqueness of the $C^{\infty}$ patches for the $\alpha$-SQG equation with 
$\alpha\in (0,1)$ were first proved by Rodrigo \cite{Rodr05}. 
By significantly using the cancellation of the curve structures, 
the local well-posedness of $\alpha$-SQG patch solutions in $L^2$-based Sobolev spaces has been established 
in the following regimes \cite{Gancedo08,CCCG12,KYZ17,CCG18,GanP21,GNP22}: 
$H^{n}$ ($n\geq 2$) for $\alpha\in (0,1)$, $H^{2+s}$ ($s>0$) for $\alpha=1$, 
and $H^n$ ($n\geq 3$) for $\alpha\in (1,2)$. 
Additionally, Kiselev and Luo \cite{KL25} demonstrated strong ill-posedness for $\alpha$-SQG patches with 
$\alpha\in (0,1)$ in H\"older space $C^{2,\gamma}$, $\gamma\in (0,1)$ 
and Sobolev space $W^{2,p}$, $p\ne 2$. 
The possible splash singularity of the $\alpha$-patches with $\alpha\in (0,2)$ has been excluded by \cite{GS14,KL23b,JZ24}.
For the $\alpha$-SQG equation in half-plane with rigid boundary, 
there is a remarkable breakthrough by Kiselev, Ryzhik, Yao, and Zlato\v s \cite{KRYZ16,KYZ17},  
which established the local well-posedness of multiple patch solutions and constructed some patch-like initial data to develop
finite-time singularity in the case $\alpha\in (0,\tfrac{1}{12})$; 
in combination with the global well-posedness result \cite{KRYZ16} for multiple vortex patches of the 2D Euler equation (i.e. $\alpha=0$ case),
this striking dichotomy highlights the critical transition at $\alpha=0$ in the behavior of patch solution.  
Subsequently, the regime of finite-time singularity formation for $\alpha$-patches was extended by Gancedo and Patel 
\cite{GanP21} to $\alpha\in (0,\tfrac{1}{3})$ (see also \cite{Zlat23} for further improvement). 

For the active scalar equation \eqref{m-SQG} with $m(\cdot)$ satisfying \eqref{H2a}, the associated vortex patch problem is slightly more singular than in the bounded-$m$ case, and the velocity field is no longer Lipschitz continuous (in contrast to the Euler case). Consequently, an $\varepsilon$-loss of regularity appears in the boundary regularity propagation of vortex patches \cite{BC94,BCD11}.
When the Osgood-type condition \eqref{eq:Osgood-cond} holds, together with other mild assumptions, Elgindi \cite{Elgindi14} identified new cancellation mechanisms and, by applying the losing-estimates method \cite{BC94}, established global $C^{1,\gamma-\varepsilon}$ ($0<\varepsilon<\gamma<1$) regularity for the evolving patches of the whole-space 2D Loglog-Euler equation \eqref{m-SQG}, starting from $C^{1,\gamma}$ patch initial data.
On the other hand, when the Osgood-type condition \eqref{eq:Osgood-cond} fails (i.e., $\int_2^{+\infty}\frac{1}{r(\log r)m(r)},\dd r < +\infty$), the authors and Miao in \cite{MTXX} proved the formation of finite-time singularities in patch solutions of \eqref{m-SQG} in the half-plane with rigid boundary conditions, thereby extending the results of \cite{KRYZ16,GanP21} for the $\alpha$-SQG equations.
These two results together suggest that the Osgood condition \eqref{eq:Osgood-cond} may serve as a sharp threshold distinguishing global regularity from finite-time blow-up in patch solutions of the model \eqref{m-SQG}.
\vskip 1em

In this paper, we revisit the seminal result of Elgindi \cite{Elgindi14} on patch solutions to the 2D Loglog-Euler type equation \eqref{m-SQG}, and provide an alternative proof under the explicit assumptions on $m$ (i.e., \eqref{H1}-\eqref{H2a}-\eqref{eq:Osgood-cond}).
In contrast to the frequency-space-based argument involving losing estimates in \cite{Elgindi14}, we develop an entirely physical-space-based approach that avoids the use of Littlewood-Paley theory.
This formulation offers several advantages, particularly for extensions to the half-plane setting (see Remark~\ref{rmk:half-plane}) or to bounded smooth domains.
Moreover, we investigate the global propagation of higher $C^{n,\gamma}$ boundary regularity for patch solutions with arbitrary $n \in \mathbb{N}^\star$, and study the evolution of multiple interacting patches.

Our first result establishes a Yudovich type theorem for the 2D Loglog-Euler type equation 
\eqref{m-SQG}, which can serve as a basis for the study of patch solutions. 
\begin{theorem}\label{thm:Yudovich_type}
Assume that $m(\xi)= m(|\xi|)$ is a radial function of $\RR^2$ with $m(r)$ satisfying \eqref{H1}-\eqref{H2a}-\eqref{eq:Osgood-cond}.
Let $\omega_0\in L^{1}\cap L^{\infty}(\RR^2)$, then the 2D Loglog-Euler type equation \eqref{m-SQG} 
admits a unique global weak solution (in the sense of Definition \ref{def:weak-solu})
\begin{align*}
  \textrm{$\omega\in L^{\infty}([0,+\infty);L^{1}\cap L^{\infty}(\RR^2))$\; and\; $\omega(x,t)=\omega_0(\Phi^{-1}_t(x))$},
\end{align*}
where the flow map $\Phi_t(x) \triangleq \Phi_{t,0}(x)$, its inverse $\Phi_t^{-1}(x)\triangleq \Phi_{0,t}(x)$, 
and $\Phi_{t,s}:\mathbb{R}^2 \rightarrow \mathbb{R}^2$ is uniquely defined by
\begin{equation}\label{eq:flow_map_def}
  \frac{\dd \Phi_{t,s}(x)}{\dd t} = u(\Phi_{t,s}(x),t),\quad
  \Phi_{t,s}(x)\big|_{t=s} = x.
\end{equation}
Furthermore, the flow map $\Phi_t(x)$ satisfies the following estimate,
\begin{align}\label{eq:bound-flow-map-general}
  \frac{1}{\mathsf{H}^{-1}\big(\mathsf{H}(|x-y|^{-1})+Ct\big)} \leq \big|\Phi^{\pm 1}_t(x)-\Phi^{\pm 1}_t(y)\big|\leq \frac{1}{\mathsf{H}^{-1}\big(\mathsf{H}(|x-y|^{-1})-Ct\big)},
\end{align}
where $C>0$ is a constant depending on $\lVert \omega_0\rVert_{L^1\cap L^{\infty}}$
and the map $r\in (0,+\infty)\mapsto \mathsf{H}(r)$ is defined by
\begin{align}\label{def:H-r}
  \mathsf{H}(r) \triangleq 
  \begin{cases}
    \displaystyle \int_2^{r} \frac{1}{\tilde{r} (\log \tilde{r}) m(\tilde{r})} \dd \tilde{r}, 
    \quad & \textrm{for}\;\; r\geq 2, \\[10pt]
    \frac{1}{(\log 2) m(2)} \log \frac{r}{2}, \quad & \textrm{for}\;\; r \in (0,2),
  \end{cases}
\end{align}
and $\mathsf{H}^{-1}(\cdot): \mathbb{R} \mapsto (0,+\infty)$ is its inverse function.
\end{theorem}

\begin{remark}\label{rmk:H}
For the Euler equation ($m(r)\equiv 1$), from the definition \eqref{def:H-r}, 
we have $\mathsf{H}(r)\approx \log\log r$ as $r \to +\infty$. 
Correspondingly, $\mathsf{H}^{-1}(\cdot)$ exhibits double-exponential growth:
\begin{align*}
  \mathsf{H}^{-1}(\mathrm{y})\approx \exp\big(C\exp(\mathrm{y})\big), \quad\text{as}\,\,\mathrm{y}\to+\infty.
\end{align*}
When $\lim_{r\to+\infty} m(r)=+\infty$, the function $\mathsf{H}(r)$ 
grows more slowly than $\log\log$, and $\mathsf{H}^{-1}(\cdot)$ 
grows faster than double exponential. 
The Osgood condition \eqref{eq:Osgood-cond} plays a crucial role in ensuring that $\mathsf{H}^{-1}(\cdot)$ is well defined on $\mathbb{R}$. 
In general, $\mathsf{H}^{-1}(\cdot)$ may grow arbitrarily fast.

The velocity field $u$ in Theorem \ref{thm:Yudovich_type} 
satisfies the continuity estimate \eqref{es:vel-Yud} below,
which is commonly referred to as an Osgood vector field (see \cite{AB08,La22})
and is slightly more singular than a log-Lipschitz field \eqref{eq:logLip}.

Finally, we note that Theorem~\ref{thm:Yudovich_type} is analogous to the results in \cite{Yud95,Vishik99,Tan04,BK14,BT15} on slightly generalized Yudovich classes.
\end{remark}

Theorem~\ref{thm:Yudovich_type} ensures the global existence and uniqueness of patch solutions to the 2D Loglog-Euler equation with initial data given by \eqref{eq:patch-data}, and guarantees that the patch structure is globally preserved:
\begin{align}\label{eq:patch-sol}
  \omega(x,t) = \sum_{j=1}^N a_j \mathbf{1}_{D_j(t)}(x),\quad \textrm{with}\quad D_j(t) = \Phi_t (D_j(0)),
\end{align}
where the flow map $\Phi_t(\cdot)=\Phi_{t,0}(\cdot)$ is given by \eqref{eq:flow_map_def}.

We now turn to the vortex patch problem concerning the global regularity of the patch boundaries. 
More precisely, we aim to address the following question:

\begin{quote}
  \textit{If the initial patch boundaries $\partial D_j(0)$ in \eqref{eq:patch-data} belong to the H\"older class $C^{n,\mu}$ with $n \in \mathbb{N}^\star$ and $\mu \in (0,1)$ for all $j \in {1,2,\dots,N}$, what is the (best possible) regularity of $\partial D_j(t)$ for every $t>0$ and $j \in {1,2,\dots,N}$?}
\end{quote}

We first consider the evolution of a single patch with $N=1$ and $a_1=1$, namely,
\begin{align}\label{intro:single-patch}
\omega(x,t) = \mathbf{1}_{D(t)}(x),
\quad \text{where} \quad D(t) = \Phi_t(D_0).
\end{align}
Our main result concerning the regularity of $\partial D(t)$ is stated below.

\begin{theorem}\label{thm:gSQG-GR-whole-space}	
Let $\mu \in (0,1)$ and $n \in \mathbb{N}^\star$.
Assume that $m(\xi) = m(|\xi|)$ is a radial function on $\mathbb{R}^2$, where $m(r)$ satisfies \eqref{H1}-\eqref{H2a}-\eqref{eq:Osgood-cond}.
Consider the unique global patch solution \eqref{intro:single-patch} of the 2D Loglog-Euler equation \eqref{m-SQG} associated with the patch initial data $\omega_0 = \mathbf{1}_{D_0}$, where $D_0 \subset \mathbb{R}^2$ is a simply connected bounded domain, with boundary $\partial D_0 \in C^{n,\mu}$.
Then, for any $t>0$, the patch boundary $\partial D(t) = \Phi_t(\partial D_0)$ almost preserves its regularity, in the sense that, $\partial D(t)\in C^{n,\mu-\varepsilon}$ for any $\varepsilon \in (0,\mu)$. 

More precisely, let $z(\cdot,t)$ be a parameterization of $\partial D(t)$ (see \eqref{eq:contour_eq} and \eqref{eq:tangential_vector_contour} for definition and construction). Then, for any given $\varepsilon>0$, $\epsilon>0$, 
there exists a constant $C>0$ depending on $\mu$, $\varepsilon$, $\epsilon$ and $\|z_0\|_{C^{n,\mu}}$, such that the following estimates hold:
\begin{align}\label{intro-est:reg-z}
  \|z(t)\|_{C^{n,\mu-\varepsilon}}\leq  C\Big(\mathsf{H}^{-1}\big(C (1+t)\log^{\beta+\epsilon}(e+t)\big)\Big)^{C\log^{\beta+\epsilon}(e+t)},
\end{align}
where the mapping $\mathsf{H}(\cdot)$ is given by \eqref{def:H-r}.
\end{theorem}

Theorem~\ref{thm:gSQG-GR-whole-space} shows that the initial $C^{n,\mu}$ boundary regularity persists in time, up to the loss of an arbitrarily small exponent. 
However, the estimate \eqref{intro-est:reg-z} indicates that the corresponding bound may grow rapidly in time, up to $\mathsf{H}^{-1}(t)$, with additional logarithmic factors. According to Remark~\ref{rmk:H}, this growth can be double-exponential in time, or even faster.
This is consistent with the results for the two-dimensional Euler case (see Remark~\ref{rmk:boundedm}).

The next result extends the boundary regularity theory to the case of multiple patches.
\begin{theorem}\label{thm:multiple-patch-main}	
Let $\mu \in (0,1)$ and $n \in \mathbb{N}^\star$.
Assume that $m(\xi) = m(|\xi|)$ is a radial function on $\mathbb{R}^2$, where $m(r)$ satisfies \eqref{H1}-\eqref{H2a}-\eqref{eq:Osgood-cond}.
Consider the unique global patch solution \eqref{eq:patch-sol} of the 2D Loglog-Euler equation \eqref{m-SQG} associated with the patch initial data \eqref{eq:patch-data}, where $D_j(0) \subset \mathbb{R}^2$ are simply connected disjoint bounded domains, with boundaries $\partial D_j(0) \in C^{n,\mu}$ for $j\in\{1,\cdots,N\}$.
Then, for any $t>0$, the patches $D_j(t)$ remain disjoint, and the boundaries $\partial D_j(t) = \Phi_t(\partial D_j(0))$ almost preserve their regularity, in the sense that, $\partial D_j(t)\in C^{n,\mu-\varepsilon}$ for any $\varepsilon \in (0,\mu)$ and $j\in \{1,\cdots,N\}$. 
\end{theorem}

A more detailed version of Theorem~\ref{thm:multiple-patch-main} is given in Theorem~\ref{thm:multi-patch}, which provides a more quantitative bound in \eqref{ineq:multi-est-2}. Compared to the bound \eqref{intro-est:reg-z} for the single-patch case, the growth behaves like $\exp(\mathsf{H}^{-1}(t))$, which is triple-exponential in time or even faster. The additional exponential factor arises from the fact that the distance between patches may decrease over time, thereby influencing the boundary regularity. See Remark~\ref{rmk:multi} for further discussion.
\vskip 1em

In what follows, we outline the main ideas underlying the proofs of our main theorems, emphasizing the key analytical mechanisms and novelties involved.

For the proof of Theorem~\ref{thm:Yudovich_type}, we first establish that the velocity field satisfies the crucial modulus-of-continuity estimate (see \eqref{eq:u-log-Lip-est}) and that the flow map enjoys the property \eqref{eq:bound-flow-map-general}. By combining these \textit{a priori} estimates with a standard approximation argument, as in the 2D Euler case, we obtain the global existence of Yudovich-type solutions.

For uniqueness, note that the classical approach for the 2D Euler equation (see \cite{Chemin93,MB02}) relies on the harmonic-analysis estimate
\begin{align*}
    \|\nabla^\perp \nabla (-\Delta)^{-1}\omega\|_{L^p}
\leq C \tfrac{p^2}{p-1} \|\omega\|_{L^p}, \qquad 1<p<\infty,
\end{align*}
for some universal constant $C>0$. However, obtaining an analogous bound for $\nabla^\perp \nabla (-\Delta)^{-1} m(\Lambda)\omega$ with unbounded $m$ appears nontrivial.
Instead, we adopt the approach of \cite[Sec 2.3]{MP94}, which is based on the analysis of the flow map.
In particular, we introduce the quantity $\delta(t)$ defined in \eqref{eq:def-diff-flow-map}, analogous to \cite[Eq.  (3.30)]{MP94}, but formulated on the whole space rather than a bounded smooth domain. Using the modulus-of-continuity estimate together with the properties of the flow map, we derive an Osgood-type inequality, which in turn implies the uniqueness of the Yudovich-type solutions under consideration.

Next, we turn to the proof of Theorem~\ref{thm:gSQG-GR-whole-space}. Following the framework described in Section \ref{subsec:formulation}, 
we represent $\partial D(t)$ using a level-set characterization: a function $\varphi$ such that $\varphi(x,t) = 0$ 
for all $x \in \partial D(t)$ (see Definition~\ref{def:levelset} for details).
Let $W \triangleq \nabla^{\perp}\varphi$ denote the tangential vector field along $\partial D(t)$, 
which satisfies equation \eqref{eq:transport-whole},
and define the tangential derivatives by
\begin{align*}
  \partial_W \triangleq W\cdot\nabla, \quad
  \textrm{and} \quad \partial_W^k \triangleq \underbrace{\partial_W \cdots \partial_W}_{k\;\textrm{times}}.
\end{align*}
As demonstrated in Section \ref{subsec:formulation}, for any $t>0$ and $n\in\mathbb{N}^\star$, $\mu\in(0,1)$, we have
\begin{align*}
  \lVert \partial_W^{n-1}W(\cdot,t)\rVert_{C^{\mu}}<+\infty \quad \Longrightarrow\quad \partial D(t)\in C^{n,\mu}.
\end{align*}

For $n=1$, the main objective is to propagate H\"older regularity $\|W(\cdot,t)\|_{C^{\mu-\varepsilon}}$. This theory has been established in \cite{Elgindi14}, through a frequency-space-based argument, in which the two quantities $\|W(t)\|_{\dot C^{\mu-\varepsilon}}$ and $|W(t)|_{\inf} \triangleq \inf_{x \in \partial D(t)} |W(x,t)|$ are propagated simultaneously, employing the $\varepsilon$-regularity-losing estimates from \cite{BC94,BCD11}.

In contrast, we develop a new approach that is entirely physical-space-based. Compared with the frequency-space framework, our method is more amenable to extension from the whole space to domains with boundaries.

Our approach relies on three key new ingredients.
\begin{itemize}
\setlength\itemsep{.5em}
\item \textit{Near-Lipschitz estimate for the velocity field $u=\nabla^\perp (-\Delta)^{-1} m(\Lambda) (\mathbf{1}_{D(t)})$}.

\item[] In Lemma~\ref{lem:continuity-u}, we show that in the single-patch setting, the velocity $u$ admits a modulus of continuity of the form
$\rho\mapsto \rho \big(m(\rho^{-1})+1\big)$ up to an additional logarithmic factor 
\begin{align*}
  f(t) \triangleq \log\Big(1+\tfrac{\|W(t)\|_{\dot C^\gamma}}{|W(t)|_{\inf}}\Big),\qquad \gamma \in (0,\mu].
\end{align*}
Compared with the velocity field associated with general vorticity $\omega \in L^1 \cap L^\infty$, which satisfies the modulus of continuity $\rho\mapsto \rho \big(\log\rho^{-1}\,m(\rho^{-1})+1\big)$ (see \eqref{def:nu}), this represents a logarithmic improvement.
Such an improvement exploits the inherent geometry of patch structure and plays a crucial role in implementing the regularity-losing estimates.

\item \textit{Key integral representation formula and refined H\"older-type estimate of $\partial_W u$}.

\item[] In Lemma~\ref{lem:tangential_velocity}, we derive a singular integral representation for $\partial_W u$, which generalizes the result of \cite[Proposition~2]{BC93} obtained for the 2D Euler case.
However, this generalization is nontrivial due to the unboundedness of $\nabla u$ and the additional singularity introduced by the multiplier $m(\Lambda)$ (e.g., \cite[p.984]{Elgindi14}). To overcome these difficulties, we adopt a slightly different and more general approach to justify the integral representation, partially inspired by \cite{Radu22}.

\item[] The refined H\"older-type estimate in Lemma~\ref{lem:striated_est_mu} also generalizes the following bound from \cite[Corollary~1]{BC93}, which holds for the 2D Euler case:
\begin{align*}
  \|\partial_W u\|_{\dot C^\mu} \leq C \|\nabla u\|_{L^\infty} \big( \|W\|_{\dot C^\mu} + 1 \big).
\end{align*}
To accommodate the presence of the multiplier $m(\Lambda)$, we introduce the $m$-adapted H\"older space $(C^\mu_{\mathrm{m}},\|\cdot\|_{C^\mu_{\mathrm{m}}})$ defined in \eqref{eq:Csig-m}. Using the near-Lipschitz estimate, we establish an analogous bound for $\|\partial_W u\|_{C^\mu_{\mathrm{m}}}$ in terms of $\|W\|_{\dot C^\mu}$, up to a logarithmic factor $f$.

\item \textit{Refined estimate of the flow map $\Phi_{t,s}$}. 

\item[] Building on the near-Lipschitz regularity of the velocity field, the flow map $\Phi_{t,s}(\cdot)$ defined in \eqref{eq:flow_map_def} satisfies a sharper continuity estimate (see \eqref{eq:flow-map-up-l}), which significantly improves upon the bound \eqref{eq:bound-flow-map-general} obtained in the Yudovich theory.
This refined property of the flow map allows us to directly analyze $W(x,t)$ along the flow map.
\end{itemize}

Combining these ingredients with the losing-estimates method within the physical-space framework, we establish the propagation estimate for the $C^{\mu-\varepsilon}$-regularity of $W(\cdot,t)$ (see \eqref{ineq:holder-part}).

To handle the logarithmic factor $f$, we derive a lower bound for $|W(t)|_{\inf}$ in \eqref{eq:inf-part}, following the approach of \cite{Elgindi14}.
The key step is to prove the point-wise estimate on $\nabla u \,W \cdot W$, as stated in Lemma~\ref{lem:es-u-whole}.

With both propagation estimates for $\|W(\cdot,t)\|_{\dot C^{\mu-\varepsilon}}$ and $|W(t)|_{\inf}$ at hand, we find that the quantity $f$ satisfies an Osgood-type inequality \eqref{ineq:f}, which ultimately yields the estimates
\begin{align}\label{intro-est:inf-1}
  |W(t)|_{\inf} & \geq |W_0|_{\inf}\Big(\mathsf{H}^{-1}\big(C (1+t)\log^{\beta+\epsilon}(e+t)\big)\Big)^{-C},\\
  \|W(t)\|_{C^{\mu-\varepsilon}(\RR^2)} & \leq  C \|W_0\|_{C^\mu}\Big(\mathsf{H}^{-1}\big(C (1+t)\log^{\beta+\epsilon}(e+t)\big)\Big)^{C\log^{\beta+\epsilon}(e+t)}.\label{intro-est:reg-1}
\end{align}
This directly leads to the desired bound \eqref{intro-est:reg-z}.

For general $n\in \mathbb{N}^\star$, the proof of Theorem~\ref{thm:gSQG-GR-whole-space} follows the same strategy as in the case $n=1$. 
One crucial step is to establish the higher-order analogue of the key integral representation formula \eqref{eq:block_higher_tangential_derivative} for $\partial_W^k u$. 
Based on this novel formula, we derive H\"older-type estimates for $\partial^k_W u$ in Lemma \ref{lem:higher_tangential_est}, treating the leading term involving $\partial_W^{k-1}W$ analogously to the $n=1$ case, and handling the lower-order striated terms separately. 
Finally, an induction argument combined with the losing-estimates method leads to the bound
\begin{align}\label{intro-est:reg-n}
  \|\partial_W^{n-1}W(t)\|_{C^{\mu-\varepsilon}(\RR^2)}  \leq  C \|\varphi_0\|_{C^{n,\mu}}\Big(\mathsf{H}^{-1}\big(C (1+t)\log^{\beta+\epsilon}(e+t)\big)\Big)^{C\log^{\beta+\epsilon}(e+t)},
\end{align}
which completes the proof of Theorem~\ref{thm:gSQG-GR-whole-space} for general $n$.

For the proof of Theorem~\ref{thm:multiple-patch-main} concerning multiple patches, we first observe that the distance between any two evolved patches admits a positive lower bound by Theorem~\ref{thm:Yudovich_type}. Consequently, for each patch, the influence of the others remains controlled, while the singular contribution still arises from the patch itself. Arguing as in the single-patch case, we then complete the proof of Theorem~\ref{thm:multiple-patch-main}.
\vskip1mm

Finally, we present some remarks as follows.
\begin{remark}[Bounded-$m$ case]\label{rmk:boundedm}
Theorems~\ref{thm:Yudovich_type}, \ref{thm:gSQG-GR-whole-space}, and \ref{thm:multiple-patch-main} remain valid if the assumption \eqref{H2a} is replaced by \eqref{H2c} with $\beta = \varepsilon = \epsilon = 0$ in the statements. Consequently, these results apply to the quasi-geostrophic shallow water equation \cite{DHR19} (i.e. $m(r)=\frac{r^2}{r^2+\lambda^2}$, $\lambda>0$). We do not aim to identify the most general conditions under which the results hold. For example, they also apply to the Euler-$\lambda$ equation~\cite{Rou23b} (i.e. $m(r)=\frac{1}{1 +\lambda^2r^2}$, $\lambda>0$), even though $m$ is not monotone increasing.
In particular, the bounds \eqref{intro-est:inf-1}-\eqref{intro-est:reg-n} for the single-patch case and \eqref{ineq:multi-est-1}-\eqref{ineq:multi-est-2} for the multi-patch case coincide with the corresponding bounds for the 2D Euler equation, as given in \cite[Eqs.~(2.13),  (2.11)]{BC93} and \cite[p.~930]{KRYZ16}, respectively.

Indeed, when $m$ is bounded, the $m$-adapted space $C^{\mu}_{\mathrm{m}}$ coincides with the standard H\"older space $C^{\mu}$. The main estimates (see Lemmas~\ref{lem:continuity-u}, \ref{lem:striated_est_mu}, \ref{lem:es-u-whole}, and \ref{lem:higher_tangential_est}) can then be adapted directly. Moreover, since the velocity field is Lipschitz continuous (Lemma~\ref{lem:continuity-u}), the losing-estimates argument is no longer required. Consequently, the global persistence of $C^{n,\mu}$ boundary regularity for patch solutions follows.  
\end{remark}

\begin{remark}[Half-plane case]\label{rmk:half-plane}
Consider the vortex patch problem for the 2D Loglog-Euler equation \eqref{m-SQG} in the half-space $\mathbb{R}^2_+$ with a no-flow (rigid) boundary condition.
If initially the patch boundaries are disjoint from the rigid boundary $\partial \mathbb{R}^2_+$, then, by performing an odd extension with respect to the $x_2$-variable and applying the argument of Theorem~\ref{thm:multiple-patch-main}, one obtains global persistence of patch boundary regularity in the half-space, analogous to the whole-space case.

If the patch boundaries touch the rigid boundary, as in the blowup scenario studied in \cite{KRYZ16}, the situation becomes more delicate.
One may attempt to combine the techniques developed in \cite{KRYZ16} (for the 2D Euler case) with the arguments of the present paper to obtain a global persistence result.
It appears that most of the necessary estimates can indeed be extended, except for one crucial estimate, an analogue of  \cite[Eq.~(3.29)]{KRYZ16}:
\begin{quote}
Given $\widetilde{B}_x \triangleq B(O_x, r_x)$ and
$u_{\widetilde{B}_x}(z) =  \nabla^{\perp}(-\Delta)^{-1}m(\Lambda)(\mathbf{1}_{\widetilde{B}_x})(z)$,
does it hold that for $|z - O_x| = r_x + d(x)$ and $d(x) \in (0, \tfrac{1}{4}r_x]$,
\begin{align*}
 \big|\nabla^2u_{\widetilde{B}_x}(z)\big| \leq C \big(1+m(d(x)^{-1})\big) r_x^{-1}?
\end{align*}
\end{quote}
Due to the nonlocal nature of the multiplier $m(\Lambda)$, it is unclear whether the above inequality holds, or, if it does, how to establish it as in the $m \equiv 1$ case.
Consequently, this interesting vortex patch problem remains open.
\end{remark}

The remainder of this paper is organized as follows. In Section \ref{sec:prelim}, we collect several useful properties of $m(r)$ and of the kernel function $G(\rho)$, and present the expression formula for the velocity $u=\nabla^\perp (-\Delta)^{-1}m(\Lambda) \omega$.
In Section~\ref{sec:Yud}, we prove Theorem~\ref{thm:Yudovich_type}, establishing the Yudovich-type result for the 2D Loglog-Euler type equation. 
Section~\ref{sec:C1mu-single} is devoted to the proof of Theorem~\ref{thm:gSQG-GR-whole-space} with $n=1$, concerning the single-patch case, while the general case $n \in \mathbb{N}^\star$ of Theorem~\ref{thm:gSQG-GR-whole-space} is proved in Section~\ref{sec:main-proof-n}. 
Finally, the proof of Theorem~\ref{thm:multiple-patch-main} on multiple patches is presented in Section~\ref{sec:mul-patch}.

\vskip 1em
\noindent\textbf{Notations.}\quad
We use the symbol $C$ to denote a generic positive constant, which may vary from line to line. The dependence of $C$ on specific parameters will be clear from the context and explicitly indicated when necessary.
We write $A \approx_p B$ to mean that there exists a constant $C > 0$, depending on $p$, such that $C^{-1}A\leq B\leq CA$.

\section{Preliminaries}\label{sec:prelim}

In this section, we shall deduce some useful properties of $m(r)$, 
write out the expression formula of the velocity field $u(x)$ 
in convolution form, and recall various estimates of the kernel function $G(\rho)$ determined by the multiplier $m$.

\subsection{\texorpdfstring{Properties of $m(r)$}{Properties of m(r)}}

We list some useful properties of $m(r)$ that satisfy the assumptions 
\eqref{H1}-\eqref{H2a}. Before proceeding, we note that the limits in \eqref{H2a} 
have the following equivalent form:
\begin{align}
  \lim_{r\rightarrow +\infty} \frac{r(\log r)(\log\log r) m'(r) }{m(r)} 
  & = \lim_{r\rightarrow +\infty} \frac{r(\log r)\widetilde{m}'(r) }{\widetilde{m}(r)} = \beta, \nonumber \\
  \lim_{r\rightarrow +\infty} \frac{r(\log r) m'(r) }{m(r)}  
  & =  \lim_{r\rightarrow +\infty} \frac{r\widetilde{m}'(r) }{\widetilde{m}(r)} =\beta_1, \label{eq:limit-beta1} \\
  \lim_{r\rightarrow +\infty} \frac{(\log r) (r m''(r) + m'(r)) }{m'(r)} 
  & = \lim_{r\rightarrow +\infty} \frac{r\widetilde{m}''(r) }{\widetilde{m}'(r)} = \beta_2, \label{eq:limit-beta2}
\end{align}
with $\beta \in [0,+\infty]$, $\beta_1\in [0,+\infty)$, $\beta_2\in (-2,+\infty)$.

\begin{itemize}
\item Under the condition \eqref{H1}, for every $\lambda>0$, 
\begin{align}\label{eq:m-prop3}
  m(\lambda r) \approx_\lambda m(r),\quad \forall\, r>0, 
\end{align}
and there exists a constant $C>0$ such that 
\begin{align}\label{eq:m-h-0}
	m'(r)\leq Cr^{-1}m(r). 
\end{align}
See, e.g. \cite[Eq. (43)]{MTXX} for the proof of \eqref{eq:m-prop3}. To obtain \eqref{eq:m-h-0}, we start with the inequality
\begin{align*}
  \frac{\dd }{\dd r}\big(rm'(r)\big)=m'(r)+rm''(r)\leq Cm'(r).
\end{align*}
Integrating in $(0,r)$ leads to the desired bound
\begin{align*}
 rm'(r) \leq \lim_{r\to 0^{+}} rm'(r) + C\big(m(r)-m(0^+)\big)\leq Cm(r),
\end{align*}
where we have used the fact that $m(0^+)\geq0$, and 
\begin{align}\label{eq:rm'(r)-limit}
  \lim_{r\to 0^{+}}rm'(r)=0. 
\end{align}
To see \eqref{eq:rm'(r)-limit}, suppose $\lim_{r\to 0^{+}}rm'(r)=c>0$. Then we have $m'(r)>\frac{c}{2r}$ if $r$ is sufficiently small. This contradicts the fact that $m'(r)$ is integrable near the origin.

\item Under the condition \eqref{H2a}, if $\beta<+\infty$,  
there exists a large constant $r_0\geq 2$, depending on $\beta$, such that 
\begin{align}\label{eq:m-prop1b}
  \forall~\varepsilon>0,\quad \textrm{$r\mapsto (\log r)^{-\beta-\varepsilon}\, \widetilde{m}(r) $\; is monotonously  decreasing on\; $[r_0,\infty)$}.
\end{align}
This is due to 
\begin{align*}
 \frac{\dd}{\dd r} (\log r)^{-\beta-\varepsilon}\widetilde{m}(r) & = (\log r)^{-\beta-\varepsilon}\widetilde{m}'(r)\Big(1-(\beta+\varepsilon) \frac{\widetilde{m}(r)}{r(\log r)\widetilde{m}'(r)}\Big), 	
\end{align*}
and 
\begin{align*}
    1-(\beta+\varepsilon) \frac{\widetilde{m}(r)}{r(\log r)\widetilde{m}'(r)}\xrightarrow{r\to+\infty}-\tfrac\varepsilon\beta<0, 
\end{align*}
with $\tilde{m}'(r)\geq 0$ for all $r\in \RR$. 
Hence $t\in [r_0,+\infty)\mapsto (\log r)^{-\beta-\varepsilon}\, \widetilde{m}(r)$
decreases for large enough $r_0$.

If $\beta=+\infty$, we apply a similar argument on the limit \eqref{eq:limit-beta1} and obtain a slightly weaker result
\begin{align*}
  \forall~\varepsilon>0,\quad \textrm{$r\mapsto (\log r)^{-\beta_1-\varepsilon} m(r)$\; is monotonously  decreasing on\; $[r_0,\infty)$},
\end{align*}
where the large constant $r_0$ depends on $\beta_1$.
As a direct consequence, we also get
\begin{align}\label{eq:m-prop2}
  m (r) \leq m(r_0) \max\big\{1, (\log_+ r)^{\beta_1+\varepsilon} \big\},
  \quad \forall\, r>0.
\end{align}

\item  
Under the condition \eqref{H2a}, we have
\begin{align}\label{eq:r-tildm-prop}
  \textrm{$r \mapsto r\,\widetilde{m}(r)$\, is monotonously increasing and convex in \;$[r_0,+\infty)$}, 
\end{align}
where we also denote the large constant by $r_0$. To see this, we compute
\begin{align*}
  \frac{\dd}{\dd r} \big(r\widetilde{m}(r)\big) = \widetilde{m}(r)\Big(1+\frac{r\widetilde{m}'(r)}{\widetilde{m}(r)}\Big),
  \quad \textrm{and} \quad
  \frac{\dd^2}{\dd r^2}  \big(r\widetilde{m}(r)\big) = \widetilde{m}'(r)\Big(2+\frac{r\widetilde{m}''(r)}{\widetilde{m}'(r)}\Big),
\end{align*}
and 
\begin{equation*}
  1+\frac{r\widetilde{m}'(r)}{\widetilde{m}(r)}
  \xrightarrow{r\to+\infty} 1+\beta_1 >0, \quad 
  2+\frac{r\widetilde{m}''(r)}{\widetilde{m}'(r)}
  \xrightarrow{r\to+\infty} 2+\beta_2>0, 
\end{equation*}
with $\widetilde{m}>0$, $\widetilde{m}'(r)\geq 0$ for all $r>0$. 

\item Suppose $\beta<+\infty$. For any $\kappa\geq1$, there exists a constant $C>0$, depending on $\kappa$ and $\beta$, such that
\begin{align}\label{eq:H3}
 \widetilde{m}(r^\kappa)\leq C\widetilde{m}(r),\quad\forall\, r>0.	
\end{align}
Indeed, the inequality follows from \eqref{eq:m-prop1b}. For $r\leq r_0$, we apply the monotonicity of $\widetilde{m}$ 
to get $\widetilde{m}(r^\kappa)\leq \widetilde{m}(r_0^{\kappa})\leq C_{r_0}\widetilde{m}(0)\leq C_{r_0}\widetilde{m}(r)$. 
For $r\geq r_0$, we deduce from \eqref{eq:m-prop1b} that
\begin{align*}
 \widetilde{m}(r^\kappa)=(\log r^\kappa)^{\beta+\varepsilon}(\log r^\kappa)^{-(\beta+\varepsilon)}\widetilde{m}(r^\kappa)\leq
 (\kappa\log r)^{\beta+\varepsilon}(\log r)^{-(\beta+\varepsilon)}\widetilde{m}(r)=\kappa^{\beta+\varepsilon}\widetilde{m}(r).
\end{align*}

The same argument can be applied to $m$ for $r\geq r_0$, replacing $\beta$ by $\beta_1$, while for $r\leq r_0$, \eqref{eq:m-prop3} implies $m(r^{\kappa})\leq m(r_0^{\kappa-1}r)\leq C_{r_0}m(r)$. It yields
\begin{align}\label{eq:m-prop5}
 m(r^\kappa)\leq Cm(r),\quad\forall\, r>0,	
\end{align}
where $C$ depends on $\kappa$ and $\beta_1$. Note that \eqref{eq:m-prop5} is weaker than \eqref{eq:H3}, but it holds when $\beta=+\infty$.

As a consequence of \eqref{eq:m-prop5}, we get that for every $\lambda>0$,
\begin{align}\label{eq:m-prop3-2}
  \widetilde{m}(\lambda r) \approx_\lambda \widetilde{m}(r),\quad \forall\, r>0.
\end{align}
Indeed, this inequality follows directly from the relation $\widetilde{m}(\lambda r)=m((e^r)^{\lambda})$, \eqref{eq:m-prop5} 
and the fact that $m'(r)\geq 0$ ($\forall r>0$). 

\item For any $\mu\in [\mu_1,\mu_2]\subset (0,1)$, there exists $C_{\mu_1} >0$ such that 
\begin{align}\label{eq:m-prop4}
  \rho_1^\mu m(\rho_1^{-1})\leq C_{\mu_1} \rho_2^\mu \big(m(\rho_2^{-1})+1\big),\quad \forall \rho_2\geq\rho_1>0.
\end{align}
Indeed, there exists some $c_{\mu_1}>0$ such that for every $\rho\in (0,c_{\mu_1})$,
\begin{align*}
  \frac{\dd (\rho^{\mu}m(\rho^{-1}))}{\dd \rho} = \rho^{\mu-1}(\mu\, m(\rho^{-1}) - \rho^{-1}m'(\rho^{-1}))
  \geq \rho^{\mu-1}m(\rho^{-1})\Big(\mu_1  - \frac{\rho^{-1}m'(\rho^{-1})}{m(\rho^{-1})}\Big) >0,
\end{align*}
since from \eqref{H2a}, we have 
\begin{align*}
\lim_{\rho\to0}\frac{\rho^{-1}m'(\rho^{-1})}{m(\rho^{-1})}=	\lim_{r\to+\infty}\frac{rm'(r)}{m(r)}=\lim_{r\to+\infty}\frac{\beta_1}{\log r}=0.
\end{align*}
Hence, \eqref{eq:m-prop4} holds with $C_{\mu_1}=1$ for the case $\rho_2\leq c_{\mu_1}$. When $\rho_2>c_{\mu_1}$, we have 
\begin{align*}
\text{For } \rho_1\leq c_{\mu_1}: &\qquad 
\rho_1^\mu m(\rho_1^{-1})\leq c_{\mu_1}^\mu m(c_{\mu_1}^{-1})\leq \rho_2^\mu m(c_{\mu_1}^{-1}),\\
\text{For } \rho_1> c_{\mu_1}: &\qquad 
\rho_1^\mu m(\rho_1^{-1})\leq \rho_2^\mu m(c_{\mu_1}^{-1}),
\end{align*}
and \eqref{eq:m-prop4} holds with $C_{\mu_1}=m(c_{\mu_1}^{-1})$.
\end{itemize}

When the Osgood condition \eqref{eq:Osgood-cond} holds, the limits in \eqref{H2} admit the following characterization.

\begin{lemma}\label{lem:H2H3}
Suppose that $m(r)$ satisfies \eqref{H1}-\eqref{H2}. Then:
\begin{itemize}
    \item If \eqref{H2a} holds, then \eqref{eq:Osgood-cond} imposes
     $\beta\leq 1$, $\beta_1=0$, and $\beta_2=-1$.	
    \item If \eqref{H2b} holds, then \eqref{eq:Osgood-cond} must fail.
    \item If \eqref{H2c} holds, then \eqref{eq:Osgood-cond} must hold.    
\end{itemize}
\end{lemma}

\begin{proof}
We first show that under assumptions \eqref{H2a} and \eqref{eq:Osgood-cond}, we have $\beta\leq1$. 
Through a similar argument as in \eqref{eq:m-prop1b}, we get 
\begin{align*}
  \forall~\varepsilon>0,\quad \textrm{$r\mapsto (\log r)^{-\beta+\varepsilon}\, \widetilde{m}(r) $\; is monotonously increasing on\; $[r_0,\infty)$},
\end{align*}	
for a large constant $r_0$. Suppose $\beta>1$. Taking $\varepsilon=\frac{\beta-1}{2}>0$, we have
\begin{align*}
 \widetilde{m}(r) = (\log r)^{1+\varepsilon} (\log r)^{-(1+\varepsilon)} \widetilde{m}(r) \geq (\log r)^{1+\varepsilon} (\log r_0)^{-(1+\varepsilon)} \widetilde{m}(r_0) = C_{r_0} (\log r)^{1+\varepsilon},\quad \forall\,r\geq r_0.
\end{align*}
This further implies
\begin{align*}
 \int_{r_0}^{+\infty}\frac{\dd r}{r\widetilde{m}(r)} \leq \frac{1}{C_{r_0}}\int_{r_0}^{+\infty}\frac{\dd r}{r(\log r)^{1+\varepsilon}} =\frac{1}{C_{r_0}\varepsilon(\log r_0)^{\varepsilon}}<+\infty,	
\end{align*}
which contradicts the Osgood condition \eqref{eq:Osgood-cond}. Hence, we argue by contradiction and conclude that $\beta\leq1$.

For $\beta_1$, we have
\begin{align*}
 \beta_1 = 	\lim_{r\rightarrow +\infty} \frac{r\widetilde{m}'(r) }{\widetilde{m}(r)} =\lim_{r\rightarrow +\infty} \frac{\beta}{\log r} \leq \lim_{r\rightarrow +\infty} \frac{1}{\log r}=0.
\end{align*}
On the other hand, $\beta_1\geq0$ due to \eqref{H1}. Therefore, $\beta_1=0$.

Finally, we apply l'H\^{o}pital's rule and get
\begin{align*}
 \beta_1 = 	\lim_{r\rightarrow +\infty} \frac{r\widetilde{m}'(r) }{\widetilde{m}(r)} = \lim_{r\rightarrow +\infty} \frac{\widetilde{m}'(r) + r \widetilde{m}''(r)}{\widetilde{m}'(r)} = 1+\beta_2.	
\end{align*}
We conclude with $\beta_2=\beta_1-1=-1$.

Under assumption \eqref{H2b}, we have $m(r) \geq C r^{\alpha/2}$ for $r>r_0$ with a large enough $r_0$. Hence,
\begin{align*}
 \int_2^{+\infty}\frac{1}{r(\log r)m(r)}\dd r \leq \int_2^{r_0}\frac{1}{r(\log r)m(r)}\dd r + \int_{r_0}^{+\infty}\frac{1}{r(\log r)\,Cr^{\frac{\alpha}{2}}}\dd r<+\infty.
\end{align*}
Therefore, the Osgood condition \eqref{eq:Osgood-cond} fails.

Finally, under assumption \eqref{H2c}, we have $m(r)\leq C$ for any $r>0$, and the Osgood condition \eqref{eq:Osgood-cond} must hold:
\begin{align*}
 \int_2^{+\infty}\frac{1}{r(\log r)m(r)}\dd r \geq \int_2^{+\infty}\frac{1}{Cr(\log r)}\dd r=+\infty.
\end{align*}
\end{proof}

\subsection{\texorpdfstring{Expression formula of the velocity $u(x)$}{Expression formula of the velocity u(x)}}
We gather some useful expression formulas related to $u(x)$.
\begin{lemma}\label{lem:int-exp}
Let $u(x) = \nabla^{\perp}(-\Delta)^{-1}m(\Lambda) \omega(x)$,
where $\nabla^\perp = (-\partial_{x_2}, \partial_{x_1})$,
$m(\Lambda)$ is a Fourier multiplier operator with the symbol
$m(\xi)=m(|\xi|)$ a radial function satisfying that
$m(\xi)\in C^2(\RR^2\setminus \{0\})$ and $\lim\limits_{r\to 0^{+}}m(r)$,
$\lim\limits_{r\to 0^{+}}rm'(r)$ exist, and also
\begin{align*}
  \lim\limits_{r\to +\infty}r^{-\frac{1}{2}}m(r)=0,\quad
  \lim\limits_{r\rightarrow +\infty} r^{\frac{1}{2}} m'(r) =0.
\end{align*}
Then the following statements hold true.
\begin{enumerate}
\item
For every $x\in \RR^2$ we have
\begin{align}\label{eq:u-exp1}
  u(x) = K\ast \omega(x) = \int_{\RR^2} K(x-y) \omega(y) \dd y,\quad
  K(x)  \triangleq  \frac{x^{\perp}}{|x|^2} G(|x|),
\end{align}
where  $x^\perp = (x_2,-x_1)$, and
\begin{equation}\label{eq:G-exp1}
\begin{split}
  G(\rho) & \triangleq  \frac{1}{2\pi}m(0^+) + \frac{1}{2\pi} \int_0^\infty J_0(\rho r)\, m'(r) \dd r \\
  & =  \frac{1}{2\pi} m(0^+) +  \frac{1}{(2\pi)^2} \int_{\RR^2} e^{i x\cdot \xi} \frac{m'(|\xi|)}{|\xi|} \dd \xi ,
\end{split}
\end{equation}
with $m(0^+)\triangleq \lim\limits_{r\rightarrow 0^+}m(r)$ and $J_0(\cdot)$ the zero-order Bessel function.
\item For every $x\in \mathbb{R}^2$, the symmetric part of $\mathbf{S}(\nabla u)(x)$ can be expressed as
\begin{align}\label{def:Sym-nab-u}
  \mathbf{S}(\nabla u)(x) \triangleq \frac{\nabla u(x) + (\nabla u)^T(x)}{2} = \mathrm{p.v.}\int_{\RR^2} (\nabla K)_{\mathrm{S}}(x-y) \omega(y) \dd y,
\end{align}
where
\begin{align}\label{def:sigma}
  (\nabla K)_{\mathrm{S}}(x) \triangleq \frac{1}{2} \frac{2G(|x|)-|x|G'(|x|)}{|x|^2} \sigma(x),\quad
  \sigma(x)\triangleq
  \begin{pmatrix}
	\frac{2x_1x_2}{|x|^2} &	\frac{x_2^2-x_1^2}{|x|^2}\\
	\frac{x_2^2-x_1^2}{|x|^2} & \frac{-2x_1x_2}{|x|^2}
  \end{pmatrix}.
\end{align}
\end{enumerate}
\end{lemma}

\begin{remark}
We remind that, if $\omega(x) = \mathbf{1}_{D}(x)$ where $D\subset \mathbb{R}^2$ is a simply connected bounded open domain,
then $\nabla u(x)$ is well defined for any $x\not\in D$ and $\nabla u(x)=\nabla K\ast \omega(x)$,
with $\nabla K(\cdot)$ given by \eqref{def:widetid-K}. In addition, for $x\in D$, we also see that
\begin{align*}
  u(x)=\int_{D} K(x-y)\dd y=\int_{D\setminus B(x,d_x)} K(x-y)\dd y,
\end{align*}
where $K(x) = \frac{x^{\perp}}{|x|^2} G(|x|)$ and $d_x \triangleq d(x,\partial D)$.
Consequently, $\nabla u(x)$ is bounded outside and inside the domain $D$,
but it may be unbounded on the boundary $\partial D$.
\end{remark}

\begin{proof}[Proof of Lemma \ref{lem:int-exp}]
The statement $(\textrm{i})$ has been proved in \cite[Lemma 1]{MTXX}. In the following, we only prove statement $(\textrm{ii})$.

We start with the calculation of $\nabla u$ in the distributional sense. For every test function $\widetilde\chi\in C^\infty_c(\RR^2)$, we apply the expression \eqref{eq:u-exp1} and get
\begin{align*}
 & \big(\partial_{x_j} u_i, \widetilde\chi\big) = - \big(u_i, \partial_{x_j} \widetilde\chi\big)
 = -\int_{\RR^2}\int_{\RR^2}K_i(x-y)\omega(y)\,\dd y\,\partial_{x_j} \widetilde\chi(x)\,\dd x\\
 & = -\lim_{\varepsilon\to0}\int_{\RR^2}\int_{|x-y|\geq\varepsilon} K_i(x-y)\partial_{x_j}\widetilde\chi(x)\,\dd x\,\omega(y)\,\dd y\\
 & = \lim_{\varepsilon\to0}\int_{\RR^2}\Big(\int_{|x-y|\geq\varepsilon} \partial_{x_j}K_i(x-y)\widetilde\chi(x)\,\dd x - \int_{|x-y|=\varepsilon} K_i(x-y)\frac{x_j-y_j}{|x-y|}\widetilde\chi(x)\,\dd S(x)\Big)\omega(y)\,\dd y.
\end{align*}
We focus on the term 
\begin{align*}
 I_\varepsilon^{i,j}(y) \triangleq \int_{|x-y|=\varepsilon} K_i(x-y)\frac{x_j-y_j}{|x-y|}\widetilde\chi(x)\,\dd S(x).
\end{align*}
Since the limit $\lim_{\varepsilon\to0}I_\varepsilon^{i,j}(y)$ may be unbounded, we need to explore extra cancellations. Define 
\begin{align*}
 J_\varepsilon^{i,j}(y) \triangleq \widetilde\chi(y)\int_{|x-y|=\varepsilon} K_i(x-y)\frac{x_j-y_j}{|x-y|}\,\dd S(x).
\end{align*}
Then, we control the difference by
\begin{align*}
 |I_\varepsilon^{i,j}(y) - J_\varepsilon^{i,j}(y)| & = \Big|\int_{|x-y|=\varepsilon} K_i(x-y)\frac{x_j-y_j}{|x-y|}\big(\widetilde\chi(x)-\widetilde\chi(y)\big)\,\dd S(x)\Big|\\
 &\leq \frac{G(\varepsilon)}{\varepsilon} \cdot 1 \cdot \lVert\nabla \widetilde\chi\rVert_{L^\infty}\varepsilon\cdot 2\pi\varepsilon
 \leq C \varepsilon \,m(\varepsilon^{-1})\xrightarrow{\varepsilon\to0}0,
\end{align*}
where we have used \eqref{eq:G-prop1} (the proof is independent of the statement $(\textrm{ii})$) and \eqref{eq:m-prop2}.
For $J_\varepsilon^{i,j}(y)$, by symmetry we have
\begin{align*}
	J_{\varepsilon}^{1,1}(y) = -J_\varepsilon^{2,2}(y)  &= \widetilde\chi(y)\int_{|x-y|=\varepsilon} \frac{(x_1-y_1)(x_2-y_2)}{|x-y|^3}\,\dd S(x) = 0,\\
	J_{\varepsilon}^{1,2}(y) + J_{\varepsilon}^{2,1}(y) & = \widetilde\chi(y)\int_{|x-y|=\varepsilon} \frac{(x_1-y_1)^2-(x_2-y_2)^2}{|x-y|^3}\,\dd S(x) = 0.
\end{align*}
Therefore, we conclude with
\begin{align*}
 \mathbf{S}(\nabla u)(x) & = \lim_{\varepsilon\to0}\int_{\RR^2}\int_{|x-y|\geq\varepsilon}  \frac{\nabla K(x-y)+(\nabla K)^T(x-y)}{2}\widetilde\chi(x)\,\dd x \omega(y) \dd y\\
 &\quad + \lim_{\varepsilon\to0}\int_{\RR^2}  \begin{pmatrix}
	I_\varepsilon^{1,1} & \frac{I_\varepsilon^{1,2}(y)+I_\varepsilon^{2,1}(y)}{2}\\
	\frac{I_\varepsilon^{1,2}(y)+I_\varepsilon^{2,1}(y)}{2} & I_\varepsilon^{2,2}(y)
  \end{pmatrix}\omega(y)\,\dd y\\
  & = \int_{\RR^2}\Big(\mathrm{p.v.}\int_{\RR^2} (\nabla K)_{\mathrm{S}}(x-y) \omega(y) \dd y\Big)\widetilde\chi(x)\,\dd x,
\end{align*}
where direct computation yields
\begin{align}\label{def:widetid-K}
  \nabla K(x) =\frac{1}{2} \frac{2G(|x|)-|x|G'(|x|)}{|x|^2} \sigma(x)
  -\frac{1}{2} \frac{G'(|x|)}{|x|}
  \begin{pmatrix}
	0 &	1\\
	-1 & 0
  \end{pmatrix}
 \triangleq (\nabla K)_{\mathrm{S}}(x)+(\nabla K)_{\mathrm{A}}(x),
\end{align}
where $\sigma(x)$ defined by \eqref{def:sigma}, also $(\nabla K)_{\mathrm{S}}=\frac{\nabla K + (\nabla K)^T}{2}$ and $(\nabla K)_{\mathrm{A}}=\frac{\nabla K - (\nabla K)^T}{2}$ are the symmetric and antisymmetric parts of $\nabla K$, respectively.
This finishes the proof of \eqref{def:Sym-nab-u}.
\end{proof}

\subsection{\texorpdfstring{Properties of $G(\rho)$}{Properties of G(rho)}}
We collect some crucial estimates of $G(\rho)$ given by \eqref{eq:G-exp1} under suitable assumptions on $m$.
\begin{lemma}\label{lem:G-prop}
  Assume that $m(r)= m(|\xi|)$ satisfies \eqref{H1}-\eqref{H2a} with $n\in \mathbb{N}^\star$. 
Then there exist constants $\bar{c}_0 >0$ and $C>0$ such that $G(\rho)=G(|x|)$ defined by \eqref{eq:G-exp1} verifies the following statements:
\begin{align}\label{eq:G-prop1}
  \textrm{for}\,\,\rho\in(0,\bar{c}_0),\quad \quad G(\rho) \approx m(\rho^{-1}), 
\end{align}
and
\begin{align}\label{eq:G-prop1b}
  \textrm{for}\,\,\rho\in(0,\bar{c}_0),\quad \quad
  |G^{(l)}(\rho)| \leq \frac{C m(\rho^{-1})}{\rho^{l}},\quad \forall\, l\in \{1,2,\cdots,n+1\},
\end{align}
and
\begin{align}\label{eq:G-prop2}
\textrm{for}\,\,\rho\in [\bar{c}_0,+\infty),\quad \quad 
  |G^{(l)}(\rho)| \leq \frac{C}{\rho^l},\quad \forall\, l\in\{1,2,\cdots, n+1\}.
\end{align}
\end{lemma}

Lemma \ref{lem:G-prop} can be proved by examining the explicit expression \eqref{eq:G-exp1} of $G$. 
See, e.g. \cite[Lemma 2]{MTXX} for the proof of \eqref{eq:G-prop1} 
and \eqref{eq:G-prop1b} with $l=1,2$ 
(noting that assumptions (\textbf{H}$1$)-(\textbf{H}$2a$) in \cite[Lemma 2]{MTXX} are readily verified). 
For the sake of completeness, we include a proof of \eqref{eq:G-prop1b} and \eqref{eq:G-prop2} in Appendix \ref{sec:app}.

As a direct consequence, we obtain the following estimate on the kernel $K$.
\begin{coro}
 Assume that $m(r)= m(|\xi|)$ satisfies \eqref{H1}-\eqref{H2a} with $n\in \mathbb{N}^\star$, and  
$K$ is defined in \eqref{eq:u-exp1}. Then
\begin{align}\label{eq:Kbound}
  |\nabla^l K(x)|\leq C|x|^{-(l+1)}\big(m(|x|^{-1})+1\big),\quad\forall\,x\in\RR^2, \quad \forall\,l\in\{0,1,\cdots,n+1\}.
\end{align}
\end{coro}

\begin{proof}
 From the definition of $K$ in \eqref{eq:u-exp1}, applying the Leibniz rule and Lemma \ref{lem:G-prop}, we compute
 \begin{align*}
  	|\nabla^l K(x)| \leq \sum_{j=0}^l \frac{l!}{j!}\frac{\big|G^{(j)}(|x|)\big|}{|x|^{l-j+1}} 
  	\leq \sum_{j=0}^l\frac{l!}{j!}\frac{C\big(m(|x|^{-1})+1\big)}{|x|^j\cdot|x|^{l-j+1}} \leq \frac{C\big(m(|x|^{-1})+1\big)}{|x|^{l+1}},
 \end{align*}
 for any $l\in\{0,1,\cdots,n+1\}$.
\end{proof}

\section{Yudovich type theorem for the 2D Loglog-Euler type equation}\label{sec:Yud}
In this section, we establish the Yudovich-type theorem for the 2D Loglog-Euler equation \eqref{m-SQG} 
associated with initial data $\omega_0\in L^1\cap L^\infty(\mathbb{R}^2)$.

We first introduce the definition of weak solutions for the equation \eqref{m-SQG}.
\begin{definition}[Weak solutions]\label{def:weak-solu}
  Let $\omega_0 \in L^1\cap L^\infty(\mathbb{R}^2)$. 
We say that $(u,\omega)$ is a weak solution to the 2D Loglog-Euler type equation \eqref{m-SQG}
with initial data $\omega_0(x)$, provided that for any $T>0$,
\begin{enumerate}
\item $\omega\in L^\infty([0,T]; L^1\cap L^\infty(\mathbb{R}^2))$. 
\item $ u = \nabla^\perp (-\Delta)^{-1} m(\Lambda) \omega = K * \omega$ with the kernel $K$ given by \eqref{eq:u-exp1}.
\item For every $\widetilde{\chi}\in C^1([0,T]; C^1_0(\mathbb{R}^2))$,
\begin{align*}
  \int_{\mathbb{R}^2} \omega(x,T) \widetilde{\chi}(x,T) \dd x 
  - \int_{\mathbb{R}^2} \omega_0(x) \widetilde{\chi}(x,0) \dd x 
  = \int_0^T \int_{\mathbb{R}^2} \omega(x,t)\, \big(\partial_t \widetilde{\chi} + u\cdot \nabla \widetilde{\chi} \big)(x,t)
  \dd x \dd t.
\end{align*}
\end{enumerate}
\end{definition}

The next result is concerned with the continuity estimate of the velocity field associated with bounded and integrable vorticity.
\begin{lemma}\label{lem:u-est-general-data}
Assume that $m(\xi)= m(|\xi|)$ is a radial function of $\RR^2$ satisfying \eqref{H1}-\eqref{H2a}.
For any $\omega\in L^1\cap L^\infty(\RR^2)$, the vector $u=\nabla^\perp (-\Delta)^{-1}m(\Lambda)(\omega)$
satisfies
\begin{align}\label{es:vel-Yud}
  \lVert u\rVert_{L^{\infty}(\RR^2)}+\sup_{|x-\tilde{x}|>0}\frac{|u(x)-u(\tilde{x})|}{\nu(|x-\tilde{x}|)}
  \leq C\lVert \omega\rVert_{L^{1}\cap L^{\infty}(\RR^2)},
\end{align}
where $C>0$ is a universal constant and the modulus of continuity $\nu$ is defined as
\begin{align}\label{def:nu}
  \nu(\rho)\triangleq 
  \begin{cases}
    \rho (\log \rho^{-1}) m(\rho^{-1}), \quad & \textrm{for}\;\; 0< \rho\leq \frac{1}{2}, \\
    \rho\, (\log 2) \, m(2), \quad & \textrm{for}\;\; \rho> \frac{1}{2}.
  \end{cases}
\end{align}
\end{lemma}

\begin{proof}
We start by applying \eqref{eq:u-exp1} and Lemma~\ref{lem:G-prop} to deduce that
\begin{align*}
  |u(x)|& \leq \Big|\int_{\{|x-y|\leq \bar{c}_0\}} \frac{(x-y)^{\perp}}{|x-y|^2}G(|x-y|)\omega(y)\dd y\Big|
  +  \Big|\int_{\{|x-y|\geq \bar{c}_0 \}} \frac{(x-y)^{\perp}}{|x-y|^2}G(|x-y|)\omega(y)\dd y\Big| \\
  & \leq C\lVert \omega\rVert_{L^{\infty}}\int_0^{\bar{c}_0} m(\rho^{-1}) \dd \rho
  + C \int_{\{|x-y|\geq \bar{c}_0\}}|\omega(y)|\dd y 
  \leq C \lVert \omega\rVert_{L^1\cap L^{\infty}}.
\end{align*}
Here, $\bar{c}_0 > 0$ denotes the constant introduced in Lemma~\ref{lem:G-prop}. Note that from 
\eqref{eq:m-prop2}, $m(\rho^{-1})$ behaves like $|\log\rho|^{\beta_1+}$ near zero. Hence, 
the integral $\int_0^{\bar{c}_0} m(\rho^{-1}) \dd \rho$ is finite.
Consequently, $u\in L^\infty(\RR^2)$.

Next, we will focus on the continuity property of $u$.
For any $\tilde{x}, x\in\RR^2$ such that $|x-\tilde{x}|<\frac{\bar{c}_0}{3}$ (without loss of generality assuming 
$\bar{c}_0\leq 1$), we apply \eqref{eq:u-exp1} and obtain
\begin{align*}
  |u(x)-u(\tilde{x})| & \leq \int_{\RR^2}\big|K(|x-y|) - K(|\tilde{x}-y|)\big|\cdot|\omega(y)|\dd y\\
  & \leq \int_{\{|x-y|\leq 2|x-\tilde{x}|\}}\big|K(|x-y|) - K(|\tilde{x}-y|)\big|\cdot|\omega(y)|\dd y \\
  & \quad + \int_{\{|x-y|\geq 2|x-\tilde{x}|\}}\big|K(|x-y|) - K(|\tilde{x}-y|)\big|\cdot|\omega(y)|\dd y \\
  & \triangleq  I+ II.
\end{align*}
For $I$, we deduce from \eqref{eq:G-prop1} that
\begin{align*}
  I & \leq 2\lVert \omega\rVert_{L^\infty} \int_{\{|x-y|\leq 3|x-\tilde{x}|\}}\frac{G(|x-y|)}{|x-y|}\dd y =4\pi \lVert \omega\rVert_{L^\infty}\int_0^{3|x-\tilde{x}|}G(\rho)\dd\rho \leq C\lVert \omega\rVert_{L^{\infty}}\int_0^{3|x-\tilde{x}|} m(\rho^{-1})\dd \rho. 
\end{align*}
It follows from L'H\^opital's rule and \eqref{eq:limit-beta1} in \eqref{H2a} that
\begin{align*}
  \lim_{\rho\to 0}\frac{\int_0^\rho m(\tilde\rho^{-1})\dd \tilde\rho}{\rho m(\rho^{-1})} = \lim_{\rho\to 0} \frac{m(\rho^{-1})}{m(\rho^{-1})-\rho^{-1}m'(\rho^{-1})}
  = \lim_{r\to +\infty} \frac{1}{1-\frac{r m'(r)}{m(r)}} = 1.
\end{align*}
Hence, there exists a constant $C>0$, depending on $\bar{c}_0$, such that for any $\rho< c_0$,
\begin{align}\label{eq:continuity-cond}
  \int_0^{\rho} m(\tilde\rho^{-1}) \dd \tilde\rho \leq C\rho m(\rho^{-1}).
\end{align}
Therefore, by \eqref{eq:continuity-cond}, we deduce
\begin{align*}
  I \leq C\|\omega\|_{L^\infty} |x-\tilde{x}|\,m(|x-\tilde{x}|^{-1}).
\end{align*}

For $II$, we apply the mean value theorem and get
\begin{align*}
 K(x-y)-K(\tilde{x}-y) = (x-\tilde{x})\cdot\nabla K(\theta x+(1-\theta)\tilde{x}-y),
\end{align*}
for some $\theta\in[0,1]$. Applying \eqref{eq:Kbound} with $l=1$, we have
\begin{align}\label{eq:gradKbound}
 |\nabla K(x)|\leq C\, \frac{m(|x|^{-1})+1}{|x|^2},
\end{align}
for every $x\neq0$. Moreover, for $|x-y|\geq 2|x-\tilde{x}|$, we have
\begin{align*}
\tfrac{1}{2}|x-y|\leq |\theta x+(1-\theta)\tilde{x}-y|\leq \tfrac{3}{2}|x-y|,
\end{align*}
for any $\theta\in [0,1]$.
Therefore, we apply the above estimates and \eqref{eq:m-prop3} to obtain
\begin{align*}
  II & \leq C |x-\tilde{x}|\int_{\{|x-y|\geq 2 |x-\tilde{x}|\}}\frac{m(|x-y|^{-1})+1}{|x-y|^2}\,|\omega(y)|\dd y\\
  &\leq C\lVert \omega\rVert_{L^{\infty}}|x-\tilde{x}|\int_{\{2|x-\tilde{x}|\leq |x-y|\leq 2\}}
  \frac{m(|x-y|^{-1})+1}{|x-y|^2}\,\dd y
  + C |x-\tilde{x}|\int_{\RR^2}|\omega(y)|\dd y\\
  &\leq C\lVert \omega\rVert_{L^\infty}|x-\tilde{x}| \bigg(\int^{\frac{1}{2|x-\tilde{x}|}}_{\frac{1}{2}}\frac{m(r)+1}{r}\dd r\bigg) + C\lVert \omega\rVert_{L^1}|x-\tilde{x}|\\
  & \leq C\lVert \omega\rVert_{L^1\cap L^{\infty}} |x-\tilde{x}| m(|x-\tilde{x}|^{-1})\log |x-\tilde{x}|^{-1}\leq C\lVert \omega\rVert_{L^1\cap L^{\infty}}\nu(|x-\tilde{x}|),
\end{align*}
where we have used that $m(\cdot)$ is increasing in the penultimate inequality to get
\begin{align*}
m(r)+1\leq m(\tfrac{1}{2|x-\tilde{x}|}) + \frac{m(\tfrac{1}{2|x-\tilde{x}|})}{m(\tfrac12)}\leq \Big(1+\frac1{m(\frac12)}\Big)m(|x-\tilde{x}|^{-1}).
\end{align*}
Gathering the above estimates, we conclude with the desired bound
\begin{align*}
 |u(x)-u(\tilde{x})|\leq C\lVert \omega\rVert_{L^1\cap L^{\infty}}\nu(|x-\tilde{x}|).
\end{align*}

When $|x-\tilde{x}|\geq\frac{\bar{c}_0}{3}$, we have the following direct estimate
\begin{align}\label{eq:largeeps}
\frac{|u(x)-u(\tilde{x})|}{\nu(|x-\tilde{x}|)}\leq \frac{2\|u\|_{L^\infty}}{\inf_{\rho\geq \frac{\bar{c}_0}{3}}\nu(\rho)}\leq C \lVert \omega\rVert_{L^1\cap L^{\infty}},
\end{align}
where we have used the fact that $\nu$ is strictly positive away from zero.
\end{proof}

Now we present the proof of Theorem \ref{thm:Yudovich_type}.
\begin{proof}[Proof of Theorem \ref{thm:Yudovich_type}]
From the transport structure of \eqref{eq:loglogEuler} and $\mathrm{div}\, u=0$, we have
\begin{align*}
  \lVert \omega(t)\rVert_{L^{1}\cap L^{\infty}}=\lVert \omega_0\rVert_{L^{1}\cap L^{\infty}}.
\end{align*}
By virtue of Lemma \ref{lem:u-est-general-data}, the velocity field $u=\nabla^{\perp}(-\Delta)^{-1}m(\Lambda)\omega$ satisfies
\begin{align}\label{eq:u-log-Lip-est}
  \lVert u(t)\rVert_{L^{\infty}(\RR^2)}+\sup_{|x-y|>0}\frac{|u(x,t) - u(y,t)|}{\nu(|x-y|)}
  \leq C\lVert \omega_0\rVert_{L^{1}\cap L^{\infty}(\RR^2)}.
\end{align}
The Osgood condition \eqref{eq:Osgood-cond} translates into
\begin{align*}
  \int_0^{\frac{1}{2}}\frac{\dd\rho}{\nu(\rho)}=\int_0^{\frac{1}{2}}\frac{\dd \rho}{\rho (\log \rho^{-1}) m(\rho^{-1})} = \int_2^{+\infty}\frac{\dd r}{r(\log r)m(r)}=+\infty.
\end{align*}
Then, from the standard Peano's existence theorem and Osgood's uniqueness theorem for ordinary differential equations, the flow map equation \eqref{eq:flow_map_def} admits a unique global solution $\Phi_{t,s}: \mathbb{R}^2\to \mathbb{R}^2$ for every $t,s\in \mathbb{R}_+$.

According to \eqref{eq:flow_map_def} and \eqref{eq:u-log-Lip-est}, it follows that for every $t,s\in \mathbb{R}_+$,
\begin{align}\label{es:flow-map1}
  \Big|\frac{\dd (\Phi_{t,s}(x) - \Phi_{t,s}(y))}{\dd t}\Big|
  \leq \big|u(\Phi_{t,s}(x),t)-u(\Phi_{t,s}(y),t)\big|
  \leq C\lVert \omega_0\rVert_{L^1\cap L^\infty} \,\nu(|\Phi_{t,s}(x)-\Phi_{t,s}(y)|).
\end{align}

The function $\mathsf{H}(r)$ defined in \eqref{def:H-r} is monotonically increasing, and from the Osgood condition \eqref{eq:Osgood-cond},
\begin{align*}
  \lim_{r\to+\infty}\mathsf{H}(r)=\int_2^{+\infty}\frac{\dd r}{r (\log r) m(r)} = +\infty,\quad \text{and}\quad
  \lim_{r\to0}\mathsf{H}(r) = -\infty.
\end{align*}
Hence, $\mathsf{H}(\cdot):(0,+\infty) \mapsto \RR$ is an invertible map.
Moreover, $\mathsf{H}$ satisfies 
\begin{align*}
 \frac{\dd}{\dd \rho}\mathsf{H}(\rho^{-1})= - \frac{1}{\nu(\rho)},\quad \forall\,\rho\in (0, +\infty).
\end{align*}
Then, from \eqref{es:flow-map1} we get
\begin{align*}
  \Big|\mathsf{H}(|\Phi_{t,s}(x) - \Phi_{t,s}(y)|^{-1}) - \mathsf{H}(|x-y|^{-1}) \Big| =\Big|\int_{|x-y|}^{|\Phi_{t,s}(x) - \Phi_{t,s}(y)|}\frac{\dd \rho}{\nu(\rho)}\Big|
  \leq C\lVert \omega_0\rVert_{L^1\cap L^{\infty}} |t-s|,
\end{align*}
and hence
\begin{align*}
  \big|\Phi_{t,s}(x)-\Phi_{t,s}(y)\big|^{-1} & \geq \mathsf{H}^{-1}\Big(\mathsf{H}(|x-y|^{-1})-C \|\omega_0\rVert_{L^1\cap L^\infty}|t-s|\Big),\\
  \big|\Phi_{t,s}(x)-\Phi_{t,s}(y\big)|^{-1} & \leq \mathsf{H}^{-1}\Big(\mathsf{H}(|x-y|^{-1})+C\|\omega_0\rVert_{L^1\cap L^{\infty}}|t-s|\Big),
\end{align*}
which directly implies \eqref{eq:bound-flow-map-general}.
Similarly, letting $0\leq t_1\leq t_2 <+\infty$, we have $\Phi_{0,t_2}(x) = \Phi_{0,t_1}\circ \Phi_{t_1,t_2}(x)$
and
\begin{align*}
  |\Phi_{t_1}^{-1}(x) - \Phi_{t_2}^{-1}(x)|^{-1} & = |\Phi_{0,t_1}(x) - \Phi_{0,t_1}\circ \Phi_{t_1,t_2}(x)|^{-1} \\
  & \geq \mathsf{H}^{-1}\Big(\mathsf{H}(|x - \Phi_{t_1,t_2}(x)|^{-1}) - C \|\omega_0\|_{L^1\cap L^\infty} t_1\Big) \\
  & \geq \mathsf{H}^{-1} \Big(\mathsf{H}(C^{-1}\|\omega_0\|_{L^1\cap L^\infty}^{-1} |t_1-t_2|^{-1}) 
  - C \|\omega_0\|_{L^1\cap L^\infty} t_1 \Big),
\end{align*}
where in the last line we have used the equation \eqref{eq:flow_map_def} and estimate \eqref{eq:u-log-Lip-est}.
With the above \textit{a priori} estimates, and following the argument used in the proof of Yudovich's theorem for the 2D Euler equation (see, for instance, \cite[Sec.~8.2]{MB02}), one can establish the global existence result.

Next, by adapting the ideas in \cite[Sec.~2.3]{MP94} with suitable modifications (see also \cite{Vishik99} and \cite[Chap.~7]{BCD11}, where the Littlewood-Paley theory is used), we provide a proof of the uniqueness part of Theorem \ref{thm:Yudovich_type}.
Assume that $(\omega^1, u^1)$ and $(\omega^2, u^2)$ are two weak solutions on $[0,T]$ 
of the 2D Loglog-Euler type equation \eqref{eq:loglogEuler} associated with the same initial data 
$\omega_0\in L^1\cap L^\infty(\mathbb{R}^2)$. 
Let $\Phi_t^i(\cdot)= \Phi^i_{t,0}(\cdot)$ ($i=1,2$) denote the flow map generated by the velocity field $u^i$, satisfying \eqref{eq:flow_map_def}. 
Without loss of generality, we assume $\omega_0 \not\equiv 0$. 
Set 
\begin{align}\label{eq:def-diff-flow-map}
  \delta(t)\triangleq \frac{1}{\lVert \omega_0\rVert_{L^1}}\int_{\RR^2}|\Phi^1_t(x)-\Phi^2_t(x)| \cdot|\omega_0(x)|\dd x,
\end{align}
By \eqref{eq:flow_map_def}, it follows that 
\begin{align*}
  \Phi^1_t(x)-\Phi^2_t(x) & = \int_0^t \Big(u^1(\Phi^1_\tau(x),\tau)-u^2(\Phi^2_\tau(x),\tau)\Big) \dd \tau \\
  & = \int_0^t \Big(u^1(\Phi^1_\tau(x),\tau)-u^1(\Phi^2_\tau(x),\tau) \Big)\dd \tau 
  + \int_0^t \Big(u^1(\Phi^2_\tau(x),\tau)-u^2(\Phi^2_\tau(x),\tau)\Big) \dd \tau. 
\end{align*}
For the first term, in view of \eqref{eq:u-log-Lip-est} and \eqref{def:nu}, 
there exists some constant $C>0$ depending only on $\|\omega_0\|_{L^1\cap L^\infty}$ such that 
\begin{align*}
  |u^1(\Phi^1_t(x),t)-u^1(\Phi_t^2(x),t)|\leq C\nu(|\Phi^1_t(x)-\Phi_t^2(x)|).
\end{align*}
For the second term, by the change of variables and using the properties of flow map $\Phi_t^i$, for all $t\in [0,T]$ with $T>0$, 
we have
\begin{align*}
  & \int_{\RR^2}|u^1(\Phi_t^2(x),t)-u^2(\Phi_t^2(x),t)|\cdot |\omega_0(x)|\dd x\\
  & = \int_{\mathbb{R}^2} \Big| \mathrm{p.v.} \int_{\mathbb{R}^2} K(\Phi_t^2(x)-z) \big( \omega^1(z,t) - \omega^2(z,t) \big) \dd z \Big| \cdot |\omega_0(x)| \dd x \\
  & = \int_{\RR^2}\Big|\int_{\RR^2}\big(K(\Phi_t^2(x)-\Phi^{1}_{t}(y))-K(\Phi_t^2(x)-\Phi^{2}_{t}(y))\big)\omega_0(y)\dd y\Big|\cdot|\omega_0(x)|\dd x \\
  & \leq \int_{\RR^2}|\omega_0(y)|\left(\int_{\RR^2}\big|K(\Phi_t^2(x)-\Phi^{1}_{t}(y))-K(\Phi_t^2(x)-\Phi^{2}_{t}(y))\big|\cdot|\omega_0(x)|\dd x\right)\dd y\\
  &=\int_{\RR^2}|\omega_0(y)|\left(\int_{\RR^2}\big|K(x-\Phi^{1}_{t}(y))-K(x-\Phi^{2}_{t}(y))\big|\cdot|\omega^2(x,t)|\dd x\right)\dd y\\
  & \leq C\lVert \omega_0\rVert_{L^1\cap L^\infty}
  \int_{\RR^2}\nu(|\Phi^{1}_{t}(y)-\Phi^{2}_{t}(y)|)\cdot|\omega_0(y)|\dd y,
\end{align*}
where the last inequality follows from the same procedure used to estimate $I$ and $II$ in the proof of Lemma~\ref{lem:u-est-general-data}.
Combining the above estimates, we infer that 
\begin{align*}
\delta(t) \leq C\int_0^{t}\int_{\RR^2} 
  \nu(|\Phi^1_\tau(x)-\Phi^2_\tau(x)|)\cdot|\omega_0(x)|\dd x\dd \tau.
\end{align*}

If we assume $\nu$ to be concave, then by Jensen's inequality, we obtain
\begin{align*}
\int_{\RR^2} 
  \nu(|\Phi^1_t(x)-\Phi^2_t(x)|)\cdot\frac{|\omega_0(x)|}{\lVert \omega_0\rVert_{L^1}}\dd x\leq \nu\Big(\int_{\RR^2} 
  |\Phi^1_t(x)-\Phi^2_t(x)|\cdot\frac{|\omega_0(x)|}{\lVert \omega_0\rVert_{L^1}}\dd x\Big)=\nu(\delta(t)),
\end{align*}
which yields
\begin{align}\label{eq:deltabound}
 \delta(t) \leq C\int_0^{t}\nu(\delta(\tau))\dd \tau.
\end{align}
Even if $\nu$ is not globally concave, inequality \eqref{eq:deltabound} still holds. To see this, we observe that $\nu$ is concave near the origin. Indeed,
\begin{align*}
\nu''(0^+) & =\lim_{\rho\to0^+}\Big(-\rho^{-1}m(\rho^{-1})+2\rho^{-2}m'(\rho^{-1})+\rho^{-3}\log\rho^{-1}m''(\rho^{-1})\Big)\\
& = \lim_{r\to+\infty} rm(r)\Big(-1 + \frac{rm'(r)}{m(r)}\cdot\frac{r\log r\,m''(r)+2m'(r)}{m'(r)}\Big)=-\infty,
\end{align*}
where we have applied \eqref{eq:limit-beta1}-\eqref{eq:limit-beta2} in \eqref{H2a}. We then define 
\begin{align*}
  \nu_1(\rho) \triangleq
  \begin{cases}
	\nu(\rho) = \rho (\log \rho^{-1}) m(\rho^{-1}),\qquad &\text{if } \rho\in (0,c_1],\\
	c_1 (\log c_1^{-1}) m(c_1^{-1})+\widetilde{\nu}'(c_1)(\rho-c_1),\qquad &\text{if } \rho\in (c_1,+\infty),
  \end{cases}
\end{align*}
where the constant $c_1\in(0,\frac12)$ is chosen sufficiently small so that $\nu_1$ is concave. Moreover, 
\begin{align*}
 \nu_1(\rho)\approx_{c_1} \nu(\rho),\quad \forall \rho\in(0,+\infty).	
\end{align*}
Repeating the above estimates with $\widetilde{\nu}$ in place of $\nu$, we obtain
\begin{align*}
 \delta(t) \leq C\int_0^{t}\int_{\RR^2}\nu_1(|\Phi^1_\tau(x)-\Phi^2_\tau(x)|)\cdot|\omega_0(x)|\dd x\dd \tau\leq C\int_0^t\nu_1(\delta(\tau))\dd\tau\leq C\int_0^{t}\nu(\delta(\tau))\dd \tau.
\end{align*}

Using the Osgood condition \eqref{eq:Osgood-cond}, we deduce from \eqref{eq:deltabound} that $\delta(t)\equiv 0$ for all $t\in [0,T]$, which implies $\Phi_t^{1}(x)=\Phi_t^{2}(x)$ for all $x\in$ supp $\omega_0$ and all $t\in [0,T]$. It then yields
\begin{align*}
  \omega^1(\Phi_t^{1}(x),t)=\omega_0(x)=\omega^2(\Phi_t^{2}(x),t), \quad \forall x\in \text{supp }\omega_0.
\end{align*}
Moreover, since $\Phi_t^1(\text{supp }\omega_0)=\Phi_t^2(\text{supp }\omega_0)$, we have
\begin{align*}
  \omega^1(\Phi_t^{1}(x),t)=0=\omega^2(\Phi_t^{2}(x),t), \quad \forall x\not\in \text{supp }\omega_0.
\end{align*}
Therefore, we conclude that the uniqueness result holds, namely $\omega^1(\cdot,t)\equiv\omega^2(\cdot,t)$ on $\RR^2$.
\end{proof}

\section{\texorpdfstring{Global regularity of $C^{1,\mu}$ single vortex patch}{Global regularity of C(1,mu) vortex patch}}
\label{sec:C1mu-single}
In this section, we show the global regularity of the 2D Loglog-Euler type equation \eqref{m-SQG} 
with the $C^{1,\mu}$ single patch initial data. 

\subsection{Mathematical formulation for vortex patch problem} \label{subsec:formulation}
Let $D_0\subset \mathbb{R}^2$ be a simply connected bounded domain with boundary 
$\partial D_0\in C^{n,\mu}$, $n\in \mathbb{N}^\star$, $\mu\in (0,1)$, and 
$\omega_0(x)=\mathbf{1}_{D_0}(x)\in L^1\cap L^\infty$ 
be the initial data of the 2D Loglog-Euler type equation \eqref{m-SQG}. 
According to Theorem \ref{thm:Yudovich_type},
there exists a unique global-in-time weak solution
\begin{align*}
  \textrm{$\omega(x,t)=\mathbf{1}_{D(t)}(x)$\;\; and \;\;$D(t) = \Phi_t(D_0)$,}
\end{align*}
where the flow map $\Phi_t(\cdot)=\Phi_{t,0}(\cdot)$ is given by \eqref{eq:flow_map_def}.
The vortex patch problem is concerned with the global regularity of $\partial D(t)$ for any $t>0$.

Let $z_0(\xi):\mathbb{T}\mapsto \partial D_0$ be a $C^{n,\mu}$-parameterization of the boundary $\partial D_0$, where $\mathbb{T}$ denotes the one-dimensional periodic domain.
Then 
$$z(\xi,t)=\Phi_t(z_0(\xi))$$
is a parameterization of $\partial D(t)$ 
and satisfies the following \textit{contour dynamics equation}
\begin{equation}\label{eq:contour_eq}
  \frac{\dd z(\xi,t)}{\dd t}= u(z(\xi,t),t),\quad  z(\xi, 0) = z_0(\xi).
\end{equation}
where
\begin{align}\label{eq:velo-Lag}
  u(z(\xi,t),t) & = \int_{D(t)} \nabla^\perp_y (\widetilde{R}(|z(\xi,t)-y|)) \dd y 
  = \int_{\partial D(t)} \widetilde{R}\big(|z(\xi,t) - y| \big) \mathbf{n}^\perp(y,t) \dd S(y) \nonumber \\
  & = \int_{\mathbb{T}} \widetilde{R}\big(|z(\xi,t) - z(\eta,t)|\big) \partial_\eta z(\eta,t) \dd \eta,
\end{align}
with $\widetilde{R}(\rho) = \int_\rho^1 \frac{G(r)}{r} \dd r + C$ and $C\in \mathbb{R}$ some constant chosen for convenience.
This formulation in Lagrangian coordinates works well in obtaining the local-in-time regularity of
$\partial D(t)$. However, even for the 2D Euler equation, it is still open how to get the global regularity of $\partial D(t)$
directly from \eqref{eq:contour_eq}-\eqref{eq:velo-Lag}.

Following \cite{BC93,Chemin93}, we introduce the level-set formulation for the vortex patch problem,
which is based on Eulerian coordinates. 
\begin{definition}[Level-set characterization of a domain]\label{def:levelset}
 Let $D$ be a simply connected and bounded domain $\partial D\in C^{k,\mu}$. Let $\varphi\in C^{k,\mu}(\RR^2)$ with $k\in \mathbb{N}^\star$ and $\mu\in (0,1)$. We say $\varphi$ is a level-set characterization of the domain $D$ if
 \begin{align}\label{eq:levelsetrep}
  D=\{x\in \mathbb{R}^2\, | \, \varphi(x)>0\},\quad \inf_{x\in \partial D} |\nabla \varphi(x)| \geq c > 0. 
\end{align}
\end{definition}
A level-set characterization $\varphi$ can be constructed by solving the elliptic problem $-\Delta \varphi = 1$ in $D$ with the Dirichlet boundary condition $\varphi|_{\partial D} = 0$, and then extending $\varphi$ to $\mathbb{R}^2$. 
One can see, e.g. \cite{KRYZ16} for more details of the construction. 
As a direct consequence of \eqref{eq:levelsetrep}, we get
\begin{align}\label{cond:varphi-outer-domain}
  \varphi = 0 \quad \textrm{on}\;\; \partial D,\qquad \textrm{and}\qquad
  \varphi <0\quad \textrm{in}\;\; U\cap \overline{D}^c,
\end{align}
where $U\subset \mathbb{R}^2$ is a small open neighborhood of $\partial D$.

Let $\varphi_0$ be a $C^{n,\mu}$ level-set characterization of the domain $D_0$. For any $x_0\in \partial D_0$,
the solution of the following ordinary differential equation
\begin{align}\label{eq:parameter_varphi}
  \frac{\dd z_0(\xi)}{\dd \xi} = \nabla^\perp \varphi_0(z_0(\xi)) \triangleq W_0(z_0(\xi)), \quad
  z_0(0)=x_0,
\end{align}
yields a $C^{n,\mu}$-parameterization of the curve $\partial D_0$.
Let $z(\xi,t)$ be a solution of \eqref{eq:contour_eq} with $z_0$ given by \eqref{eq:parameter_varphi}
and $\varphi(x,t)=\varphi_0(\Phi_t^{-1}(x))$. Then $\varphi(x,t)$ is the solution of
\begin{align}\label{eq:level-set-eq}
  \partial_t\varphi+u\cdot \nabla \varphi=0,
  \quad u=\nabla^{\perp}(-\Delta)^{-1} m(\Lambda) (\mathbf{1}_{D(t)}),\quad
  \varphi|_{t=0}=\varphi_0,
\end{align}
with
$D(t)=\{x\in \mathbb{R}^2\, : \, \varphi(x,t)>0\}$.
Furthermore, it follows from direct computation (e.g. see \cite[Lemma 1.4]{MB02}) that
\begin{align}\label{eq:tangential_vector_contour}
  \partial_{\xi}z(\xi,t)= W (z(\xi,t),t),\quad \textrm{with}\quad W(\cdot,t) \triangleq \nabla^\perp \varphi (\cdot, t).
\end{align}
Using \eqref{eq:tangential_vector_contour} repeatedly, we find that for $k\in \mathbb{N}^\star$,
\begin{align}\label{eq:striated_regularity}
  \partial_{\xi}^{k}z(\xi,t)=\big((W \cdot\nabla)^{k-1}W\big)(z(\xi,t),t) = \partial_W^{k-1}W (z(\xi,t),t),
\end{align}
where $\partial_W = W\cdot\nabla$.
Therefore, to study the boundary regularity of $\partial D(t)$,
it is sufficient to explore the regularity of $\partial_W^{k-1}W(\cdot,t)$.
Indeed, from \eqref{eq:contour_eq}, Lemma \ref{lem:u-est-general-data} and \eqref{eq:striated_regularity}, we have
\begin{align}\label{eq:z0bound}
	\lVert z\rVert_{L^{\infty}}\leq \lVert z_0\rVert_{L^{\infty}}+\int_0^{t}\lVert u(\cdot,s)\rVert_{L^{\infty}}\dd s\leq \lVert z_0\rVert_{L^{\infty}}+C\lVert \omega_0\rVert_{L^1\cap L^{\infty}}t, 
\end{align}
and for $k\geq 1$, 
\begin{align}\label{eq:zkbound}
	\|\partial_{\xi}^{k}z(\cdot,t)\|_{L^{\infty}} \leq \|\partial_W^{k-1}W(\cdot,t)\|_{L^{\infty}},
\end{align}
and
\begin{align}\label{eq:zHolder}
	\|\partial_{\xi}^{k}z(\cdot,t)\|_{\dot{C}^\mu} \leq \|\partial_W^{k-1}W(\cdot,t)\|_{\dot{C}^\mu}\|\partial_\xi z(\cdot,t)\|_{L^\infty}^\mu = \|\partial_W^{k-1}W(\cdot,t)\|_{\dot{C}^\mu}\|W(\cdot,t)\|_{L^\infty}^\mu. 
\end{align}

According to \eqref{eq:level-set-eq} and the fact that $[\partial_W,\partial_t+u\cdot \nabla]=0$, we infer that for any $k\in \mathbb{N}^\star$,
\begin{align}\label{eq:parWk-1W}
  (\partial_t+u\cdot \nabla)\big(\partial_W^{k-1}W\big)  = \partial_W^k u,\quad \partial_W^{k-1} W|_{t=0} =\partial_{W_0}^{k-1} W_0.
\end{align}
Hence, the vortex patch problem can be studied through \eqref{eq:level-set-eq} 
and \eqref{eq:parWk-1W}.
 
Before embarking on our proof, we introduce the following notation and function spaces:
\begin{align*}
  |\nabla \varphi(\cdot)|_{\inf} & \triangleq \inf_{x\in \partial D}|\nabla \varphi(x)|,	\\
  \|\nabla \varphi(\cdot)\|_{\dot{C}^\mu(\RR^2)}
  &\triangleq\sup_{x\ne \tilde{x}}\frac{|\nabla \varphi(x)-\nabla \varphi(\tilde{x})|}{|x-\tilde{x}|^{\mu}},
  \quad \mu\in (0,1),\\
  \|\nabla \varphi(\cdot)\|_{C^\mu(\RR^2)}
  &\triangleq\|\nabla \varphi(\cdot)\|_{L^\infty(\RR^2)}+\|\nabla \varphi(\cdot)\|_{\dot{C}^\mu(\RR^2)},
  \quad \mu\in (0,1),
\end{align*}
and
\begin{align}\label{def:Delta-gamma}
  \mathbf{\Delta}_\mu \triangleq \frac{\|\nabla \varphi\|_{\dot{C}^\mu(\RR^2)}}{|\nabla \varphi|_{\inf}} + 1.
\end{align}
The $m$-adapted H\"older space $C^\mu_{\mathrm{m}}(\mathbb{R}^2)$ ($\mu\in (0,1]$) is composed of functions $f$ such that
\begin{align}\label{eq:Csig-m}
  \lVert f\rVert_{C^\mu_{\mathrm{m}}} = \lVert f\rVert_{C^\mu_{\mathrm{m}}(\RR^2)} \triangleq \lVert f\rVert_{L^\infty(\RR^2)}
  + \sup_{|x- y|>0} \frac{|f(x)-f(y)|}{|x-y|^\mu \big(m(|x-y|^{-1})+1\big)}.
\end{align}

\subsection{\texorpdfstring{Near-Lipschitz estimate of $u$}{Near-Lipschitz estimate of u}}

The following lemma shows that the continuity of
$u$ is stable under a small perturbation produced by the slowly increasing multiplier $m$. 
\begin{lemma}\label{lem:continuity-u}
Suppose that $m(r)$ satisfies the assumptions \eqref{H1}-\eqref{H2a} and $D\subseteq \mathbb{R}^2$
is a bounded domain with $C^{1,\mu}$ $(0<\mu<1)$-boundary. 
Let $\varphi$ be the level-set characterization of domain $D$ where \eqref{eq:levelsetrep}-\eqref{cond:varphi-outer-domain} hold.
Then for any $\gamma\in (0,\mu]$, the velocity $u=\nabla^\perp (-\Delta)^{-1}m(\Lambda)(\mathbf{1}_D)$ satisfies
\begin{align}\label{eq:u-patch-imp-est}
  \lVert u\rVert_{C^1_{\mathrm{m}}(\mathbb{R}^2)}
  \leq C \big(1 + \log \mathbf{\Delta}_\gamma \big),
\end{align}
where the constant $C>0$ depends only on $\gamma$ and the domain area $|D|$.
\end{lemma}

Compared with Lemma~\ref{lem:u-est-general-data}, the core idea of Lemma~\ref{lem:continuity-u} is that the cancellation structure in a regular patch enhances the regularity of the velocity field by logarithmic order.

We emphasize that Lemma~\ref{lem:continuity-u} is formulated in the physical space. It serves as an alternative representation of the following estimate, which is expressed in the frequency space.
\begin{lemma}[{\cite[p.~985]{Elgindi14}}]\label{lem:nabla-u}
Let $S_j=\chi(2^{-j} \Lambda)$, $j\in \mathbb{N}$ be the low-frequency cut-off operator in the Littlewood-Paley theory.
Under the assumptions of Lemma \ref{lem:continuity-u}, and for every $\gamma\in (0,\mu]$, the following estimate holds:
\begin{align*}
  \lVert S_j\nabla u\rVert_{L^{\infty}(\RR^2)}\leq C m(2^j) \big(1 + \log \mathbf{\Delta}_\gamma \big),
\end{align*}
where
$C>0$ depends only on $\gamma$ and $|D|$ (the area of domain $D$).
\end{lemma}
The equivalence between Lemma~\ref{lem:continuity-u} and Lemma~\ref{lem:nabla-u} can be guaranteed by \cite[Sec. 2.11]{BCD11}.

\begin{proof}[Proof of Lemma \ref{lem:continuity-u}]
Taking advantage of Lemma \ref{lem:u-est-general-data}, we directly get $\lVert u\rVert_{L^{\infty}}\leq C$.
It remains to control the modulus-of-continuity estimate of $u$ in \eqref{eq:Csig-m}.
	
From the representation \eqref{eq:u-exp1}, for any $x,h\in \RR^2$ such that $|h|\leq \frac{\bar{c}_0}{3}$ 
(recalling that $\bar{c}_0$ is the constant appearing in Lemma \ref{lem:G-prop}), we write
\begin{align*}
  u(x)-u(x+h ) & = \int_{D\cap\{|x-y|\leq 2|h|\}}K(x-y)\dd y
  - \int_{D\cap\{|x-y|\leq 2|h|\}}K(x+h-y)\dd y \\
  & \quad + \int_{D\cap\{|x-y|\geq 2|h|\}} \big( K(x-y)-K(x+h-y) \big) \dd y \\
  & \triangleq \mathfrak{N}_1 + \mathfrak{N}_2 + \mathfrak{N}_3.
\end{align*}

For $\mathfrak{N}_1$, arguing as the estimation for $I$ in the proof of Lemma~\ref{lem:u-est-general-data}, we obtain
\begin{align*}
  |\mathfrak{N}_1|\leq C \int_0^{2|h|}G(\rho)\dd \rho\leq C\int_0^{2|h|}m(\rho^{-1})\dd \rho\leq C|h|m(|h|^{-1}).
\end{align*}
Similarly, we bound $\mathfrak{N}_2$ by
\begin{align*}
  |\mathfrak{N}_2|\leq C \int_0^{3|h|}|G(\rho)|\dd \rho\leq C|h|m(|h|^{-1}),
\end{align*}
utilizing the fact $|x+h-y|\leq |x-y|+|h|$.

For $\mathfrak{N}_3$, by mean value theorem, it follows that
\begin{align*}
  |\mathfrak{N}_3| & =\Big|\int_0^1\int_{D\cap\{|x-y|\geq 2|h|\}} h\cdot \nabla K(x+\theta h-y)\dd y \dd \theta\Big|\\
  &\leq \int_0^1\Big|\int_{D\cap\{|x+\theta h-y|\geq 2|h|\}}h\cdot \nabla K(x+\theta h-y)\dd y \Big|\dd \theta\\
  & \quad +\int_0^1\int_{\{|x-y|\geq 2|h|\}\triangle\{|x+\theta h-y|\geq 2|h|\}}|h\cdot \nabla K(x+\theta h-y)|\dd y\dd \theta\\
  & \triangleq \mathfrak{N}_{31}+\mathfrak{N}_{32},
\end{align*}
where the notation $A\triangle B \triangleq (A\backslash B) \cup (B\backslash A)$.
For $\mathfrak{N}_{32}$, observe that
\begin{align*}
\{|x-y|\geq 2|h|\}\triangle\{|x+\theta h-y|\geq 2|h|\}\subset \{|h|\leq |x+\theta h-y|\leq 3|h|\}.	
\end{align*}
In view of \eqref{eq:gradKbound}, we infer that
\begin{align*}
  \mathfrak{N}_{32} \leq& |h|\int_0^1\int_{\{|h|\leq |x+\theta h-y|\leq 3|h|\}} |\nabla K(x+\theta h-y)|\dd y \dd \theta\\
  \leq& C\int_0^1\int_{|h|}^{3|h|}\frac{m(\rho^{-1})}{\rho^2}\rho\dd\rho\dd \theta\leq C|h|m(|h|^{-1}).
\end{align*}
Note that in the second inequality, we drop the second term in \eqref{eq:gradKbound}, since $|x+\theta h-y|\leq 3|h|\leq \bar{c}_0$.

For the remaining term $\mathfrak{N}_{31}$, we state the following claim:
\begin{align}\label{claim:refine_ineq}
  \Big|\int_{D\cap\{|x-y|\geq 2|h|\}}\nabla K(x-y)\dd y\Big|
  \leq C \big(1+\log\mathbf{\Delta}_\gamma \big) \big(m(|h|^{-1})+1\big),\quad\forall \gamma\in (0,1),
\end{align}
where $C>0$ depends only on $\gamma$. Applying \eqref{claim:refine_ineq} and replacing $x$ by
$x+\theta h$, we obtain
\begin{align*}
  \mathfrak{N}_{31}\leq C|h|\cdot\big(1+\log\mathbf{\Delta}_\gamma \big) \big(m(|h|^{-1})+1\big).
\end{align*}

Collecting all the estimates, we obtain the desired bound \eqref{eq:u-patch-imp-est}.
For the case $|h| > \frac{\bar{c}_0}{3}$, we argue analogously to \eqref{eq:largeeps} and derive the same bound.
\medskip

We are left to show the claim \eqref{claim:refine_ineq}. Applying \eqref{eq:gradKbound} and the fact that $m(r)$ is non-decreasing, we obtain
\begin{align}\label{eq:far-away-part}
  & \Big|\int_{D\cap\{|x-y|\geq 2|h|\}}\nabla K(x-y)\dd y\Big|
  \leq C\int_{D\cap\{|x-y|\geq 2|h|\}}\frac{m(|x-y|^{-1})+1}{|x-y|^2}\,\dd y \nonumber\\
  & \leq Cm(|h|^{-1})\int_{D\cap\{|x-y|\geq 2|h|\}}\frac{\dd y}{|x-y|^2} + Cm((\tfrac{\bar{c}_0}{3})^{-1})\int_{D\cap\{|x-y|\geq 2|h|\}} \frac{\dd y}{|x-y|^2}\nonumber \\
  & \leq C m(|h|^{-1}) \int_{2|h|}^L \frac{1}{\rho^2}\, \rho\dd\rho 
  \leq C m(|h|^{-1}) \big(1+\log |h|^{-1}\big),
\end{align}
where $L\triangleq \sqrt{|D|/\pi}$ and in the third line we have used the rearrangement inequality.

If $2|h|\geq \delta_\gamma$, with
\begin{align*}
  \delta_\gamma \triangleq\mathbf{\Delta}_{\gamma}^{- \frac{1}{\gamma}}
  = \Big(\tfrac{\|\nabla \varphi\|_{\dot{C}^\gamma}}{|\nabla \varphi|_{\inf}} + 1 \Big)^{-\frac{1}{\gamma}},
\end{align*}
then \eqref{claim:refine_ineq} is a direct consequence of \eqref{eq:far-away-part}.

If $2|h|<\delta_\gamma$, we exploit the symmetry properties of $K$ and the patch structure to establish cancellations, analogous to \cite[Geometric lemma]{BC93}. Denote $d_x$ is the distance between $x$ and $\partial D$, namely 
\begin{align}\label{eq:dx}
d_x \triangleq \inf_{z\in \partial D} |x-z|.
\end{align}
For every $x\in \mathbb{R}^2$, choosing $\tilde{x}\in \partial D$ such that $d_x =|x - \tilde{x}|$,
we define the sets
\begin{align}\label{def:SSigmaR}
  \mathcal{S}_\rho(x) \triangleq \{z||z|=1, x+\rho z\in D\},\;\;
  \varSigma(x)\triangleq \{z||z|=1,\,
  \nabla_x\varphi(\tilde{x})\cdot z\geq 0\},\;\; \mathcal{R}_\rho(x)\triangleq \mathcal{S}_\rho(x) \triangle \varSigma(x).
\end{align}
Note that
\begin{align*}
 D\cap\{|x-y|\geq 2h\} = \{y=x+\rho z: \rho\geq 2h, \, z\in \mathcal{S}_\rho(x)\},	
\end{align*}
and $\Sigma(x)$ is a semicircle where symmetry can be employed. The geometric lemma in \cite{BC93} characterizes their difference $\mathcal{R}_\rho(x)$:
for all $\rho\geq d_x$, $\gamma\in (0,1)$ and for each $x$ such that $d_x\leq \delta_{\gamma}=\mathbf{\Delta}_{\gamma}^{-1/\gamma}$,
\begin{align}\label{lem:Geometric-lemma}
  \mathcal{H}^1\big(\mathcal{R}_\rho(x)\big)
  \leq 2\pi \bigg((1+2^\gamma) \frac{d_{x}}{\rho}+2^\gamma \Big(\frac{\rho}{\delta_{\gamma}}\Big)^\gamma\bigg),
\end{align}
where $\mathcal{H}^1$ is the Hausdorff measure on the unit circle.

From \eqref{def:widetid-K}, we decompose $\nabla K$ into the symmetric part $(\nabla K)_{\mathrm{S}}$ and the antisymmetric part $(\nabla K)_{\mathrm{A}}$.
For the symmetric part, we observe that $\sigma(z)$ has zero mean in unit circle or semicircle:
\begin{align*}
\rho<d_x:\quad &\int_{\mathcal{S}_\rho(x)} \sigma(z) \dd z=0,\\
\rho\geq d_x:\quad &\Big|\int_{\mathcal{S}_\rho(x)} \sigma(z) \dd z\Big|\leq \Big|\int_{\varSigma(x)} \sigma(z) \dd z\Big|+\int_{\mathcal{R}_\rho(x)} |\sigma(z)| \dd z \leq C\mathcal{H}^1(\mathcal{R}_\rho(x)).
\end{align*}
Utilizing the above cancellations, and applying \eqref{eq:gradKbound} and \eqref{lem:Geometric-lemma}, we get
\begin{align}\label{eq:crucial-est}
&  \Big|\int_{D\cap\{|x-y|\geq 2|h|\}}(\nabla K)_{\mathrm{S}}(x-y) \dd y\Big|
= \Big|\int_{2|h|}^{\delta_\gamma}\frac{1}{2} \frac{2G(\rho)-\rho G'(\rho)}{\rho^2}\int_{\mathcal{S}_\rho(x)} \sigma(z) \dd z\, \rho\dd\rho	\Big|\nonumber\\
& \leq C\int_{\max\{2|h|,d_x\}}^{\delta_{\gamma}}\frac{m(\rho^{-1})+1}{\rho}\,\mathcal{H}^1(\mathcal{R}_{\rho}(x))\dd \rho 
  \leq C\big(m(|h|^{-1})+1\big)\int_{d_x}^{\delta_{\gamma}}\frac{1}{\rho}\Big(\frac{d_x}{\rho}+\big(\frac{\rho}{\delta_\gamma}\big)^\gamma\Big)\dd \rho \nonumber \\
  & \leq C\big(m(|h|^{-1})+1\big).
\end{align}

For the antisymmetric part, we have
\begin{align*}
\int_{D\cap\{|x-y|\geq 2|h|\}}(\nabla K)_{\mathrm{A}}(x-y) \dd y
=\int_{2|h|}^{\delta_\gamma}\frac{G'(\rho)}{2\rho} \mathcal{H}^1(\mathcal{S}_\rho(x))\, \rho\dd\rho   \begin{pmatrix}
	0 &	1\\
	-1 & 0
  \end{pmatrix}.
\end{align*}
For $\mathcal{H}^1(\mathcal{S}_\rho(x))$, we have
\begin{align*}
\rho<d_x:\quad &\mathcal{H}^1(\mathcal{S}_\rho(x)) = 2\pi,\\
\rho\geq d_x:\quad &\mathcal{H}^1(\mathcal{S}_\rho(x))\leq \mathcal{H}^1(\varSigma(x))+\mathcal{H}^1(\mathcal{R}_\rho(x)) = \pi +\mathcal{H}^1(\mathcal{R}_\rho(x)).
\end{align*}
Therefore, by using \eqref{eq:G-prop1}, when $2|h|\geq d_x$, we deduce that
\begin{align*}
  &\Big|\int_{2|h|}^{\delta_\gamma}\frac{G'(\rho)}{2\rho} \mathcal{H}^1(\mathcal{S}_\rho(x))\, \rho\dd\rho\Big|
   \leq \frac\pi2\Big|\int_{2|h|}^{\delta_{\gamma}}G'(\rho) \dd \rho\Big|
  + \frac12\int_{2|h|}^{\delta_{\gamma}}|G'(\rho)| \mathcal{H}^1(\mathcal{R}_{\rho}(x)) \dd \rho\\
  & \leq \frac\pi2 |G(2|h|) - G(\delta_{\gamma})|
  + C\int_{2|h|}^{\delta_{\gamma}}\frac{m(\rho^{-1})+1}{\rho}\,\mathcal{H}^1(\mathcal{R}_{\rho}(x))\dd \rho  \leq C\big(m(|h|^{-1})+1\big),
\end{align*}
and when $2|h|<d_x$, we get
\begin{align*}
  &\Big|\int_{2|h|}^{\delta_\gamma}\frac{G'(\rho)}{2\rho} \mathcal{H}^1(\mathcal{S}_\rho(x))\, \rho\dd\rho\Big|
   \leq \Big|\pi\int_{2|h|}^{d_x}G'(\rho) \dd \rho+\frac\pi2\int_{d_x}^{\delta_{\gamma}}G'(\rho) \dd \rho\Big|
  + \frac12\int_{d_x}^{\delta_{\gamma}}|G'(\rho)| \mathcal{H}^1(\mathcal{R}_{\rho}(x)) \dd \rho\\
  & \leq \frac\pi2 |2G(2|h|) - G(d_x) - G(\delta_{\gamma})|
  + C\big(m(d_x^{-1})+1\big) \leq C\big(m(|h|^{-1})+1\big).
\end{align*}
This concludes the proof of \eqref{claim:refine_ineq}.
\end{proof}

\subsection{\texorpdfstring{The bound of $\partial_W u$}{The bound of partial W u}}

Inspired by \cite[Proposition 2]{BC93}, we first deduce an alternative integral expression for $\partial_W u$, 
which fully exploits the structure of a regular patch. 
\begin{lemma}\label{lem:tangential_velocity}
Let $D\subseteq \mathbb{R}^2$ be a bounded domain with $C^{1,\mu}$ ($0<\mu<1$) boundary and let
$\varphi$ be the level-set characterization of domain $D$. Let $W=\nabla^{\perp}\varphi\in C^{\mu}(\RR^2)$
be a vector field tangent to $\partial D$.  Assume that
$u=\nabla^{\perp}(-\Delta)^{-1}m(\Lambda)(\mathbf{1}_{D})$ and $m(r)$ satisfies the assumptions \eqref{H1}-\eqref{H2a}, 
then the following identity holds true for all $x\in \RR^2$, 
\begin{align}\label{eq:desingularization_tangential_vector}
  \partial_W u(x) = W \cdot \nabla u(x) = \mathrm{p.v.}\int_D \big(W(x) - W(y)\big)\cdot\nabla K(x-y)\dd y,
\end{align}
where $\nabla K(x)$ is defined in \eqref{def:widetid-K}.
\end{lemma}

\begin{proof}
Since $W=\nabla^{\perp}\varphi$ is divergence free, we have
\begin{align*}
  \partial_W u_i = \textrm{div }(W\, u_i),\quad i=1,2.
\end{align*}
In the following, we will compute $\mathrm{div}\,(W u_i)$ in the distributional sense.
For every $\widetilde\chi\in C^\infty_c(\RR^2)$, we have
\begin{align*}
  (\textrm{div} (W\, u_i),\widetilde{\chi}) = -(W u_i, \nabla\widetilde{\chi})
  = -\int_{\RR^2}u_i(x) W (x)\cdot \nabla \widetilde{\chi}(x)\dd x.
\end{align*}
In view of the expression formula \eqref{eq:u-exp1} and Fubini's theorem, it follows that
\begin{align*}
  \big(\textrm{div} (W \, u_i),\widetilde{\chi}\big)
  & = -\int_{\RR^2} \Big(\int_{D}K_i(x-y)\dd y\Big) \,W(x)\cdot \nabla \widetilde{\chi}(x)\dd x \\
  & = -\int_{D}\int_{\RR^2}K_i(x-y) W(x)\cdot \nabla \widetilde{\chi}(x)\dd x\dd y \\
  & = -\lim_{\varepsilon\to 0}\int_{D}\int_{|x-y|\geq \varepsilon}
  K_i(x-y) W(x)\cdot \nabla \widetilde{\chi}(x)\dd x\dd y.
\end{align*}
Through integration by parts, we find that
\begin{align*}
  \int_D \int_{|x-y|\geq \varepsilon}K_i(x-y) W(x)\cdot \nabla \widetilde{\chi}(x)\dd x \dd y
  = - \int_D \int_{|x-y|\geq \varepsilon} W(x)\cdot \nabla_xK_i(x-y) \widetilde{\chi}(x)\dd x \dd y - I_\varepsilon ,
\end{align*}
with
\begin{align*}
  I_\varepsilon \triangleq
  \int_D \int_{|x-y|=\varepsilon}K_i(x-y) W(x)\cdot\frac{(x-y)}{|x-y|}\widetilde{\chi}(x)\dd S(x) \dd y.
\end{align*}
Note that $\lim_{\varepsilon\to0}I_\varepsilon$ might be infinite. Define
\begin{align*}
  J_\varepsilon \triangleq
  \int_D \int_{|x-y|=\varepsilon}K_i(x-y) W(y)\cdot\frac{(x-y)}{|x-y|}\widetilde{\chi}(x)\dd S(x) \dd y,
\end{align*}
and we can take the limit of their difference $R_\varepsilon \triangleq I_\varepsilon-J_\varepsilon$. Indeed, we apply \eqref{eq:G-prop1} and \eqref{eq:m-prop2} to obtain
\begin{align*}
  \lim_{\varepsilon\to0}|R_\varepsilon| 
  & = \lim_{\varepsilon\to0}\Big|\int_D \int_{|x-y|=\varepsilon}K_i(x-y) (W(x) - W(y))\cdot \frac{(x-y)}{|x-y|}\widetilde{\chi}(x)\dd S(x) \dd y\Big|\\
  & \leq \lim_{\varepsilon\to0}\int_D \Big|\frac{G(\varepsilon)}{\varepsilon}\cdot\|W\|_{\dot C^\mu} \varepsilon^\mu\cdot 1\cdot\lVert \widetilde{\chi}\rVert_{L^\infty} \cdot 2\pi\varepsilon\Big| \dd y 
  \leq C\lim_{\varepsilon\to 0} \varepsilon^\mu m(\varepsilon^{-1})= 0.
\end{align*}
For $J_\varepsilon$, using Fubini's theorem and integration by parts, we infer that
\begin{align}\label{eq:Jepsilon}
  J_\varepsilon & = \int_{\RR^2}\int_{D\cap\{|x-y|=\varepsilon\}} K_i(x-y) W(y)\cdot \frac{(x-y)}{|x-y|}\dd S(y)\widetilde{\chi}(x)\dd x\nonumber\\
  & = \int_{\RR^2}\bigg(\int_{D\cap\{|x-y| \geq\varepsilon\}} \hspace{-.2in} W(y)\cdot\nabla_y(K_i(x-y))\dd y 
  -\int_{\partial D\cap\{|x-y|\geq \varepsilon\}} \hspace{-.2in} K_i(x-y) W(y)\cdot \mathbf{n}(y)\,\dd S(y)\bigg)\widetilde{\chi}(x)\dd x\nonumber\\
  & = \int_{\RR^2}\int_{D\cap\{|x-y|\geq\varepsilon\}} W(y)\cdot\nabla_y(K_i(x-y))\dd y\,\widetilde{\chi}(x) \dd x\nonumber\\
  & =-\int_{\RR^2}\int_{D\cap\{|x-y|\geq\varepsilon\}} W(y)\cdot\nabla_xK_i(x-y)\dd y\,\widetilde{\chi}(x) \dd x,
\end{align}
where in the second equality, we use the fact that $W(y)\cdot \mathbf{n}(y)=0$ for $y\in \partial D$ with $\mathbf{n}$ the outward normal vector of $\partial D$.
Therefore, based on the above estimates,
\begin{align*}
  (\textrm{div} ( W u_i),\widetilde{\chi})
  = &\lim_{\varepsilon\to 0}\bigg(\int_{\RR^2}\int_{D\cap\{|x-y|\geq\varepsilon\}} \big(W(x)- W(y)\big) \cdot \nabla_xK_i(x-y)\dd y\, \widetilde{\chi}(x)\dd x +R_\varepsilon\bigg) \\
  = & \int_{\RR^2} \Big(\mathrm{p.v.}\int_{D}( W(x)- W(y))\cdot \nabla_xK_i(x-y)\dd y\Big)
  \widetilde{\chi}(x)\dd x.
\end{align*}
This finishes the proof of \eqref{eq:desingularization_tangential_vector}.
\end{proof}

Based on Lemma \ref{lem:tangential_velocity}, we show the following H\"older-type estimate of $\partial_W u$.
\begin{lemma}\label{lem:striated_est_mu}
Under the assumptions of Lemma \ref{lem:tangential_velocity}, the following estimate holds, 
\begin{align}\label{eq:part-W-u-Hold}
  \lVert \partial_W u\rVert_{C_\mathrm{m}^{\mu}(\RR^2)}
  \leq C (1+\log \mathbf{\Delta}_{\gamma})
  \lVert W\rVert_{\dot{C}^{\mu}(\RR^2)},
\end{align}
where $C>0$ depends only on $\mu$, $\gamma$ and $|D|$.
\end{lemma}

\begin{proof}
By virtue of the expression \eqref{eq:desingularization_tangential_vector}, \eqref{eq:gradKbound}, the rearrangement inequality, and \eqref{eq:m-prop2}, we deduce
\begin{align*}
  \lVert\partial_W u\rVert_{L^{\infty}}
  & \leq \int_D |W(x)-W(y)|\cdot|\nabla K(x-y)| \dd y\\
  &\leq C\lVert W\rVert_{\dot{C}^{\mu}(\RR^2)}\int_{D}|x-y|^{\mu}\frac{m(|x-y|^{-1})+1}{|x-y|^2}\,\dd y \\
  & \leq C\lVert W\rVert_{\dot{C}^{\mu}(\RR^2)}\int_{0}^{L}\frac{m(\rho^{-1})+1}{\rho^{1-\mu}}\,\dd \rho \leq C\lVert W\rVert_{\dot{C}^{\mu}(\RR^2)},
\end{align*}
with $L= \sqrt{|D|/\pi}$.
Choosing $x,h\in \RR^2$ such that $|h|\le\frac{\bar{c}_0}{3}$, we write
\begin{align*}
 \partial_W u(x)-\partial_W u(x+h) = \mathfrak{J}_1 + \mathfrak{J}_2 + \mathfrak{J}_3 + \mathfrak{J}_4,	
\end{align*}
where
\begin{align*}
  \mathfrak{J}_1 & \triangleq \mathrm{p.v.}\int_{D\cap\{|x-y|<2|h|\}}\big(W(x)-W(y)\big)\cdot\nabla K(x-y)\dd y, \\
  \mathfrak{J}_2 & \triangleq - \mathrm{p.v.}\int_{D\cap\{|x-y|<2|h|\}}\big(W(x+h)-W(y)\big)\cdot\nabla K(x+h-y)\dd y,\\
  \mathfrak{J}_3 & \triangleq \int_{D\cap\{|x-y|\geq 2|h|\}}\big(W(x)-W(x+h)\big)\cdot\nabla K(x-y)\dd y, \\
  \mathfrak{J}_4 & \triangleq \int_{D\cap\{|x-y|\geq 2|h|\}}\big(W(x+h)-W(y)\big)\cdot\big(\nabla K(x-y)-\nabla K(x+h-y)\big)\dd y.
\end{align*}
For $\mathfrak{J}_1$, from \eqref{eq:gradKbound}, we get
\begin{align*}
  |\mathfrak{J}_1| \leq C\lVert W\rVert_{\dot{C}^{\mu}(\RR^2)}\int_{|x-y|\leq 2|h|}|x-y|^{\mu}\frac{m(|x-y|^{-1})}{|x-y|^2}\dd y \leq  C\lVert W\rVert_{\dot{C}^{\mu}(\RR^2)}\int_0^{2|h|}\frac{m(\rho^{-1})}{\rho^{1-\mu}}\dd\rho.
\end{align*}
Using \eqref{eq:limit-beta1} in \eqref{H2a} and the L'H\^opital rule gives
\begin{align*}
  \lim_{\rho\to 0}\frac{\int_0^{\rho}\tilde\rho^{\mu-1}m(\tilde\rho^{-1})\dd \tilde\rho}{\rho^{\mu}m(\rho^{-1})}
  = \lim_{\rho\to 0}\frac{\rho^{\mu-1}m(\rho^{-1})}{\mu\rho^{\mu-1}m(\rho^{-1})-\rho^{\mu-1}m'(\rho^{-1})\rho^{-1}}
  =\lim_{r\to+\infty}\frac{1}{\mu-\frac{rm'(r)}{m(r)}}=\frac{1}{\mu}.
\end{align*}
Then, there exists a constant $C>0$ depending on $\mu$ and $\bar{c}_0$, such that for any $\rho<\bar{c}_0$,
\begin{align}\label{eq:approx-m-int-1}
  \int_0^{\rho}\tilde\rho^{\mu-1}m(\tilde\rho^{-1})\dd \tilde\rho\leq C\rho^{\mu}m(\rho^{-1}).
\end{align}
Together with \eqref{eq:m-prop3}, we find
\begin{align*}
  |\mathfrak{J}_1| \leq C \lVert W\rVert_{\dot{C}^{\mu}(\RR^2)} |h|^{\mu} m(|h|^{-1}).
\end{align*}
Performing the same procedure, we can also show that
\begin{align*}
  |\mathfrak{J}_2| \leq C\lVert W\rVert_{\dot{C}^{\mu}(\RR^2)} \int_0^{3|h|}\frac{m(\rho^{-1})}{\rho^{1-\mu}}\dd\rho
  \leq C\lVert W\rVert_{\dot{C}^{\mu}(\RR^2)} |h|^{\mu} m(|h|^{-1}).
\end{align*}
For the term $\mathfrak{J}_3$, it follows from \eqref{claim:refine_ineq} that
\begin{align*}
  |\mathfrak{J}_3| & \leq \lVert W\rVert_{\dot{C}^{\mu}(\RR^2)}|h|^{\mu}\Big|\mathrm{p.v.}\int_{D\cap\{|x-y|\geq 2|h|\}}\nabla K(x-y)\dd y\Big|\\
   &\leq  C \big( 1+\log \mathbf{\Delta}_{\gamma} \big)\lVert W\rVert_{\dot{C}^{\mu}(\RR^2)}|h|^{\mu}\big(m(|h|^{-1})+1\big).
\end{align*}
Finally, for the term $\mathfrak{J}_4$, we apply \eqref{eq:Kbound} with $l=2$ and get
\begin{align*}
 |\nabla^2 K(x)|\leq C\,\frac{m(|x|^{-1})+1}{|x|^3}.
\end{align*}
Utilizing the mean value theorem and proceeding as in the treatment of $II$ in Lemma~\ref{lem:u-est-general-data}, we obtain
\begin{align*}
  |\mathfrak{J}_4| & \leq C\lVert W\rVert_{\dot{C}^{\mu}(\RR^2)}\int_{D\cap\{|x-y|\geq 2|h|\}}|x-y|^\mu\,\frac{m(|x-y|^{-1})+1}{|x-y|^3}\,|h|\dd y \\
  & \leq C\lVert W\rVert_{\dot{C}^{\mu}(\RR^2)}|h|\int_{2|h|}^{L}\frac{m(\rho^{-1})+1}{\rho^{3-\mu}}\,\rho\dd \rho 
  \leq C\lVert W\rVert_{\dot{C}^{\mu}(\RR^2)}|h|^{\mu}\big(m(|h|^{-1})+1\big).
\end{align*}
Gathering the above estimates leads to the desired result \eqref{eq:part-W-u-Hold}.
\end{proof}

Next, we will present a refined point-wise estimate for the term
$\partial_W u = W\cdot\nabla u$ in the (tangential) direction $W = \nabla^{\perp}\varphi$, 
which is essentially from \cite{Elgindi14} but with more technical details here. 
\begin{lemma}\label{lem:es-u-whole}
Under the assumptions of Lemma \ref{lem:tangential_velocity},
we have the following point-wise estimate
\begin{align}\label{es:nab-u-tang}
  |\nabla u\, \mathbf{w}\cdot\mathbf{w}| \leq C \big(m(\mathbf{\Delta}_\gamma) + 1\big) \big(1+\log \mathbf{\Delta}_\gamma \big),
  \quad \forall \gamma\in (0 ,\mu],
\end{align}
where $\mathbf{w}\triangleq \frac{W}{|W|} = \frac{\nabla^\perp \varphi}{|\nabla^\perp \varphi|}$, and 
$C>0$ depends only on $\gamma$ and $|D|$.
\end{lemma}

\begin{proof}
Since $\nabla u\, \mathbf{w}\cdot\mathbf{w} = \mathbf{S}(\nabla u)\, \mathbf{w}\cdot\mathbf{w}$ with $\mathbf{S}(\nabla u) = \frac{\nabla u + (\nabla u)^T}{2}$ being the symmetric part of $\nabla u$, we will concentrate on the analysis of $\mathbf{S}(\nabla u)\, \mathbf{w}\cdot\mathbf{w}$. 
We recall from \eqref{def:Sym-nab-u} that
\begin{align*}
  \mathbf{S}(\nabla u)(x) = \int_{D}(\nabla K)_{\mathrm{S}}(x-y)\dd y = \int_D \frac{1}{2} \frac{2G(|x-y|)-|x-y|G'(|x-y|)}{|x-y|^2} \sigma(x-y) \dd y,
\end{align*}
where $\sigma$ is given by \eqref{def:sigma}.
For every $x\in \RR^2$, we divide $\mathbf{S}(\nabla u)(x)$ into
\begin{align*}
  \mathbf{S}(\nabla u)(x) & =  \int_{D \cap \{ |x-y| \geq \delta_\gamma \}}
  (\nabla K)_{\mathrm{S}}(x-y)\dd y  + \int_{D \cap \{|x-y| \leq \delta_\gamma \}}
  (\nabla K)_{\mathrm{S}}(x-y)  \dd y \\
  & \triangleq \mathcal{I}_1(x) + \mathcal{I}_2(x),
\end{align*}
where $\delta_\gamma\triangleq (\mathbf{\Delta}_\gamma)^{-\frac{1}{\gamma}}$.

For $\mathcal{I}_1$, we proceed as \eqref{eq:far-away-part} and obtain
\begin{align*}
  |\mathcal{I}_1(x)| \leq C \int_{D\cap \{|x-y|\geq\delta_\gamma\}} \frac{m(|x-y|^{-1})+1}{|x-y|^2}\,\dd y 
  \leq C m(\delta_\gamma^{-1}) \big(1+\log\delta_\gamma^{-1}\big),
\end{align*}
and consequently,
\begin{align*}
 |\mathcal{I}_1(x)\mathbf{w}(x)\cdot\mathbf{w}(x)|\leq C m(\mathbf{\Delta}_\gamma) \big(1+\log\mathbf{\Delta}_\gamma\big).	
\end{align*}

For $\mathcal{I}_2$, the goal is to exploit additional cancellations in order to sharpen the result of Lemma~\ref{lem:continuity-u}. Denote $\tilde x\in\partial D$ such that $|x-\tilde x|=d_x$, where $d_x$ is defined in \eqref{eq:dx}.
If $\delta_\gamma \leq d_x$, then by the mean-zero property of $\sigma(z)$, we have $\mathcal{I}_2(x)=0$. 
Thus, we assume that $\delta_\gamma >d_{x}$ without loss of generality, and
\begin{align*}
  \mathcal{I}_2(x) = \int_{D \cap \{ d_x\leq |x-y| \leq \delta_\gamma \}}
  (\nabla K)_{\mathrm{S}}(x-y) \dd y.
\end{align*}

Since $\varphi$ is a level-set characterization of $D$, we know that $\tilde{x}-x$ is parallel to $\nabla \varphi(\tilde{x})$, and hence orthogonal to $\mathbf{w}(\tilde{x})$. Therefore, we shall seek cancellations in $ \mathcal{I}_2(x)\mathbf{w}(\tilde{x})\cdot\mathbf{w}(\tilde{x})$.
Decompose
\begin{align*}
  \mathcal{I}_2(x)\mathbf{w}(x)\cdot\mathbf{w}(x) = \mathcal{I}_2(x)\mathbf{w}(\tilde{x})\cdot\mathbf{w}(\tilde{x}) + \mathcal{I}_2(x)\big(\mathbf{w}(x)+\mathbf{w}(\tilde{x})\big)\cdot\big(\mathbf{w}(x)-\mathbf{w}(\tilde{x})\big).
\end{align*}

Let us start with the estimate of the difference
\begin{align*}
|\mathcal{I}_2(x)\big(\mathbf{w}(x)+\mathbf{w}(\tilde{x})\big)\cdot\big(\mathbf{w}(x)-\mathbf{w}(\tilde{x})\big)|\leq 2|\mathcal{I}_2(x)|\cdot |\mathbf{w}(x)-\mathbf{w}(\tilde{x})|.
\end{align*}
By virtue of the definitions of $\mathbf{w}(x)$, $\tilde{x}$ and $d_x$,
we have
\begin{align*}
  |\mathbf{w}(x)-\mathbf{w}(\tilde{x})|
  & = \frac{\big|\big(W(x)-W(\tilde{x})\big)-\mathbf{w}(\tilde{x})\big(|W(x)|-|W(\tilde{x})|\big)\big|}{|W(x)|}\leq \frac{2\|\nabla\varphi\|_{\dot{C}^\gamma}|x-\tilde{x}|^\gamma}{|\nabla\varphi|_{\inf}}
   \leq 2\mathbf{\Delta}_\gamma (d_x)^\gamma.
\end{align*}
A similar argument as in \eqref{eq:crucial-est} yields
\begin{align*}
 |\mathcal{I}_2(x)| 
 \leq C\int_{d_x}^{\delta_\gamma}\frac{m(\rho^{-1})+1}{\rho}\,\mathcal{H}^1(\mathcal{R}_{\rho}(x))\dd \rho 
  \leq C\big(m(d_x^{-1})+1\big).
\end{align*}
Hence,
\begin{align*}
  |\mathcal{I}_2(x)\big(\mathbf{w}(x)+\mathbf{w}(\tilde{x})\big)\cdot\big(\mathbf{w}(x)-\mathbf{w}(\tilde{x})\big)|
  & \leq C \mathbf{\Delta}_\gamma (d_x)^\gamma\big(m(d_x^{-1})+1\big)
  \leq C \mathbf{\Delta}_\gamma (\delta_\gamma)^\gamma\big(m(\delta_\gamma^{-1})+1\big)\\
  &\leq C \big(m(\mathbf{\Delta}_\gamma^{\frac{1}{\gamma}})+1\big)
  \leq C \big(m(\mathbf{\Delta}_\gamma)+1\big).
\end{align*}
where we have used \eqref{eq:m-prop4} in the second inequality, and \eqref{eq:m-prop5} in the last inequality.
	
Next, we focus on the term $ \mathcal{I}_2(x)\mathbf{w}(\tilde{x})\cdot\mathbf{w}(\tilde{x})$. A main observation is that $(\nabla K)_{\mathrm{S}}(x-\cdot)\mathbf{w}(\tilde{x})\cdot\mathbf{w}(\tilde{x})$ is an odd function with respect to the line 
\begin{align*}
 \mathsf{L}:s\mapsto x+s\mathbf{w}(\tilde{x})^\perp,
\end{align*}
which go across $x$ and $\tilde{x}$. To see this, we denote $\bar{y}$ the reflection point of $y$ with respect to $\mathsf{L}$. Then, if we represent $y = x + s_1 \mathbf{w}(\tilde{x})^\perp + s_2\mathbf{w}(\tilde{x})$, we have $\bar{y} = x + s_1 \mathbf{w}(\tilde{x})^\perp - s_2\mathbf{w}(\tilde{x})$. This implies
\begin{align*}
 |x-\bar{y}| = \sqrt{s_1^2+s_2^2} =|x-y|.	
\end{align*}
Direct computation shows
\begin{align*}
  \sigma(x -y)\mathbf{w}(\tilde{x}) \cdot \mathbf{w}(\tilde{x})
  = \frac{2 \big(\mathbf{w}(\tilde{x})^\perp\cdot (x-y)\big) \big( \mathbf{w}(\tilde{x})\cdot (x-y)\big)}{|x-y|^2}=\frac{2s_1s_2}{s_1^2+s_2^2},
\end{align*}
and
\begin{align*}
  \sigma(x -\bar{y})\mathbf{w}(\tilde{x}) \cdot \mathbf{w}(\tilde{x})
  =-\frac{2s_1s_2}{s_1^2+s_2^2}=-\sigma(x -y)\mathbf{w}(\tilde{x}) \cdot \mathbf{w}(\tilde{x}).
\end{align*}
Hence, we conclude with
\begin{align*}
(\nabla K)_{\mathrm{S}}(x-\bar{y})\mathbf{w}(\tilde{x})\cdot\mathbf{w}(\tilde{x})	=-(\nabla K)_{\mathrm{S}}(x-y)\mathbf{w}(\tilde{x})\cdot\mathbf{w}(\tilde{x}).
\end{align*}
Define the half plane
\begin{align*}
  \mathfrak{F} \triangleq \{x\in \mathbb{R}^2 \, : \,\nabla \varphi(\tilde{x}) \cdot (x-\tilde{x}) \geq 0\}.
\end{align*}
Since $\mathfrak{F} \cap \{ d_x\leq |x-y| \leq \delta_\gamma \}$ 
is part of the annulus which is symmetric with respect to $\mathsf{L}$, 
we deduce that
\begin{align}\label{claim:cancel}
  \mathcal{I}_{21}(x) \mathbf{w}(\tilde{x})\cdot \mathbf{w}(\tilde{x})=0,
\end{align}
where
\begin{align*}
  \mathcal{I}_{21}(x) \triangleq \int_{\mathfrak{F} \cap \{ d_x\leq |x-y| \leq \delta_\gamma \}}
  (\nabla K)_{\mathrm{S}}(x-y) \dd y.
\end{align*}

Via \eqref{claim:cancel}, it remains to control the difference $\mathcal{I}_{22}(x)\mathbf{w}(\tilde{x})\cdot \mathbf{w}(\tilde{x})$ where $\mathcal{I}_{22}\triangleq \mathcal{I}_2-\mathcal{I}_{21}$ satisfies
\begin{align*}
  |\mathcal{I}_{22}(x)| & \leq \int_{( D \triangle \, \mathfrak{F})\cap \{ d_x\leq |x-y|\leq \delta_\gamma \}}
  \big|(\nabla K)_{\mathrm{S}}(x-y) \big|\dd y 
  \leq C\int_{( D \triangle \, \mathfrak{F})\cap \{ d_x\leq |x-y|\leq \delta_\gamma \}}\frac{m(|x-y|^{-1})+1}{|x-y|^2}\dd y\\
  & \leq C\int_{( D \triangle \, \mathfrak{F})\cap \{ 0\leq |\tilde{x}-y|\leq 2\delta_\gamma \}}
  \frac{m(|\tilde{x}-y|^{-1})+1}{|\tilde{x}-y|^2}\dd y,
\end{align*}
where we apply \eqref{eq:gradKbound} in the second inequality. 
For the last inequality, we have used the fact that
\begin{align*}
  0\leq |\tilde{x}-y|\leq |\tilde{x}-x|+|x-y| = d_x + |x-y| \leq 2|x-y|,
\end{align*}
which implies $\{ d_x\leq |x-y|\leq \delta_\gamma \}\subset\{ 0\leq |\tilde{x}-y|\leq 2\delta_\gamma \}$. We also make use of the monotonicity of $m$ as well as \eqref{eq:m-prop3}.

Now, we argue that the set $( D \triangle \, \mathfrak{F})\cap \{ 0\leq |\tilde{x}-y|\leq 2\delta_\gamma \}$ is small. 
Rewrite the integral in polar coordinates centered in $\tilde{x}$, and borrow the notation in \eqref{def:SSigmaR}. It yields
\begin{align*}
  |\mathcal{I}_{22}(x)| \leq C \int_0^{2\delta_\gamma} \frac{m(\rho^{-1})+1}{\rho}
  \mathcal{H}^1(\mathcal{S}_\rho\big(\tilde{x})\triangle\Sigma(x)\big) \rho \dd\rho.	
\end{align*}
A variation of the geometric lemma \eqref{lem:Geometric-lemma} (setting $d_x=0$) reads
\begin{align*}
  \mathcal{H}^1\big(\mathcal{S}_\rho\big(\tilde{x})\triangle\Sigma(x)\big)\leq 2\pi\cdot 2^\gamma 
  \Big(\frac{\rho}{\delta_\gamma}\Big)^\gamma,\quad \forall\rho>0.
\end{align*}
Then, we finish the estimate by
\begin{align*}
  |\mathcal{I}_{22}(x)| & \leq C\int_0^{2\delta_\gamma} \frac{m(\rho^{-1})+1}{\rho} 
  \Big(\frac{\rho}{\delta_\gamma}\Big)^\gamma \dd \rho 
  \leq C \Big((2\delta_\gamma)^\gamma m\big((2\delta_\gamma)^{-1}\big)(\delta_\gamma)^{-\gamma}+\frac{2^\gamma}{\gamma}\Big)\\
  & \leq C \big(m(\delta_\gamma^{-1})+1\big)\leq C \big(m(\mathbf{\Delta}_\gamma)+1\big),
\end{align*}
where we have used \eqref{eq:approx-m-int-1} in the second inequality, and \eqref{eq:m-prop5} in the last inequality.

Collecting all the estimates, we conclude with the desired inequality \eqref{es:nab-u-tang}.
\end{proof}

\subsection{The refined estimate on the flow map} 
In this section, we derive a sharper estimate for the flow map, compared to \eqref{eq:bound-flow-map-general} in Theorem~\ref{thm:Yudovich_type}, which will play a fundamental role in the following section.
\begin{align*}
 \widetilde{\nu}(\rho) \triangleq \rho\big(m(\rho^{-1})+1\big),
 \quad	\text{and}\quad
 \widetilde{\mathsf{H}}(r) \triangleq \int_{\frac1r}^1\frac{\dd \rho}{\widetilde{\nu}(\rho)} = 
 \int_{\frac1r}^1\frac{\dd\rho}{\rho\big(m(\rho^{-1})+1\big)},
\end{align*}
Compared to $\nu$ and $\mathsf{H}$, there is an improvement in a logarithmic factor. 
Similarly to $\mathsf{H}$, the function $\widetilde{\mathsf{H}}$ is invertible since $\widetilde{\mathsf{H}}'(r)=\frac{1}{r^2\widetilde{\nu}(r^{-1})}>0$. Moreover, from the Osgood condition \eqref{eq:Osgood-cond},
\begin{align*}
  \lim_{r\to+\infty}\widetilde{\mathsf{H}}(r) = \int_0^{1}\frac{\dd \rho}{\rho\big(m(\rho^{-1})+1\big)} 
  =\int_1^{+\infty}\frac{\dd r}{r\big(m(r)+1\big)}\geq\frac{\log2}{1+\frac{1}{m(1)}}\int_2^{+\infty}\frac{\dd r}{r(\log r)m(r)}=+\infty.
\end{align*}
Hence, the inverse function $\widetilde{\mathsf{H}}^{-1}(\cdot): \RR \to (0,+\infty)$ is also increasing monotonically.

\begin{lemma}\label{lem:flow-map-os}
Under the assumptions of Theorem \ref{thm:Yudovich_type}, if $u\in L^{1}([0,+\infty);C^1_{\mathrm{m}}(\RR^2))$, then we have that for any $t, s\in \mathbb{R}_+$,
\begin{align}\label{eq:bound-flow-map-refined}
  |\Phi_{t,s}(x)-\Phi_{t,s}(y)|^{-1}\geq \widetilde{\mathsf{H}}^{-1}
  \Big(\widetilde{\mathsf{H}}(|x-y|^{-1})- \Big|\int_s^{t}\lVert u(\tau)\rVert_{C^1_\mathrm{m}(\RR^2)}\dd \tau\Big|\Big).
\end{align}
Moreover, for any $t, s\in\mathbb{R}_+$,
\begin{align}\label{eq:flow-map-up-l}
  |\Phi_{t,s}(x)-\Phi_{t,s}(y)|\leq |x-y|\exp\bigg(\big(m(|x-y|^{-1}) + 1\big)
  \Big|\int_s^t\lVert u(\tau)\rVert_{C^1_\mathrm{m}(\RR^2)}\dd \tau \Big|\bigg).	
\end{align}
\end{lemma}

\begin{proof}
According to \eqref{eq:flow_map_def} and \eqref{eq:Csig-m}, it follows that for every $t,s\in \mathbb{R}_+$,
\begin{align*}
  \Big|\frac{\dd (\Phi_{t,s}(x) - \Phi_{t,s}(y))}{\dd t}\Big|
  \leq \big|u(\Phi_{t,s}(x),t)-u(\Phi_{t,s}(y),t)\big|
  \leq \lVert u(t)\rVert_{C^1_\mathrm{m}(\RR^2)} \,\widetilde\nu(|\Phi_{t,s}(x)-\Phi_{t,s}(y)|).
\end{align*}
Since $\frac{\dd}{\dd \rho}\widetilde{\mathsf{H}}(\rho^{-1})= - \frac{1}{\widetilde{\nu}(\rho)}$, we get
\begin{align*}
  \Big|\widetilde{\mathsf{H}}(|\Phi_{t,s}(x) - \Phi_{t,s}(y)|^{-1}) - \widetilde{\mathsf{H}}(|x-y|^{-1}) \Big| =\Big|\int_{|x-y|}^{|\Phi_{t,s}(x) - \Phi_{t,s}(y)|}\frac{\dd \rho}{\widetilde\nu(\rho)}\Big|
  \leq \Big|\int_s^t\lVert u(\tau)\rVert_{C^1_\mathrm{m}(\RR^2)}\dd \tau \Big|,
\end{align*}
which directly implies \eqref{eq:bound-flow-map-refined}.

If $|\Phi_{t,s}(x) - \Phi_{t,s}(y)|\leq |x-y|$, then \eqref{eq:flow-map-up-l} clearly holds. If $|\Phi_{t,s}(x) - \Phi_{t,s}(y)|> |x-y|$, we have
\begin{align*}
 \Big|\int_{|x-y|}^{|\Phi_{t,s}(x) - \Phi_{t,s}(y)|}\frac{\dd \rho}{\widetilde\nu(\rho)}\Big|
  \geq \frac{1}{m(|x-y|^{-1})+1} \int_{|x-y|}^{|\Phi_{t,s}(x) - \Phi_{t,s}(y)|}\frac{\dd \rho}{\rho} 
  = \frac{\log \frac{|\Phi_{t,s}(x) - \Phi_{t,s}(y)|}{|x-y|}}{m(|x-y|^{-1})+1},
\end{align*}
and consequently,
\begin{align*}
  |\Phi_{t,s}(x)-\Phi_{t,s}(y)|
  & \leq |x-y|\exp\bigg(\big(m(|x-y|^{-1}) + 1\big)
  \,\Big|\int_{|x-y|}^{|\Phi_{t,s}(x) - \Phi_{t,s}(y)|}\frac{\dd \rho}{\widetilde\nu(\rho)}\Big|\bigg)\\
  & \leq |x-y|\exp\bigg(\big(m(|x-y|^{-1}) + 1\big)
  \Big|\int_s^t\lVert u(\tau)\rVert_{C^1_\mathrm{m}(\RR^2)}\dd \tau \Big|\bigg).	
\end{align*}
This finishes the proof of \eqref{eq:flow-map-up-l}.
\end{proof}

\subsection{Closing estimates}\label{subsec:main-proof-1}
Recall that $\varphi(x,t)$ is a level-set characterization of the domain $D(t)$, 
satisfying \eqref{eq:level-set-eq}, and that $W = \nabla^\perp \varphi$. The dynamics \eqref{eq:parWk-1W} with $k=1$ reads 
\begin{align}\label{eq:transport-whole}
  \partial_t W + (u\cdot\nabla) W = W\cdot \nabla u  = \partial_W u,
  \quad W|_{t=0}=W_0.
\end{align}
Using the flow map $\Phi_{t,s}$ defined in \eqref{eq:flow_map_def}, we deduce 
\begin{align}\label{eq:tranport-intexp}
  W(x,t) = W_0(\Phi_{0,t}(x))
  +\int_0^t \big(\partial_W u\big)(s,\Phi_{s,t}(x))\dd s.
\end{align}
We apply the bound \eqref{eq:part-W-u-Hold} on $\lVert \partial_W u\rVert_{C_m^\mu(\RR^2)}$ in Lemma \ref{lem:striated_est_mu} and obtain the $\lVert W\rVert_{L^\infty}$ bound:
\begin{equation}\label{eq:ineq-L-infty}
  \lVert W(t)\rVert_{L^{\infty}} \leq \lVert W_0\rVert_{L^{\infty}}
  + \int_0^t \lVert \partial_W u (s)\rVert_{L^{\infty}}\dd s \leq \lVert W_0\rVert_{L^\infty}
  + C \int_0^t \big(1+\log \mathbf{\Delta}_\gamma\big) \lVert W(s) \rVert_{\dot{C}^{\mu(s)}}\dd s,
\end{equation}
for any $\mu\in (0,1)$, $\gamma\in(0,\mu]$ and the mapping $s\in [0,t]\mapsto \mu(s)\in [\frac{\mu}{2},\mu]$.

Next, we derive a lower bound estimate on $|\nabla\varphi|_{\inf}$. Taking an inner product on both sides of \eqref{eq:transport-whole} with $W$ leads to
\begin{align*}
  \partial_t|W|^2 + u\cdot \nabla |W|^2
  = 2 (\nabla u\, \mathbf{w}\cdot\mathbf{w}) |W|^2,
  \quad \textrm{with}\;\; \mathbf{w} = \tfrac{W}{|W|}.
\end{align*}
Then, we infer that
\begin{align*}
  \partial_t\big(\log |W\circ \Phi_{t,0}|\big)
  \leq |(\nabla u\, \mathbf{w}\cdot\mathbf{w}) \circ \Phi_{t,0}|.
\end{align*}
Thus, we obtain the following key result:
\begin{align*}
  |\nabla \varphi|_{\inf}(t) \geq |\nabla \varphi_0|_{\inf}\exp\Big(-\int_0^t
  \|(\nabla u\, \mathbf{w}\cdot\mathbf{w}) \|_{L^\infty} \dd s \Big).
\end{align*}
Taking advantage of Lemma \ref{lem:es-u-whole}, we know that for any $\gamma\in (0,\mu]$,
\begin{align*}
  \|\nabla u\, \mathbf{w}\cdot\mathbf{w}\|_{L^\infty}
  \leq C \big(m(\mathbf{\Delta}_\gamma)+1\big) (1+\log \mathbf{\Delta}_{\gamma}),
\end{align*}
with $\mathbf{\Delta}_\gamma = \frac{\|\nabla \varphi\|_{\dot{C}^\gamma}}{|\nabla \varphi|_{\inf}} + 1$.
Collecting the above estimates yields that for each $\gamma\in (0,\mu]$,
\begin{align}\label{eq:inf-part}
  |\nabla \varphi|_{\inf} (t) \geq |\nabla \varphi_0|_{\inf}
  \exp\Big(- C \int_0^t \big(m(\mathbf{\Delta}_\gamma(s))+1\big)\big(1+\log\mathbf{\Delta}_\gamma(s)\big)\dd s\Big).
\end{align}

Now we perform the method of losing estimates on $\|\nabla \varphi\|_{C^\mu(\RR^2)}=\|W\|_{C^\mu(\RR^2)}$,
see \cite[Section 3.3]{BCD11} for a general statement of this approach for the transport equation.
Our novelty is that here we can avoid the use of the Littlewood-Paley decomposition,
which is crucial when we tackle this problem in the domain with boundary.

Let $T>0$ and $\mu\in (0,1)$, $\varepsilon \in (0,\frac{\mu}{2})$ be fixed, and let $\gamma\in (0,\mu)$ be chosen later, 
then we define
\begin{align}\label{eq:V-t-def}
  V_\gamma(t) \triangleq \int_0^t \big(1+\log\,\mathbf{\Delta}_{\gamma}(s)\big) \dd s,\quad
  \eta_{\gamma,T} \triangleq \frac{\varepsilon}{V_\gamma(T)},\quad \mu(t) \triangleq \mu - \eta_{\gamma,T} V_\gamma(t).
\end{align}
It is clear to see that $\mu(0)=\mu$, $\mu(T)=\mu-\varepsilon\geq\frac\mu2$, and $\mu(t)$ is decreasing for $t\in [0,T]$.

For any $x,y\in \RR^2$, from \eqref{eq:tranport-intexp} we have
\begin{align}\label{eq:diff-phi}
  &|W(x,t) - W(y,t)| \leq \lVert W_0\rVert_{\dot C^\mu}|\Phi_{0,t}(x)-\Phi_{0,t}(y)|^\mu\nonumber \\
  & \qquad + \int_0^t\lVert \partial_W u (s)\rVert_{C^{{\mu(s)}}_\mathrm{m}}
  |\Phi_{s,t}(x)-\Phi_{s,t}(y)|^{\mu(s)} \big(m(|\Phi_{s,t}(x)-\Phi_{s,t}(y)|^{-1}) + 1\big) \dd s.
\end{align}
Recalling Lemma \ref{lem:continuity-u}, for any $\gamma\in (0,1)$, there exists $\widetilde{C} \geq 1$ such that
\begin{align*}
  \lVert u(t)\rVert_{C^1_{\mathrm{m}}(\mathbb{R}^2)}\leq \widetilde{C}\big(1+\log \mathbf{\Delta}_{\gamma}(t) \big),
  \quad t\in \mathbb{R}_+.
\end{align*}
Applying \eqref{eq:flow-map-up-l} in Lemma \ref{lem:flow-map-os}, we deduce that for every $t\geq s \geq 0$,
\begin{align}\label{eq:est-diff-Phi-s-t}
  |\Phi_{s,t}(x)-\Phi_{s,t}(y)| & \leq |x-y|\exp\Big(\widetilde{C} \big(m(|x-y|^{-1})+1\big) \int^t_s (1+\log \mathbf{\Delta}_\gamma(\tau))\dd \tau\Big)\nonumber\\
  & = |x-y|\, e^{\mathfrak{M}_{s,t,\gamma}(|x-y|^{-1})},
\end{align}
where for convenience, we introduce the shortcut notation
\begin{align}\label{def:MN}
  \mathfrak{M}_{s,t,\gamma}(r) \triangleq \widetilde{C} \big(m(r)+1\big) \int^t_s (1+\log \mathbf{\Delta}_\gamma(\tau))\dd \tau.
\end{align}
It follows from \eqref{eq:m-prop4}, \eqref{eq:est-diff-Phi-s-t}, and the monotonicity of $m$ that  
\begin{align}
  & |\Phi_{s,t}(x)- \Phi_{s,t}(y)|^{\mu(s)} \big(m(|\Phi_{s,t}(x)-\Phi_{s,t}(y)|^{-1}) +1\big) \nonumber\\
  & \leq C |x-y|^{\mu(s)}e^{\mu(s)\, \mathfrak{M}_{s,t,\gamma}(|x-y|^{-1}) } \Big( m\big(|x-y|^{-1} e^{-\mathfrak{M}_{s,t,\gamma}(|x-y|^{-1})}\big) + 1\Big) \nonumber\\
  & \leq C |x-y|^{\mu(s)} \big(m(|x-y|^{-1})+1\big) e^{\mu \,\mathfrak{M}_{s,t,\gamma}(|x-y|^{-1})},\label{eq:est-diff-Mphi-t-s}
\end{align}
where $C>0$ depends only on $\mu$ (recalling that $\mu(s)\in [\mu-\varepsilon, \mu]\subset [\frac{\mu}{2},\mu]$). 
Applying the above estimates to \eqref{eq:diff-phi}, we find
\begin{align}\label{eq:est-main-ineq}
  &\frac{|W(x,t) - W(y,t)|}{|x-y|^{\mu(t)}}
   \leq \lVert W_0\rVert_{\dot{C}^\mu}|x-y|^{\mu-\mu(t)} e^{\mu \,\mathfrak{M}_{0,t,\gamma}(|x-y|^{-1})}
  \nonumber \\
  & \qquad + C_1 \int_0^t(1+\log \mathbf{\Delta}_{\gamma}(s))
  \lVert W(s)\rVert_{\dot{C}^{\mu(s)}}|x-y|^{\mu(s)-\mu(t)}\big(m(|x-y|^{-1})+1\big)
  e^{\mu\, \mathfrak{M}_{s,t,\gamma}(|x-y|^{-1})} \dd s \nonumber \\
  & \quad \triangleq J_1+J_2,
\end{align}
where $C_1>0$ depend only on $\mu$ and $\gamma$.

In the following, we divide into three cases to discuss the estimate of \eqref{eq:est-main-ineq}.

\medskip\noindent\textbf{Case I}: Suppose that $|x-y|\in(0,1]$ is small enough such that
\begin{align}\label{reg-cond:high-freq}
  \frac{\log |x-y|^{-1}}{m(|x-y|^{-1})+1}\geq \frac{C' \widetilde{C}}{\eta_{\gamma,T}},
\end{align}
where the constant $C'>\mu$ will be fixed later.
By \eqref{reg-cond:high-freq}, we have that
\begin{align*}
  |x-y|^{\mu(s)-\mu(t)} & = \exp\Big(\eta_{\gamma,T}\log |x-y|
  \int^t_{s}(1+\log \mathbf{\Delta}_{\gamma}(\tau))\dd \tau\Big) \\
  & \leq \exp\Big(-C' \widetilde{C} \big(m(|x-y|^{-1})+1\big) \int^t_{s}(1+\log \mathbf{\Delta}_{\gamma}(\tau))\dd \tau \Big)
  =e^{-C' \,\mathfrak{M}_{s,t,\gamma}(|x-y|^{-1})}.
\end{align*}
Consequently, the terms $J_1$ and $J_2$ can be estimated as follows
\begin{align*}
  J_1\leq \lVert W_0\rVert_{\dot{C}^\mu}e^{-(C' - \mu)\,\mathfrak{M}_{0,t,\gamma}(|x-y|^{-1})},
\end{align*}
and
\begin{align*}
  J_2	& \leq C_1 \sup_{s\in [0,t]}\lVert W(s)\rVert_{\dot{C}^{\mu(s)}}
  \int_0^t  \big(1+\log\mathbf{\Delta}_\gamma(s)\big) \big(m(|x-y|^{-1})+1\big)
  e^{-(C'- \mu)\,\mathfrak{M}_{s,t,\gamma}(|x-y|^{-1})}\dd s \\
  &= \frac{C_1}{\widetilde{C}}\sup_{s\in [0,t]}\lVert W(s)\rVert_{\dot{C}^{\mu(s)}}\int_0^t
  e^{-(C'- \mu)\, \mathfrak{M}_{s,t,\gamma}(|x-y|^{-1})}\tfrac{\dd}{\dd s} \big(-\mathfrak{M}_{s,t,\gamma}(|x-y|^{-1})\big) \dd s\\
  & \leq \frac{1}{ C' - \mu }\frac{C_1}{\widetilde{C}} \sup_{s\in [0,t]}\lVert W(s)\rVert_{\dot{C}^{\mu(s)}}.
\end{align*}
Hence, choosing $C' = \mu + 2 \frac{C_1}{\widetilde{C}}$, the estimates above imply
\begin{align}\label{eq:Holder-ineq-high-freq}
  \frac{|W(x,t) - W(y,t)|}{|x-y|^{\mu(t)}}
  \leq \lVert W_0\rVert_{\dot{C}^\mu}
  + \frac{1}{2}\sup_{s\in [0,t]}\lVert W(s)\rVert_{\dot{C}^{\mu(s)}}.
\end{align}

\medskip\noindent\textbf{Case II}: Suppose $|x-y|\in(0,1]$ and \eqref{reg-cond:high-freq} is not satisfied, that is,
\begin{align}\label{reg:low-freq}
  \frac{\log |x-y|^{-1}}{m(|x-y|^{-1})+1 }\leq \frac{C' \widetilde{C}}{\eta_{\gamma,T}} = \frac{C' \widetilde{C} V_\gamma(T)}{\varepsilon}.
\end{align}
Thanks to \eqref{eq:m-prop2}, Lemma \ref{lem:H2H3} and \eqref{reg:low-freq}, for $|x-y|\le \frac{1}{2}$, we have
\begin{align*}
  \frac{\log^{\frac{1}{2}}|x-y|^{-1}}{C} \leq \frac{\log |x-y|^{-1}}{C(1+\log^{\frac{1}{2}} |x-y|^{-1})}
  \leq \frac{\log |x-y|^{-1}}{m(|x-y|^{-1})+1}\leq \frac{C' \widetilde{C} V_\gamma(T)}{\varepsilon},
\end{align*}
which implies that (for $|x-y|\in [\frac12,1)$, the following estimate is clear)
\begin{align*}
  \log |x-y|^{-1}\leq \Big(\tfrac{C C'\widetilde{C} V_\gamma(T)}{\varepsilon}\Big)^2,
  \quad \textrm{and}\quad |x-y|^{-1}\leq e^{\big(\frac{C C' \widetilde{C} V_\gamma(T)}{\varepsilon}\big)^2}.
\end{align*}
Consequently, we have
\begin{align*}
  m(|x-y|^{-1})\leq \widetilde{m} \Big(\big(\tfrac{CC'\widetilde{C} V_\gamma(T)}{\varepsilon}\big)^2\Big)
  \leq C\widetilde{m}\Big(\tfrac{CC'\widetilde{C} V_\gamma(T)}{\varepsilon}\Big)
  \leq C \big(\widetilde{m}(V_\gamma(T))+1\big) = C \big(m(e^{V_\gamma(T)})+1\big),
\end{align*}
where we have used \eqref{H1}, \eqref{eq:H3} and \eqref{eq:m-prop3-2} in the three inequalities, respectively, and the constants $C'$, $\widetilde{C}$ and $\varepsilon$ are absorbed in $C$.
Together with \eqref{eq:est-main-ineq}, and the definition of $\mathfrak{M}_{s,t,\gamma}(\cdot)$ in \eqref{def:MN}, 
it follows that
\begin{align*}
  J_1\leq \lVert W_0\rVert_{\dot{C}^\mu}\exp\Big( C \big(m(e^{V_\gamma(T)})+1\big)
  \int_0^t\big(1+\log \mathbf{\Delta}_\gamma(\tau)\big) \dd \tau\Big) \leq \lVert W_0\rVert_{\dot{C}^\mu} e^{C \,\mathfrak{M}_{0,t,\gamma}(e^{V_\gamma(T)})}.
\end{align*}
Similarly, since $\mu(s)\geq\mu(t)$ for any $t\geq s$, we also deduce that
\begin{align*}
  J_2 \leq  C \big(m(e^{V_\gamma(T)})+1\big)\int_0^t\big(1+\log\mathbf{\Delta}_\gamma(s)\big) \lVert W(s)\rVert_{\dot{C}^{\mu(s)}}e^{C \,\mathfrak{M}_{s,t,\gamma}(e^{V_\gamma(T)})}\dd s.
\end{align*}
Hence, in this case, collecting the above estimates yields
\begin{align}\label{eq:Holder-ineq-low-freq}
  \frac{|W(x,t) - W(y,t)|}{|x-y|^{\mu(t)}}
  &\leq \lVert W_0\rVert_{\dot{C}^\mu} e^{C \,\mathfrak{M}_{0,t,\gamma}(e^{V_\gamma(T)})}\\
  & \quad  + C \big(m(e^{V_\gamma(T)})+1\big)\int_0^t\big(1+\log\mathbf{\Delta}_\gamma(s)\big)
  \lVert W(s)\rVert_{\dot{C}^{\mu(s)}}e^{C \,\mathfrak{M}_{s,t,\gamma}(e^{V_\gamma(T)})}\dd s.\nonumber
\end{align}

\medskip\noindent\textbf{Case III}:
When $|x-y|>1$, we directly use the $\lVert W\rVert_{L^{\infty}}$ estimate \eqref{eq:ineq-L-infty} to obtain
\begin{align}\label{eq:Holder-half-norm-rou-part}
  \frac{|W(x,t) - W(y,t)|}{|x-y|^{\mu(t)}}\leq 2\lVert W\rVert_{L^{\infty}}
  \leq 2\lVert W_0\rVert_{L^\infty} + C \int_0^t \big(1+\log \mathbf{\Delta}_\gamma(s)\big) \lVert W(s) \rVert_{\dot{C}^{\mu(s)}}\dd s.
\end{align}

\medskip
Based on the analysis of the three cases, it follows from \eqref{eq:Holder-ineq-high-freq}, 
\eqref{eq:Holder-ineq-low-freq} and \eqref{eq:Holder-half-norm-rou-part} that
\begin{align}
   \sup_{|x-y|>0}\frac{|W(x,t) - W(y,t)|}{|x-y|^{\mu(t)}}
  &\leq  \frac{1}{2} \sup_{s\in [0,t]}\lVert W\rVert_{C^{\mu(s)}}
  + C\lVert W_0\rVert_{\dot{C}^\mu}e^{C \,\mathfrak{M}_{0,t,\gamma}(e^{V_\gamma(T)})}\notag\\
  & \quad  + C \big(m(e^{V_\gamma(T)})+1\big)\int_0^t\big(1+\log\mathbf{\Delta}_\gamma(s)\big)
  \lVert W(s)\rVert_{\dot{C}^{\mu(s)}}e^{C \,\mathfrak{M}_{s,t,\gamma}(e^{V_\gamma(T)})}\dd s.\nonumber
\end{align}
Together with \eqref{eq:ineq-L-infty}, we deduce that
\begin{align*}
  \lVert W(t)\rVert_{C^{\mu(t)}} 
  & \leq \frac{1}{2} \sup_{s\in [0,t]}
  \lVert W(s)\rVert_{C^{\mu(s)}} + C\lVert W_0\rVert_{\dot{C}^\mu}e^{C \,\mathfrak{M}_{0,t,\gamma}(e^{V_\gamma(T)})}\\
  & \quad + C \big(m(e^{V_\gamma(T)})+1\big)\int_0^t\big(1+\log\mathbf{\Delta}_\gamma(s)\big)
  \lVert W(s)\rVert_{\dot{C}^{\mu(s)}}e^{C \,\mathfrak{M}_{s,t,\gamma}(e^{V_\gamma(T)})}\dd s.
\end{align*}
Taking supreme on $s\in[0,t]$, it immediately leads to
\begin{align}\label{eq:ineq-Holder-est}
  \sup_{s\in [0,t]}\lVert W(s)\rVert_{C^{\mu(s)}}
  & \leq C\lVert W_0\rVert_{C^\mu}e^{C \,\mathfrak{M}_{0,t,\gamma}(e^{V_\gamma(T)})}\nonumber\\
  & \quad + C \big(m(e^{V_\gamma(T)})+1\big)\int_0^t\big(1+\log\mathbf{\Delta}_\gamma(s)\big)
  \lVert W(s)\rVert_{\dot{C}^{\mu(s)}}e^{C \,\mathfrak{M}_{s,t,\gamma}(e^{V_\gamma(T)})}\dd s.
\end{align}
Denote
\begin{align*}
  \mathcal{E}(t) \triangleq \sup_{s\in [0,t]}\lVert W(s)\rVert_{C^{\mu(s)}}
  e^{-C \,\mathfrak{M}_{0,t,\gamma}(e^{V_\gamma(T)})}.
\end{align*}
Then we find from \eqref{eq:ineq-Holder-est} that
\begin{align*}
  \mathcal{E}(t)\leq C \lVert W_0\rVert_{C^\mu}
  + C \big(m(e^{V_\gamma(T)})+1\big) \int_0^t\big(1+\log\mathbf{\Delta}_\gamma(s)\big) \mathcal{E}(s)\dd s.
\end{align*}
A direct application of the Gr\"onwall's inequality implies
\begin{align*}
 \mathcal{E}(t)\leq C \lVert W_0\rVert_{C^\mu}\exp\Big(\int_0^t C \big(m(e^{V_\gamma(T)})+1\big)\big(1+\log\mathbf{\Delta}_\gamma(s)\big)\dd s\Big) = C \lVert W_0\rVert_{C^\mu} e^{C \,\mathfrak{M}_{0,t,\gamma}(e^{V_\gamma(T)})}.
\end{align*}
Therefore, we conclude that
\begin{align}\label{ineq:holder-part}
  \lVert \nabla \varphi(T)\rVert_{C^{\mu-\varepsilon}} & =  \lVert W(T)\rVert_{C^{\mu(T)}} 
  \leq \mathcal{E}(T) e^{C \,\mathfrak{M}_{0,T,\gamma}(e^{V_\gamma(T)})}
  \leq C\lVert W_0\rVert_{C^{\mu}} e^{2C \,\mathfrak{M}_{0,T,\gamma}(e^{V_\gamma(T)})}\nonumber\\ 
  & \leq C\lVert \nabla \varphi_0\rVert_{C^{\mu}} \exp\Big(C \big(m(e^{V_\gamma(T)})+1\big)V_\gamma(T)\Big),
\end{align}
for any $T>0$ and $\gamma\in (0,1)$, and the constant $C$ is independent of $T$.

Now, choosing $\gamma=\mu-\varepsilon$, from \eqref{eq:inf-part} and \eqref{ineq:holder-part}, we find that
(recalling that $\mathbf{\Delta}_{\mu -\varepsilon}$ is defined in \eqref{def:Delta-gamma})
\begin{align}\label{eq:Deltabound}
  \mathbf{\Delta}_{\mu-\varepsilon}(T)
  \leq C\exp\Big(C
  \big(m(e^{V_{\mu-\varepsilon}(T)})+1\big)V_{\mu-\varepsilon}(T) + C \int_0^{T}m\big(\mathbf{\Delta}_{\mu -\varepsilon}(s)\big)
  \big(1+\log \mathbf{\Delta}_{\mu -\varepsilon}(s)\big) \dd s\Big),
\end{align}
where $C= C(\mu,\varepsilon) > 0$ is independent of $T$,
and $V_{\mu-\varepsilon}(T)$ is given by \eqref{eq:V-t-def}.
By setting
\begin{align}\label{def:fandM}
  f(t)\triangleq V_{\mu-\varepsilon}'(t)= 1 + \log \mathbf{\Delta}_{\mu -\varepsilon}(t), \quad
  \mathcal{M}(r) \triangleq
  \begin{cases}
    r \big(\widetilde{m}(r)+1\big) = r \big(m(e^r)+1\big), & \forall\, r\geq r_0, \\
    r \big(m(e^{r_0})+1\big), & \forall\, 0< r \leq r_0,
  \end{cases}
\end{align}
where $r_0\geq 2$ is the constant appearing in \eqref{eq:r-tildm-prop}, we rewrite the above inequality as
\begin{equation}\label{ineq:main-whole}
\begin{aligned}
  f(t) & \leq C+C\Big(\widetilde{m}\Big(\int_0^t f(s)\dd s\Big)+1\Big)\int_0^t f(s)\dd s
  + C \int_0^t \widetilde{m}\big(f(s)\big)f(s)\dd s \\
  & \leq C+C\mathcal{M}\Big(\int_0^t f(s)\dd s\Big)+C\int_0^t \mathcal{M}(f(s))\dd s,
\end{aligned}
\end{equation}
where $C>0$ may depend on $\mu$, $\varepsilon$, and the initial data.
According to \eqref{eq:r-tildm-prop}, \eqref{eq:m-prop1b} and \eqref{def:fandM}, we have
\begin{align}\label{eq:Mr-prop}
  \textrm{$r\mapsto \mathcal{M}(r)$ is strictly increasing and convex in $(0,+\infty)$},
\end{align}
and for any $\lambda>0$, $r> 0$, $\epsilon>0$
\begin{align}\label{eq:mathcalM}
  \mathcal{M}(\lambda r) \leq C\lambda r\big(\widetilde{m}(\lambda r)+1\big)\leq  C\lambda(1+\log_{+}^{\beta+\epsilon} \lambda)\mathcal{M}(r).
\end{align}
Applying Jensen's inequality, we get
\begin{align*}
  \mathcal{M}\Big(\int_0^t f(s)\dd s\Big)
  = \mathcal{M}\Big(\frac{1}{t}\int_0^t tf(s)\dd s\Big)
  \leq \frac{1}{t}\int_0^t \mathcal{M}(tf(s))\dd s
  \leq C\log^{\beta+\epsilon}(e+t)\int_0^t\mathcal{M}(f(s))\dd s.
\end{align*}
Consequently, it follows from \eqref{ineq:main-whole} that
\begin{align}\label{ineq:f}
  f(t)\leq C+C\log^{\beta+\epsilon}(e+t)\int_0^t\mathcal{M}(f(s))\dd s.
\end{align}
Denote by
\begin{align*}
  R(t) \triangleq \frac{r_0}{\log^{\beta+\epsilon}(e+t)}+\int_0^{t}\mathcal{M}(f(s))\dd s,
\end{align*}
so that $R(0) = r_0$ and $f(t)\leq C\log^{\beta+\epsilon}(e+t)R(t)$.
Then, together with \eqref{eq:mathcalM}, we obtain 
\begin{align*}
  \frac{\dd R(t)}{\dd t} & = - \frac{\beta+\epsilon}{(e+t)\log^{1+\beta+\epsilon}(e+t)} + \mathcal{M}\big(f(t)\big) \\
  & \leq \mathcal{M}\big(C\log^{\beta+\epsilon}(e+t)R(t)\big) \leq C\log^{\beta+2\epsilon}(e+t)\mathcal{M}(R(t)).
\end{align*}
Integrating the above differential inequality on $t\in[0,T]$ yields
\begin{align*}
  \mathcal{H}(R(T)) = \mathcal{H}(R(T))-\mathcal{H}(R(0))\leq C(1+T)\log^{\beta+2\epsilon}(e+T),
\end{align*}
where
\begin{align}\label{def:Hr}
  \mathcal{H}(r)\triangleq \int_{r_0}^r \frac{\dd \tilde{r}}{\mathcal{M}(\tilde{r})}
  = 
  \begin{cases}
    \int_{r_0}^r \frac{1}{\tilde{r} (\widetilde{m}(\tilde{r})+1)} \dd \tilde{r},\quad & \forall\, r\geq r_0, \\
    \frac{1}{m(e^{r_0})+1} \log \frac{r}{r_0},\quad & \forall\, 0< r \leq r_0.
  \end{cases}
\end{align}
Note that $\mathcal{H}(R(0))=\mathcal{H}(r_0)=0$, $\lim_{r\to0}\mathcal{H}(r)=-\infty$, and the Osgood condition \eqref{eq:Osgood-cond} implies
\begin{align*}
 \lim_{r\to+\infty}\mathcal{H}(r)=\int_{r_0}^{+\infty}\frac{1}{r(\widetilde{m}(r)+1)}\dd r
 \geq \frac{1}{1+\frac{1}{\widetilde{m}(r_0)}}\int_{r_0}^{+\infty}\frac{\dd r}{r\,\widetilde{m}(r)}=+\infty.
\end{align*}
Together with \eqref{eq:Mr-prop}, we infer that $r\mapsto \mathcal{H}(r)$ is strictly increasing in $[0,+\infty)$,
and it has a unique inverse function $\mathcal{H}^{-1}(\cdot)$ in $(-\infty,+\infty)$.
Hence, it follows that for any $T>0$,
\begin{align}\label{eq:fT-est-final}
  f(T)\leq \log^{\beta+\epsilon}(e+T)R(T) \leq \log^{\beta+\epsilon}(e+T)\, \mathcal{H}^{-1}\big(C(1+T)\log^{\beta+2\epsilon}(e+T)\big).
\end{align}
This implies the boundedness of $\mathbf{\Delta}_{\mu -\varepsilon}(T)$, which leads to the $C^{\mu-\varepsilon}$ boundary regularity. Indeed, from \eqref{eq:inf-part}, \eqref{ineq:holder-part} and \eqref{eq:fT-est-final}, we have that for any $T>0$,
\begin{align}\label{eq:gradphiinf}
  |\nabla \varphi(T)|_{\inf}
  & \geq |\nabla \varphi_0|_{\inf}\exp \Big(-C\int_0^T \mathcal{M}(f(t))\dd t \Big) \geq |\nabla\varphi_0|_{\inf}\,e^{-CR(T)}\nonumber\\
  & \geq |\nabla\varphi_0|_{\inf}\exp\Big(-C\mathcal{H}^{-1}\big(C(1+T)\log^{\beta+2\epsilon}(e+T)\big)\Big),
\end{align}
and
\begin{align}\label{eq:gradphiCmu}
  \lVert \nabla \varphi(T)\rVert_{C^{\mu-\varepsilon}}
  & \leq C \lVert \nabla \varphi_0\rVert_{C^{\mu}}
  \exp\Big(C\mathcal{M}\Big(\int_0^T f(t)\dd t\Big)\Big) \leq C \lVert \nabla \varphi_0\rVert_{C^\mu}\exp\Big(C\log^{\beta+\epsilon}(e+T)R(T)\Big)\nonumber\\
  &\leq C \lVert \nabla \varphi_0\rVert_{C^\mu}\exp\Big(C\log^{\beta+2\epsilon}(e+T)\mathcal{H}^{-1}\big(C(1+T)\log^{\beta+2\epsilon}(e+T)\big)\Big).
\end{align}

To conclude, we have proved the following theorem. 
\begin{theorem}\label{thm:regularityCmu}
Let $\mu\in (0,1)$.
Assume that $m(\xi)= m(|\xi|)$ is a radial function of $\RR^2$ with $m(r)$ 
satisfying \eqref{H1}-\eqref{H2a}-\eqref{eq:Osgood-cond}.
Consider the unique global patch solution \eqref{intro:single-patch} 
of the 2D Loglog-Euler type equation \eqref{m-SQG} associated with the patch data 
$\omega_0 = \mathbf{1}_{D_0}$ where $D_0$ is a simply connected and bounded domain with boundary $\partial D_0\in C^{1,\mu}$. 
Then for any $t>0$ and any $\varepsilon\in (0,\mu)$, 
the patch boundary $\partial D(t)$ persists the $C^{1,\mu-\varepsilon}$-regularity, 
namely, $\partial D(t)\in C^{1,\mu-\varepsilon}$. 

More precisely, for any given $\varepsilon>0$ and any $\epsilon>0$, there is some constant $C>0$ depending on $\mu$, $\varepsilon$, $\epsilon$ and the initial data, such that the estimates \eqref{eq:gradphiinf} and \eqref{eq:gradphiCmu} hold, where $\varphi(\cdot)$ is the solution of \eqref{eq:level-set-eq} and the mapping $\mathcal{H}(\cdot)$ is given by \eqref{def:Hr}.
\end{theorem}

In view of the definitions of $\mathsf{H}$ in \eqref{def:H-r} and $\mathcal{H}$ in \eqref{def:Hr}, we have that for any $r\geq r_0$,
\begin{align*}
 \mathcal{H}(r) & \geq \frac{1}{1+\frac{1}{\widetilde{m}(r_0)}}\int_{r_0}^{r}\frac{\dd \tilde{r}}{\tilde{r}\,\widetilde{m}(\tilde{r})} = C\int_{e^{r_0}}^{e^r}\frac{\dd \tilde{r}}{\tilde{r}(\log\tilde{r})m(\tilde{r})}\\
 & = C\Big(\mathsf{H}(e^r) - \int_2^{e^{r_0}}\frac{\dd \tilde{r}}{\tilde{r}(\log\tilde{r})m(\tilde{r})}\Big) = C\big(\mathsf{H}(e^r)-C\big).
\end{align*}
This implies
\begin{align}\label{eq:HH}
 \mathcal{H}^{-1}(\mathrm{y})\leq \log \big(\mathsf{H}^{-1}(C\mathrm{y}+C)\big),\quad\forall~\mathrm{y}\geq0.	
\end{align}
Applying the bound to \eqref{eq:gradphiinf} and \eqref{eq:gradphiCmu} yields the estimates \eqref{intro-est:inf-1} and \eqref{intro-est:reg-1}. Then, a direct application of \eqref{eq:z0bound}-\eqref{eq:zkbound}-\eqref{eq:zHolder} gives \eqref{intro-est:reg-z}, thereby completing the proof of Theorem~\ref{thm:gSQG-GR-whole-space} for $n=1$.

\section{\texorpdfstring{Global regularity of $C^{n,\mu}$ single vortex patch}{Global regularity of C(n,mu) single vortex patch}}
\label{sec:main-proof-n}
In this section, we consider the global propagation of the higher boundary regularity of the patch solution 
associated with the initial data
$\omega_0(x) = \mathbf{1}_{D_0}(x)$ and $\partial D_0\in C^{n,\mu}$ with $n\geq 2$ and $\mu\in (0,1)$.

Now, we can present our main result in this section.
\begin{theorem}\label{thm:propagation_regularity}
  Let $\mu\in (0,1)$ and $n\in \mathbb{N}^\star\cap [2,+\infty)$. Assume that $m(\xi)= m(|\xi|)$ is a radial function of $\RR^2$ with $m(r)$ satisfying \eqref{H1}-\eqref{H2a}-\eqref{eq:Osgood-cond}.
 Let $\varphi_0\in C^{n,\mu}(\RR^2)$ be a level-set characterization of the domain $D_0$ that is bounded and simply connected. 
Then the global solution of \eqref{eq:level-set-eq} satisfies $\partial D(t)\in C^{n,\mu'}$ for any $t>0$ and any $\mu'\in (0,\mu)$.
\end{theorem} 

According to the framework in Section \ref{subsec:formulation}, to prove Theorem \ref{thm:propagation_regularity}, it is sufficient to control $\partial_W^{n-1}W$. More precisely, we will show that if $\varphi_0\in C^{n,\mu}(\RR^2)$, then for any $\mu'\in (0,\mu)$, there exists an increasing and positive function $R_{n,\mu,\mu'}(t)$ such that
\begin{align}\label{eq:tangential_regularity}
  \lVert \partial_W^{n-1}W(t) \rVert_{C^{\mu'}(\RR^2)}  \leq R_{n,\mu,\mu'}(t)<+\infty.
\end{align}
Indeed, given \eqref{eq:tangential_regularity}, for any $\gamma\in(0,1)$, $\mu\in(0,1)$ and $\mu'\in (0,\mu)$, we have
\begin{align*}
  \sum_{j=1}^{n-1}\lVert \partial_W^{j-1}W(t)\rVert_{C^\gamma(\RR^2)}
  +\lVert \partial_W^{n-1}W (t)\rVert_{C^{\mu'}(\RR^2)}  \leq \sum_{j=1}^{n-1}R_{j,\frac{\gamma+1}{2},\gamma}(t) + R_{n,\mu,\mu'}(t)<+\infty,
\end{align*}
where we have used the fact that $C^{n,\mu}\subset C^{j,\frac{\gamma+1}{2}}$ when $j<n$.
Consequently, we deduce $\partial D(t)\in C^{n,\mu'}$.

Moreover, the growth rate of $R_{n,\mu,\mu'}(t)$ matches with \eqref{eq:gradphiCmu}, namely
\begin{align}\label{eq:Rgrowth}
	R_{n,\mu,\mu'} \leq C\lVert\varphi_0\rVert_{C^{n,\mu}}\exp\Big(C\log^{\beta+2\epsilon}(e+T)\mathcal{H}^{-1}\big(C(1+T)\log^{\beta+2\epsilon}(e+T)\big)\Big).
\end{align}
Applying the bound \eqref{eq:HH} to \eqref{eq:Rgrowth} yields the estimate \eqref{intro-est:reg-n}. Then, a direct application of \eqref{eq:z0bound}-\eqref{eq:zkbound}-\eqref{eq:zHolder} gives \eqref{intro-est:reg-z} for $n\geq2$, finishing the proof of Theorem~\ref{thm:gSQG-GR-whole-space}.

\subsection{\texorpdfstring{The estimation of $\partial_W^k u$}{The estimate of partial W k u}}

Following the approach developed by Radu \cite[Prop. 6.2]{Radu22},
we provide a regular representation formula for $\partial_W^k u$ analogous to 
\eqref{eq:desingularization_tangential_vector}.
For convenience, we introduce the following notation
\begin{align*}
  \delta^x_y[f] \triangleq f(x) - f(y), \quad \forall\, x,y\in \mathbb{R}^2.
\end{align*}
\begin{proposition}\label{prop:higher-tangential-derivative}
Let $D\subseteq \mathbb{R}^2$ be a smooth bounded domain, and let
$\varphi$ be a level-set characterization of domain $D$ 
and assume that \eqref{H1}-\eqref{H2a} hold. 
Let $W = (W_1,W_2) = \nabla^{\perp}\varphi$
be the vector field tangent to $\partial D$. 
Let $k\in \mathbb{N}^\star\cap[1,n]$, and suppose that 
$\partial_W^l W\in C^{\gamma}(\RR^2)$
for any $\gamma\in (0,1)$ and $0\leq l\leq k-2$ (if $k=1$, this condition is not necessary),
and $\partial_W^{k-1}W\in C^{\mu}(\RR^2)$, $\mu\in (0,1)$.
Let $u(x) = \nabla^\perp (-\Delta)^{-1} m(\Lambda) (\mathbf{1}_D)(x)$, then we have 
\begin{equation}\label{eq:block_higher_tangential_derivative}
  \partial_W^k u(x) = \sum_{j=1}^k \sum_{\substack{l_1\geq\cdots\geq l_j \geq 1 \\
  l_1+\cdots + l_j = k}} c^k_{l_1,\cdots, l_j}\,
  \mathrm{p.v.}\int_D \Big(\delta^x_y\big[\partial_W^{l_1-1}W\big]\otimes
  \cdots \otimes\delta^x_y\big[\partial_W^{l_j-1}W\big] \Big)
  \cdot\nabla^j K(x-y)\dd y,
\end{equation}
where $c^k_{l_1,\cdots, l_j}\in \mathbb{R}_+$,  $c_{k}^{k}=1$, $``\otimes"$ denotes the usual tensor product (e.g. $a \otimes b = (a_i b_j)_{2\times 2}$ 
for two vectors $a,b\in \mathbb{R}^2$), 
and $K(x)=\frac{x^{\perp}}{|x|^2} G(|x|)$ is given by \eqref{eq:u-exp1}.

\end{proposition}
\begin{remark}
In particular, for $k=2,3$, we have
\begin{align*}
  \partial_W^2u(x) & = \mathrm{p.v.} \int_D \big(W(x)-W(y)\big)\otimes \big(W(x) -W(y)\big) \cdot \nabla^2 K(x-y)\dd y \\
  & \quad + \mathrm{p.v.} \int_D \big(\partial_W W(x) -\partial_W W(y) \big) \cdot \nabla K(x-y)\dd y,
\end{align*}
and
\begin{align*}
  \partial_W^3 u(x) & = \mathrm{p.v.} \int_D \big(W(x)-W(y)\big)\otimes \big(W(x) -W(y)\big) \otimes \big(W(x)- W(y)\big) 
  \cdot \nabla^3 K(x-y) \dd y \\
  & \quad + 3 \, \mathrm{p.v.} \int_D \big(\partial_W W(x) -\partial_W W(y)\big) \otimes \big(W(x) - W(y)\big) \cdot 
  \nabla^2 K(x-y)\dd y \\
  & \quad + \mathrm{p.v.} \int_D \big(\partial_W^2 W(x) - \partial_W^2 W(y) \big) \cdot \nabla K(x-y) \dd y.
\end{align*}
\end{remark}
\begin{proof}[Proof of Proposition \ref{prop:higher-tangential-derivative}]
We prove \eqref{eq:block_higher_tangential_derivative} by induction. For $k=1$, \eqref{eq:block_higher_tangential_derivative} is a direct consequence of \eqref{eq:desingularization_tangential_vector} with $c_1^1 =1$.
Suppose that \eqref{eq:block_higher_tangential_derivative} is true for $k$.
We will show that it is also true for the $k+1$ case. The proof is analogous to Lemma \ref{lem:tangential_velocity}.
	
Due to the roughness of the patch solution,
we shall compute $\partial_W^{k+1} u$ in the distributional sense.
Since $W=\nabla^{\perp}\varphi$ is divergence free, we have
\begin{align*}
  \partial_W^{k+1} u
  = \textrm{div}(W\, \partial_W^k u).
\end{align*}
Next, for any $\widetilde{\chi}\in C_c^{\infty}(\RR^2)$, we get
\begin{align*}
  (\textrm{div}(W\, \partial_W^k u),\widetilde{\chi})
  = - (\partial_W^k u, W\cdot \nabla \widetilde{\chi})
  = -\int_{\RR^2} \partial_W^k u(x) \, W(x) \cdot \nabla \widetilde{\chi}(x)\dd x.
\end{align*}
Denoting by
\begin{align}\label{def:M_W-l1-lj}
  \mathcal{D}_W^{l_1,\cdots,l_j}(x,y) \triangleq 
  \delta^x_y\big[\partial_W^{l_1-1}W\big]\otimes \cdots \otimes \delta^x_y\big[\partial_W^{l_j-1}W\big], 
\end{align} 
and through the induction hypotheses \eqref{eq:block_higher_tangential_derivative} and Fubini's theorem, we deduce that
\begin{align*}
  & (\textrm{div}(W\, \partial_W^k u),\widetilde{\chi}) \\
  & = - \sum_{j=1}^k \sum_{\substack{l_1\geq\cdots\geq l_j \geq 1 \\
  l_1+\cdots + l_j = k}} c^k_{l_1,\cdots, l_j} 
  \int_{\RR^2} \Big(\lim_{\varepsilon \rightarrow 0} \int_{D\cap\{|x-y|\geq \varepsilon\}}
  \mathcal{D}_W^{l_1,\cdots,l_j}(x,y) \cdot \nabla^j K(x-y) \dd y\Big) 
  W(x)\cdot \nabla \widetilde{\chi}(x)\dd x \\
  & = - \sum_{j=1}^k \sum_{\substack{l_1\geq l_2\cdots\geq l_j \geq 1 \\
  l_1+\cdots + l_j = k}} c^k_{l_1,\cdots, l_j} 
  \lim_{\varepsilon\to 0}\int_{D} \int_{|x-y|\geq \varepsilon} 
  \mathcal{D}_W^{l_1,\cdots,l_j}(x,y) \cdot 
  \nabla^j K(x-y)\, W(x)\cdot \nabla \widetilde{\chi}(x)\dd x\dd y\\
  & \triangleq \sum_{j=1}^k \sum_{\substack{l_1\geq l_2\cdots\geq l_j \geq 1 \\
  l_1+\cdots + l_j = k}} c^k_{l_1,\cdots, l_j} \lim_{\varepsilon\to 0} \mathcal{A}_\varepsilon.
\end{align*}
Since $W$ is divergence free, the integration by parts gives us
\begin{align*}
  \mathcal{A}_\varepsilon & = -\int_D\int_{|x-y|\geq \varepsilon} \mathcal{D}_W^{l_1,\cdots,l_j}(x,y)  \cdot 
  \nabla^j K(x-y) \, W(x)\cdot \nabla \widetilde{\chi}(x)\dd x \dd y\\
  & = \int_D\int_{|x-y|\geq \varepsilon} W(x)\cdot \nabla_x\Big( \mathcal{D}_W^{l_1,\cdots,l_j}(x,y)  
  \cdot \nabla^j K(x-y) \Big) \widetilde{\chi}(x)\dd x \dd y + I_\varepsilon
\end{align*}
where
\begin{align*}
I_\varepsilon \triangleq \int_D\int_{|x-y|=\varepsilon} \mathcal{D}_W^{l_1,\cdots,l_j}(x,y)  \cdot \nabla^j K(x-y)\, W(x)\cdot \frac{(x-y)}{|x-y|} \widetilde{\chi}(x)\dd S(x) \dd y.	
\end{align*}
Define
\begin{align*}
J_\varepsilon \triangleq \int_D\int_{|x-y|=\varepsilon} \mathcal{D}_W^{l_1,\cdots,l_j}(x,y)  \cdot \nabla^j K(x-y)\, W(y)\cdot \frac{(x-y)}{|x-y|} \widetilde{\chi}(x)\dd S(x) \dd y,
\quad\text{and}\quad R_\varepsilon \triangleq I_\varepsilon - J_\varepsilon.
\end{align*}
Noting that for $s=1,\cdots,j$,
\begin{align}\label{eq:dxyW}
  \big|\delta^x_y [\partial_W^{l_s -1} W]\big| 
  = \big|\partial_W^{l_s-1}W(x) - \partial_W^{l_s -1} W(y)  \big|
  \leq \|\partial_W^{l_s-1} W\|_{C^\gamma} |x-y|^\gamma,
\end{align}
we control the difference $R_\varepsilon$ by
\begin{align*}
  \lim_{\varepsilon\to0}|R_\varepsilon|& = \lim_{\varepsilon\to0}\Big|\int_D\int_{|x-y|=\varepsilon} \mathcal{D}_W^{l_1,\cdots,l_j}(x,y)  
  \cdot \nabla^j K(x-y) (W(x)-W(y))\cdot \frac{(x-y)}{|x-y|}
  \widetilde{\chi}(x)\dd S(x) \dd y\Big| \\
  & \leq C\lim_{\varepsilon\to0} \int_D\prod_{s=1}^j \Big(\|\partial_W^{l_s-1} W \|_{C^\gamma}\varepsilon^\gamma\Big)\cdot \frac{m(\varepsilon^{-1})+1}{\varepsilon^{j+1}}\cdot \lVert W\rVert_{C^{\gamma}}\varepsilon^\gamma \cdot 1\cdot 
  \lVert \widetilde{\chi}\rVert_{L^\infty}\cdot 2\pi\varepsilon\,\dd y\\
  & \leq C \lim_{\varepsilon\to0} \varepsilon^{(j+1)\gamma-j}\big(m(\varepsilon^{-1})+1\big)=0,
\end{align*}
by choosing $\gamma\in (\frac{k}{k+1},1)$ so that $\gamma>\frac{k}{k+1}\geq \frac{j}{j+1}$ and therefore $(j+1)\gamma-j>0$.
We have used \eqref{eq:Kbound} to estimate $\nabla^j K$.

For $J_\varepsilon$, a similar argument as \eqref{eq:Jepsilon} yields
\begin{align*}
  J_{\varepsilon}  & = \int_{\RR^2}\int_{D\cap\{|x-y|=\varepsilon\}}  \mathcal{D}_W^{l_1,\cdots,l_j}(x,y) \cdot \nabla^j K(x-y) 
  W(y)\cdot \frac{(x-y)}{|x-y|} \widetilde{\chi}(x)\dd S(y)\dd x\\
  & =\int_{\RR^2}\int_{D\cap\{|x-y|\geq \varepsilon\}} W(y)\cdot \nabla_y \Big( \mathcal{D}_W^{l_1,\cdots,l_j}(x,y) \cdot \nabla^j K(x-y)\Big) \dd y\,\widetilde{\chi}(x)\dd x.
\end{align*}

Collecting the above computations yields that
\begin{align*}
  \mathcal{A}_\varepsilon & = \int_{\RR^2}\int_{D\cap\{|x-y|\geq \varepsilon\}} W(x) \cdot 
  \nabla_x \Big( \mathcal{D}_W^{l_1,\cdots,l_j}(x,y) \cdot \nabla^j K(x-y)\Big) \dd y\,\widetilde{\chi}(x)\dd x\\
  & \quad + \int_{\RR^2}\int_{D\cap\{|x-y|\geq \varepsilon\}} W(y) \cdot 
  \nabla_y \Big( \mathcal{D}_W^{l_1,\cdots,l_j}(x,y) \cdot \nabla^j K(x-y)\Big) \dd y\,\widetilde{\chi}(x)\dd x +R_\varepsilon\\
  & = \int_{\RR^2}\int_{D\cap\{|x-y|\geq \varepsilon\}}\mathcal{D}_W^{l_1,\cdots,l_j}(x,y) 
  \otimes \big(W(x)-W(y)\big)\cdot \nabla^{j+1}K(x-y) \dd y\,\widetilde{\chi}(x)\dd x\\
  & \quad + \int_{\RR^2}\int_{D\cap\{|x-y|\geq \varepsilon\}} \sum_{s=1}^j \mathcal{D}_W^{l_1,\cdots,l_s+1,\cdots,l_j}(x,y)\cdot 
  \nabla^j K(x-y) \dd y\,\widetilde{\chi}(x)\dd x + R_\varepsilon.
\end{align*}
Observe that all the terms have the form 
$\mathcal{D}_W^{\widetilde{l}_1,\cdots,\widetilde{l}_{\widetilde{j}}}(x,y)\cdot \nabla^{\widetilde{j}} K(x-y)$, 
with $\widetilde{j}=j$ or $j+1$ so that $\widetilde{j}\in\{1,\cdots,k+1\}$, and $\widetilde{l}_1+\cdots+\widetilde{l}_{\widetilde{j}}=k+1$. 
Therefore, we obtain 
\begin{align*}
  & (\textrm{div}(W\, \partial_W^k u),\widetilde{\chi}) = \lim_{\varepsilon\to 0}\sum_{j=1}^k \sum_{\substack{l_1\geq\cdots\geq l_j \geq 1 \\
  l_1+\cdots + l_j = k}} c^k_{l_1,\cdots, l_j} \mathcal{A}_\varepsilon\\
  & = \int_{\RR^2}\bigg(\sum_{j=1}^{k+1} \sum_{\substack{l_1\geq \cdots\geq l_j \geq 1 \\
  l_1+\cdots + l_j = k+1}} c^{k+1}_{l_1,\cdots, l_j}\,
  \mathrm{p.v.}\int_D \mathcal{D}_W^{l_1,\cdots,l_j}(x,y) \cdot \nabla^j K(x-y)\dd y\bigg) \widetilde{\chi}(x)\dd x,
\end{align*}
if we define the coefficients $c_{l_1,\cdots,l_j}^{k+1}$ ($j=1,\cdots,k+1$) properly by
\begin{align*}
 c_{l_1,\cdots,l_j}^{k+1} = \sum_{\substack{s=1\\ l_s\geq2}}^j c_{l_1,\cdots,l_s-1,\cdots,l_j}^{k} + \mathbf{1}_{\{l_j=1\}}\, c_{l_1,\cdots,l_{j-1}}^k.
\end{align*}
Recalling the notation \eqref{def:M_W-l1-lj}, we conclude with \eqref{eq:block_higher_tangential_derivative} for $k+1$, as desired.
\end{proof}

Given the expression of $\partial_W^k u$ in \eqref{eq:block_higher_tangential_derivative}, 
the leading term requiring the highest regularity of $W$ corresponds to the case $j=1$ and $l_1 = k$, which takes the form
\begin{align*}
  \mathbf{G}_k(x)\triangleq \mathrm{p.v.} \int_D 
  \big(\partial_W^{k-1}W(x) -\partial_W^{k-1} W(y) \big) \cdot\nabla K(x-y) \dd y. 
\end{align*}
We obtain the following controls on $\mathbf{G}_k$ as well as the remaining lower order terms.

\begin{lemma}\label{lem:higher_tangential_est}
Let $k\in\mathbb{N}^\star\cap [2,n]$. 
Under the assumptions of Proposition \ref{prop:higher-tangential-derivative},
we have that
\begin{equation}\label{eq:tangential_higher_est_1}
  \lVert \mathbf{G}_k\rVert_{C_\mathrm{m}^{\mu}(\RR^2)}
  \leq C(1+\log \mathbf{\Delta}_{\gamma}) 
   \lVert \partial_W^{k-1} W \rVert_{\dot{C}^{\mu}(\RR^2)}, 
\end{equation}
and
\begin{equation}\label{eq:tangential_higher_est-2}
  \lVert \partial_W^k u-\mathbf{G}_k \rVert_{C^{\mu}(\RR^2)}
  \leq C\sum_{j=2}^k \sum_{\substack{l_1\geq\cdots\geq l_j \geq 1 \\
  l_1 + \cdots + l_j = k}} \Big(\prod_{s=1}^j 
  \|\partial_W^{l_s-1} W \|_{\dot C^{\gamma_k}}\Big),
\end{equation}
for any $\mu\in (0,1)$, $\gamma\in (0,1)$, and $\gamma_k \in (1-\frac{1-\mu}{k},1)$.
\end{lemma}

\begin{proof}
The proof of \eqref{eq:tangential_higher_est_1} is analogous to that of Lemma \ref{lem:striated_est_mu}, replacing $W$ on the right-hand side of \eqref{eq:part-W-u-Hold} by $\partial_W^{k-1} W$.

To prove \eqref{eq:tangential_higher_est-2}, we also follow a similar procedure as in Lemma \ref{lem:striated_est_mu}. Recall that
\begin{align*}
  \partial_W^k u(x)- \mathbf{G}_k(x) 
  = \sum_{j=2}^k \sum_{\substack{l_1\geq\cdots\geq l_j \geq 1 \\
  l_1+\cdots + l_j = k}} c^k_{l_1,\cdots, l_j}\,\Psi_{l_1,\cdots,l_j}(x)  ,
\end{align*}
where we denote by
\begin{align*}
  \Psi_{l_1,\cdots,l_j}(x) & \triangleq 
  \mathrm{p.v.}\int_D \mathcal{D}_W^{l_1,\cdots,l_j}(x,y)
  \cdot\nabla^j K(x-y)\dd y \\
  & =  \mathrm{p.v.}\int_D \Big(\delta^x_y\big[\partial_W^{l_1-1}W\big]\otimes
  \cdots \otimes\delta^x_y\big[\partial_W^{l_j-1}W\big] \Big)
  \cdot\nabla^j K(x-y)\dd y.
\end{align*}
By virtue of \eqref{eq:dxyW}, \eqref{eq:Kbound}, and \eqref{eq:m-prop2}, we have the $L^\infty$ bound:
\begin{align*}
  \|\partial_W^k u -\mathbf{G}_k\|_{L^\infty} 
  & \leq C\sum_{j=2}^k \sum_{\substack{l_1\geq\cdots\geq l_j \geq 1 \\
  l_1 + \cdots + l_j = k}}  
  \Big(\prod_{s=1}^j \|\partial_W^{l_s-1} W \|_{\dot C^{\gamma_k}}\Big) 
  \int_{D}|x-y|^{-j-1+j\gamma_k } \big(m(|x-y|^{-1}) + 1 \big)\dd y \\ 
  &\leq  C\sum_{j=2}^k \sum_{\substack{l_1\geq\cdots\geq l_j \geq 1 \\
  l_1 + \cdots + l_j = k}}  
  \Big(\prod_{s=1}^j \|\partial_W^{l_s-1} W \|_{\dot C^{\gamma_k}}\Big) 
  \int_0^L r^{-j+j\gamma_k } \big(m(r^{-1}) +1 \big)\dd y \\
  & \leq  C\sum_{j=2}^k \sum_{\substack{l_1\geq\cdots\geq l_j \geq 1 \\
  l_1 + \cdots + l_j = k}} 
  \Big(\prod_{s=1}^j \|\partial_W^{l_s-1} W \|_{\dot C^{\gamma_k}}\Big),
\end{align*}
Choosing $x,h\in \RR^2$ such that $|h|\le\frac{\bar{c}_0}{3}$ (recalling that $\bar{c}_0>0$ is the constant appearing in Lemma \ref{lem:G-prop}), we write
\begin{align*}
 \Psi_{l_1,\cdots,l_j}(x) - & \Psi_{l_1,\cdots,l_j}(x+h) = \mathfrak{J}_1 + \mathfrak{J}_2 + \mathfrak{J}_3 + \mathfrak{J}_4,	
\end{align*}
where
\begin{align*}
  \mathfrak{J}_1 & \triangleq \mathrm{p.v.}\int_{D\cap\{|x-y|<2|h|\}}\mathcal{D}_W^{l_1,\cdots,l_j}(x,y)\cdot\nabla^j K(x-y)\dd y, \\
  \mathfrak{J}_2 & \triangleq - \mathrm{p.v.}\int_{D\cap\{|x-y|<2|h|\}}\mathcal{D}_W^{l_1,\cdots,l_j}(x+h,y)\cdot\nabla^{j}K(x+h-y)\dd y,\\
  \mathfrak{J}_3 & \triangleq \int_{D\cap\{|x-y|\geq 2|h|\}}\Big(\mathcal{D}_W^{l_1,\cdots,l_j}(x,y) - \mathcal{D}_W^{l_1,\cdots,l_j}(x+h,y)\Big)\cdot\nabla^j K(x-y)\dd y, \\
  \mathfrak{J}_4 & \triangleq \int_{D\cap\{|x-y|\geq 2|h|\}}\mathcal{D}_W^{l_1,\cdots,l_j}(x+h,y) \cdot
  \Big(\nabla^j K(x-y)-\nabla^j K(x+h-y)\Big)\dd y.
\end{align*}
For $\mathfrak{J}_1$, from \eqref{eq:dxyW}, \eqref{eq:Kbound} and \eqref{eq:approx-m-int-1},  we get
\begin{align*}
  |\mathfrak{J}_1| & \leq C\int_{|x-y|\leq 2|h|}\Big(\prod_{s=1}^j \lVert \partial_W^{l_s-1} W\rVert_{\dot{C}^{\gamma_k}}|x-y|^{\gamma_k}\Big)\frac{m(|x-y|^{-1})}{|x-y|^{j+1}}\dd y\\
  & \leq  C\Big(\prod_{s=1}^j \lVert \partial_W^{l_s-1} W\rVert_{\dot{C}^{\gamma_k}}\Big)\int_0^{2|h|}\rho^{j\gamma_k-j}m(\rho^{-1})\dd\rho
  \leq C\Big(\prod_{s=1}^j \lVert \partial_W^{l_s-1} W\rVert_{\dot{C}^{\gamma_k}}\Big)|h|^{j\gamma_k-j+1}m(|h|^{-1})\\
 & \leq C\Big(\prod_{s=1}^j \lVert \partial_W^{l_s-1} W\rVert_{\dot{C}^{\gamma_k}}\Big)|h|^{\mu},
\end{align*}
where for the last inequality, we have used the fact that $j\gamma_k-j+1>\mu$, which follows from assumptions $\gamma_k>1-\frac{1-\mu}{k}$ and $j\leq k$. A similar argument yields
\begin{align*}
  |\mathfrak{J}_2| \leq C\Big(\prod_{s=1}^j \lVert \partial_W^{l_s-1} W\rVert_{\dot{C}^{\gamma_k}}\Big) \int_0^{3|h|}\rho^{j\gamma_k-j}m(\rho^{-1})\dd\rho
  \leq C\Big(\prod_{s=1}^j \lVert \partial_W^{l_s-1} W\rVert_{\dot{C}^{\gamma_k}}\Big)|h|^{\mu}.
\end{align*}
For the term $\mathfrak{J}_3$, since $|x-y|\geq 2h$, we have $|x+h-y|\leq |x-y|+|h|\leq \frac32|x-y|$. We apply \eqref{eq:dxyW} and obtain
\begin{align*}
  & |\mathcal{D}_W^{l_1,\cdots,l_j}(x,y) - \mathcal{D}_W^{l_1,\cdots,l_j}(x+h,y)|\\
  & = \Big|\sum_{s=1}^j\Big(\delta^{x+h}_y[\partial_W^{l_1-1}W]\otimes \cdots \otimes \big(\delta^x_y[\partial_W^{l_s-1}W]-\delta^{x+h}_y[\partial_W^{l_s-1}W]\big) \otimes \cdots \otimes\delta^x_y[\partial_W^{l_j-1}W]\Big)\Big|\\
  & \leq C\Big(\prod_{s=1}^j \lVert \partial_W^{l_s-1} W\rVert_{\dot{C}^{\gamma_k}}\Big) |h|^{\gamma_k}|x-y|^{(j-1)\gamma_k}.
\end{align*}
Then, it follows from \eqref{eq:Kbound} and analogous estimates for $\mathfrak{J}_1$ that
\begin{align*}
  |\mathfrak{J}_3| & \leq C\Big(\prod_{s=1}^j \lVert \partial_W^{l_s-1} W\rVert_{\dot{C}^{\gamma_k}}\Big) |h|^{\gamma_k}\mathrm{p.v.}\int_{D\cap\{|x-y|\geq 2|h|\}}|x-y|^{(j-1)\gamma_k}\frac{m(|x-y|^{-1})+1}{|x-y|^{j+1}}\dd y\\
   &\leq  C\Big(\prod_{s=1}^j \lVert \partial_W^{l_s-1} W\rVert_{\dot{C}^{\gamma_k}}\Big) |h|^{\gamma_k}\int_{2|h|}^{L}\rho^{(j-1)\gamma_k-j}\big(m(\rho^{-1})+1\big)\,\dd\rho\\
   &\leq C\Big(\prod_{s=1}^j \lVert \partial_W^{l_s-1} W\rVert_{\dot{C}^{\gamma_k}}\Big)|h|^{j\gamma_k-j+1}\big(m(|h|^{-1})+1\big)
   \leq C\Big(\prod_{s=1}^j \lVert \partial_W^{l_s-1} W\rVert_{\dot{C}^{\gamma_k}}\Big)|h|^{\mu}.
\end{align*}
Using the mean value theorem and \eqref{eq:Kbound}, the term $\mathfrak{J}_4$ can be bounded as
\begin{align*}
  |\mathfrak{J}_4| & \leq C\int_{D\cap\{|x-y|\geq 2|h|\}}\Big(\prod_{s=1}^j \lVert \partial_W^{l_s-1} W\rVert_{\dot{C}^{\gamma_k}}\big(\tfrac32|x-y|\big)^{\gamma_k}\Big)\,|h|\,\frac{m\big((\tfrac12|x-y|)^{-1}\big)+1}{\big(\tfrac12|x-y|\big)^{j+2}}\,\dd y \\
  & \leq C\Big(\prod_{s=1}^j \lVert \partial_W^{l_s-1} W\rVert_{\dot{C}^{\gamma_k}}\Big)|h|\int_{2|h|}^{L}\rho^{j\gamma_k-j-1}\big(m(\rho^{-1})+1\big)\,\dd \rho\\
  &\leq C\Big(\prod_{s=1}^j \lVert \partial_W^{l_s-1} W\rVert_{\dot{C}^{\gamma_k}}\Big)|h|^{j\gamma_k-j+1}\big(m(|h|^{-1})+1\big)
   \leq C\Big(\prod_{s=1}^j \lVert \partial_W^{l_s-1} W\rVert_{\dot{C}^{\gamma_k}}\Big)|h|^{\mu}.
\end{align*}
Gathering the above estimates yields
\begin{align*}
  \lVert \Psi_{l_1,\cdots,l_j}\rVert_{\dot{C}^{\mu}(\mathbb{R}^2)} 
  \leq C \prod_{s=1}^j \Big(\|\partial_W^{l_s-1} W \|_{\dot C^{\gamma_k}}\Big),
\end{align*}
and the desired estimate \eqref{eq:tangential_higher_est-2} follows.
\end{proof}

\subsection{Proof of Theorem \ref{thm:propagation_regularity}}
We prove \eqref{eq:tangential_regularity}-\eqref{eq:Rgrowth} by induction. The case $n=1$ has been proved in Theorem \ref{thm:regularityCmu}, where the bound is given by \eqref{eq:gradphiCmu} with $\mu'=\mu-\varepsilon$.

Assume that \eqref{eq:tangential_regularity}-\eqref{eq:Rgrowth} holds for any $n\in \{1,2,\cdots,k\}$. 
We will show that it holds for $n=k+1$, that is, if $\|\partial_{W_0}^k W_0\|_{C^{\mu}(\mathbb{R}^2)}<+\infty$, then 
\begin{align}\label{eq:induct-k-conc}
  \|\partial_W^k W(t)\|_{C^{\mu'}(\mathbb{R}^2)} \leq R_{k+1,\mu,\mu'}(t)<+\infty,
\end{align}
where $R_{k+1,\mu,\mu'}(t)$ satisfies \eqref{eq:Rgrowth}.

Recalling that $\Phi_{t,s}(x)$ is the flow map given by \eqref{eq:flow_map_def}, 
and from the equation \eqref{eq:parWk-1W}, we have 
\begin{align*}
  \partial_W^k W(x,t)  = \partial_{W_0}^k W_0(\Phi_{0,t}(x))
  + \int_0^t\partial_W^{k+1} u(s,\Phi_{s,t}(x))\dd s.
\end{align*}
Let $\varepsilon = \mu-\mu'$, $\gamma=\mu'$. Fix $T>0$, and define $V_\gamma(t)$, $\eta_{\gamma,T}$ and $\mu(t)$ as in \eqref{eq:V-t-def}. Consequently,
\begin{equation*}
\begin{aligned}
  |\partial_W^k W(x,t) - \partial_W^k W(y,t)|
  & \leq \lVert \partial_{W_0}^k W_0\rVert_{\dot{C}^\mu}|\Phi_{0,t}(x)-\Phi_{0,t}(y)|^{\mu} \\
  & \quad +\int_0^t\Big|\partial_W^{k+1}u(\Phi_{s,t}(x),s)-\partial_W^{k+1}u(\Phi_{s,t}(y),s)\Big|\dd s.
\end{aligned}
\end{equation*}
It follows from Proposition \ref{prop:higher-tangential-derivative} 
and Lemma \ref{lem:higher_tangential_est} that for every $\gamma_k\in (1- \frac{1-\mu}{k+1},1)$,
\begin{align}\label{eq:ukdiff}
  \Big|\partial_W^{k+1} u(x,s)-\partial_W^{k+1} u(y,s)\Big|
  & \leq C \big(1 + \log \mathbf{\Delta}_{\gamma}(s)\big) \|\partial_W^k W(s)\|_{\dot C^{\mu(s)}} |x-y|^{\mu(s)}\big(m(|x-y|^{-1})+1\big) \nonumber  \\
  & \quad +C\sum_{j=2}^{k+1} \sum_{\substack{l_1\geq\cdots\geq l_j \geq 1 \\ l_1 + \cdots + l_j = k+1}} \Big(\prod_{i=1}^j\|\partial_W^{l_i-1} W(s) \|_{\dot C^{\gamma_k}}\Big) |x-y|^\mu.
\end{align}
Note that the first term can be handled analogously to the argument in Theorem~\ref{thm:regularityCmu} (in the case $n=1$), while the second term is controlled by the induction hypothesis \eqref{eq:induct-k-conc}:
\begin{align*}
 \sum_{j=2}^{k+1} \sum_{\substack{l_1\geq\cdots\geq l_j \geq 1 \\ l_1 + \cdots + l_j = k+1}} \Big(\prod_{i=1}^j\|\partial_W^{l_i-1} W(t) \|_{\dot C^{\gamma_k}}\Big)\leq \sum_{j=2}^{k+1} \sum_{\substack{l_1\geq\cdots\geq l_j \geq 1 \\ l_1 + \cdots + l_j = k+1}} \Big(\prod_{i=1}^j R_{l_i,\frac{\gamma_k+1}{2},\gamma_k}(t)\Big) \triangleq \mathsf{R}_{k+1,\mu}(t)<+\infty.
\end{align*}
Indeed, $R_{l_i,\frac{\gamma_k+1}{2},\gamma_k}(t)<+\infty$ due to the fact $l_i\leq k$ (since $l_i\geq1$, $l_1 + \cdots + l_j = k+1$, and $j\geq2$), and consequently $C^{k+1,\mu}\subset C^{l_i, \frac{\gamma_k+1}{2}}$.
Applying \eqref{eq:est-diff-Phi-s-t} and \eqref{eq:est-diff-Mphi-t-s}, we obtain the following bound analogously to \eqref{eq:est-main-ineq}:
\begin{align*}
  & \frac{|\partial_W^k W(x,t)-\partial_W^k W(y,t)|}{|x-y|^{\mu(t)}}
  \leq C \|\partial^k_{W_0} W_0\|_{\dot{C}^{\mu}} |x-y|^{\mu-\mu(t)} e^{\mu \,\mathfrak{M}_{0,t,\gamma}(|x-y|^{-1})}\\
  & \quad + C_1 \int_0^t(1+\log \mathbf{\Delta}_{\gamma}(s))
  \lVert \partial_W^k W(s)\rVert_{\dot{C}^{\mu(s)}}|x-y|^{\mu(s)-\mu(t)}\big(m(|x-y|^{-1})+1\big) e^{\mu\, \mathfrak{M}_{s,t,\gamma}(|x-y|^{-1})} \dd s\\
  & \quad + C |x-y|^{\mu-\mu(t)} \int_0^t \mathsf{R}_{k+1,\mu}(s)e^{\mu \,\mathfrak{M}_{s,t,\gamma}(|x-y|^{-1})} \dd s \\
  & \leq C \Big(\|\partial^k_{W_0} W_0\|_{\dot{C}^{\mu}}+\int_0^T\mathsf{R}_{k+1,\mu}(t)\,\dd t\Big) |x-y|^{\mu-\mu(t)} e^{\mu \,\mathfrak{M}_{0,t,\gamma}(|x-y|^{-1})}\\
   & \quad + C_1 \int_0^t(1+\log \mathbf{\Delta}_{\gamma}(s))
  \lVert \partial_W^k W(s)\rVert_{\dot{C}^{\mu(s)}}|x-y|^{\mu(s)-\mu(t)}\big(m(|x-y|^{-1})+1\big) e^{\mu\, \mathfrak{M}_{s,t,\gamma}(|x-y|^{-1})} \dd s. 
\end{align*}
Then, through the same argument as in the proof of Theorem \ref{thm:regularityCmu}, we obtain a bound analogously to \eqref{eq:gradphiCmu}:
\begin{align*}
  &\lVert \partial_W^k W(T)\rVert_{C^{\mu'}} = \lVert \partial_W^k W(T)\rVert_{C^{\mu-\varepsilon}}\\
  & \leq C \Big(\|\partial^k_{W_0} W_0\|_{C^{\mu}}+\int_0^T\mathsf{R}_{k+1,\mu}(t)\,\dd t\Big)\exp\Big(C\log^{\beta+2\epsilon}(e+T)\mathcal{H}^{-1}\big(C(1+T)\log^{\beta+2\epsilon}(e+T)\big)\Big)\\
  & \triangleq R_{k+1,\mu',\mu}(T)<+\infty.
\end{align*}
This finishes the proof of \eqref{eq:induct-k-conc}. By induction, \eqref{eq:tangential_regularity} holds.

We are left to verify that $R_{k+1,\mu',\mu}(t)$ satisfies \eqref{eq:Rgrowth}. 
By the induction hypothesis, we have
\begin{align*}
 \mathsf{R}_{k+1,\mu}(t)\leq C\lVert\varphi_0\rVert_{C^{k+1,\mu}}\exp\Big(C\log^{\beta+2\epsilon}(e+t)\mathcal{H}^{-1}\big(C(1+t)\log^{\beta+2\epsilon}(e+t)\big)\Big). 	
\end{align*}
Denote
\begin{align*}
 g(t) \triangleq C\log^{\beta+2\epsilon}(e+t)\mathcal{H}^{-1}\big(C(1+t)\log^{\beta+2\epsilon}(e+t)\big).	
\end{align*}
It is easy to check that $g$ is a strictly increasing function in $[0,+\infty)$ with $\min_{t\geq0}g'(t) = c>0$. Then,
\begin{align*}
 \int_0^T\mathsf{R}_{k+1,\mu}(t)\,\dd t \leq \int_0^T C\lVert\varphi_0\rVert_{C^{k+1,\mu}}e^{g(t)}\,\dd t \leq \frac{C\lVert\varphi_0\rVert_{C^{k+1,\mu}}}{c}\int_0^T g'(t) e^{g(t)}\,\dd t \leq C\lVert\varphi_0\rVert_{C^{k+1,\mu}} e^{g(T)}.
\end{align*}
Therefore, we find that
\begin{align*}
 R_{k+1,\mu',\mu}(T) 
 & \leq C\big(\lVert\varphi_0\rVert_{C^{k+1,\mu}} + C\lVert\varphi_0\rVert_{C^{k+1,\mu}}e^{g(T)})\,e^{g(T)}\\
 & \leq C\lVert\varphi_0\rVert_{C^{k+1,\mu}}\exp\Big(C\log^{\beta+2\epsilon}(e+t)\mathcal{H}^{-1}\big(C(1+t)\log^{\beta+2\epsilon}(e+t)\big)\Big),
\end{align*}
finishing the proof.

\section{Global regularity of multiple vortex patches}\label{sec:mul-patch}

In this section, we consider the patch solutions that are composed of multiple patches.

Suppose that $\omega_0(x)$ takes the form of \eqref{eq:patch-data}, where $N>1$ and $D_i(0)\subset\mathbb{R}^2$ for $i\in\{1,\cdots,N\}$ are simply connected disjoint bounded domains with 
\begin{align*}
 d_0\triangleq \min_{i\neq j}\mathrm{dist}(D_i(0),D_j(0))>0.	
\end{align*}
Then, according to Theorem \ref{thm:Yudovich_type}, the 2D Loglog-Euler type equation \eqref{m-SQG} admits a unique global vortex patch solution
\begin{equation}\label{eq:initial_data_general}
  \omega(x,t) = \sum_{j=1}^N a_j \mathbf{1}_{D_j(t)}(x),\quad a_j\in \mathbb{R},
  \quad \textrm{with}\;\; D_j(t) \triangleq \Phi_t(D_j(0)),
\end{equation} 
where $\Phi_t(\cdot) = \Phi_{t,0}(\cdot)$ is the flow map defined by \eqref{eq:flow_map_def}.

Suppose that $\partial D_i(0)\in C^{n,\mu}$, $i\in \{1,\cdots,N\}$. 
Let $\varphi_{i,0}\in C^{n,\mu}$ be a level-set characterization of $D_i(0)$ with a compact support $\Omega_i\triangleq \overline{\{x\in \RR^2: \varphi_{i,0}(x) \ne 0\}}$ slightly larger than $D_i(0)$, such that  
\begin{align*}
  D_i(0)\subset \Omega_i,\quad \mathrm{dist}\,(\partial \Omega_i,\partial D_i(0)) \geq \tfrac{d_0}{3},\quad \text{and}\quad \mathsf{d}_i(0) \triangleq \max_{j\neq i} \mathrm{dist}\,(\Omega_i,D_j(0))\geq \tfrac{d_0}{2}. 
\end{align*}
Denote by 
\begin{align*}
  \varphi_i(x,t) = \varphi_{i,0}(\Phi_{0,t}(x)), \quad \textrm{and} \quad W_i(x,t) \triangleq \nabla^{\perp}\varphi_i(x,t). 
\end{align*}
From the framework discussed in Section \ref{subsec:formulation}, we know that $\varphi_i(\cdot,t)$ is a level-set characterization of the domain $D_i(t)$, and that the boundary regularity $\partial D_i(t)\in C^{n,\mu}$ follows from the control of $\lVert \partial_{W_i}^{n-1}W_i(t) \rVert_{C^{\mu}(\RR^2)}$. Furthermore, we have
\begin{align*}
 \mathrm{supp}\,\varphi_i(\cdot,t)=\Omega_i(t) \triangleq \Phi_t(\Omega_i).	
\end{align*}
and 
\begin{align*}
  \lVert \partial_{W_i}^l u(t)\rVert_{C^{\mu-\varepsilon}(\RR^2)} 
  = \lVert \partial_{W_i}^lu(t)\rVert_{C^{\mu-\varepsilon}(\Omega_i(t))},
  \qquad \lVert \partial_{W_i}^{l-1}W_i(t)\rVert_{C^{\mu-\varepsilon}(\RR^2)} 
  = \lVert \partial_{W_i}^{l-1}W_i(t)\rVert_{C^{\mu-\varepsilon}(\Omega_i(t))}. 
\end{align*}
Define the distance 
\begin{align*}
  \mathsf{d}_i(t)\triangleq \min_{j\ne i}\mathrm{dist}\,(\Omega_i(t),D_j(t)).
\end{align*}
Applying the estimate \eqref{eq:bound-flow-map-general} on the flow map, we obtain
\begin{align}\label{eq:est-mini-dist-t}
  \mathsf{d}_i(t)\geq \frac{1}{\mathsf{H}^{-1}\big(\mathsf{H}(\mathsf{d}_i(0)^{-1})+ Ct\big)} \geq \frac{1}{\mathsf{H}^{-1}\big(\mathsf{H}(\tfrac{2}{d_0})+ Ct\big)}>0,\quad \forall\, t\geq0.
\end{align}
Hence, different patches remain separated for all time.

Let us also comment on the velocity field $u$, which takes the form
\begin{align}\label{eq:u-decompose}
 u(x,t) = \sum_{j=1}^N u_j(x,t) \triangleq \sum_{j=1}^N \nabla^\perp (-\Delta)^{-1}m(\Lambda)\big(a_j\mathbf{1}_{D_j(t)}(x)\big).
\end{align}
Since the patches are separated, the main contribution to $u(x,t)$ for $x\in\Omega_i(t)$ comes from the $i$-th patch $u_i(x,t)$. Indeed, for $j\neq i$, we apply \eqref{eq:Kbound} and control $u_j(x,t)$ by
\begin{align}\label{eq:smooth-pertb}
 \sum_{j\ne i}\lVert \nabla^l u_j(t)\rVert_{L^\infty(\Omega_i(t))}
 &\leq \sum_{j\neq i} |a_j|\, \Big\| \int_{D_j(t)} \nabla^l K(x-y) \dd y\Big\|_{L^\infty_x(\Omega_i(t))}\nonumber\\
 &\leq C\,\frac{m(\mathsf{d}_i(t)^{-1})+1}{\mathsf{d}_i(t)^{l+1}}
 \leq C \big(\mathsf{d}_i(t)^{-l-2}+1\big),
\end{align}
for $l\in\{0,1,\cdots,n+1\}$ where we do not intend to find the optimized exponent in the last inequality. Therefore, the global regularity results for a single vortex patch extend to the case of multiple vortex patches.

\begin{theorem}\label{thm:multi-patch}
Let $\mu \in (0,1)$ and $n \in \mathbb{N}^\star$.
Assume that $m(\xi) = m(|\xi|)$ is a radial function on $\mathbb{R}^2$, where $m(r)$ satisfies \eqref{H1}-\eqref{H2a}-\eqref{eq:Osgood-cond}.
For each patch solution $\omega$ given by \eqref{eq:initial_data_general}, if $\partial D_i(0)\in C^{n,\mu}$, then the patch boundaries $\partial D_j(t) \in C^{n,\mu-\varepsilon}$ for any $\varepsilon \in (0,\mu)$ and $j\in \{1,\cdots,N\}$. 
	
Moreover, for any given $\varepsilon>0$ and $\epsilon>0$, there are some constant $C>0$ 
depending on initial data, $\varepsilon$ and $\epsilon$ such that  
for all $i=\{1,2,\cdots, N\}$, the following estimates hold: 
\begin{align}\label{ineq:multi-est-1}
  |\nabla \varphi_i(t)|_{\inf} \geq |\nabla \varphi_{i,0}|_{\inf} \exp\Big(-C\mathcal{H}^{-1}\big(C (1+t)\log^{\beta+\epsilon}(e+t) + C\mathcal{H}(\mathsf{g}(t) )\big) \Big),
\end{align}
and
\begin{align}\label{ineq:multi-est-2}
  \|\partial_{W_i}^{n-1} W_i(t)\|_{C^{\mu-\varepsilon}} \leq  C \|\nabla \varphi_{i,0}\|_{C^\mu}
  \exp\Big(C\log^{\beta+\epsilon}(e+t)\mathcal{H}^{-1}\big(C (1+t)\log^{\beta+\epsilon}(e+t) +C \mathcal{H}(\mathsf{g}(t)) \big)\Big),
\end{align}
where $\varphi_i(\cdot)$ is a level-set characterization of $D_i(t)$, $W_i(\cdot)\triangleq \nabla^{\perp}\varphi_i(\cdot)$, the mapping $\mathcal{H}(\cdot)$ is given by \eqref{def:Hr}, $\mathsf{g}(t)$ is defined by
\begin{align}\label{def:gt}
  \mathsf{g}(t)\triangleq \big(\mathsf{H}^{-1}(\mathsf{H}(\tfrac{2}{d_0})+Ct)\big)^{3}.
\end{align}
\end{theorem}

\begin{remark}\label{rmk:multi}
The main difference of the bounds \eqref{ineq:multi-est-1}-\eqref{ineq:multi-est-2} 
compared to the single-patch case \eqref{eq:gradphiinf}-\eqref{eq:gradphiCmu} 
is the appearance of the term $\mathcal{H}(\mathsf{g}(t))$, which grows exponentially (or faster) in time. 
This growth arises because patches may approach one another as time evolves. 
As a consequence, the regularity bounds grow triple-exponentially for the 2D Euler equation 
(i.e. $\beta=\varepsilon=\epsilon =0$), and potentially faster in the general case.
\end{remark}

\begin{proof}[Proof of Theorem \ref{thm:multi-patch}]
We start by proving this result for the case $n=1$. Define 
\begin{align*}
  \mathbf{\Delta}_{i,\gamma}(t)\triangleq \frac{\lVert \nabla \varphi_i(t)
  \rVert_{\dot{C}^{\gamma}(\Omega_i(t))}}{\lVert \nabla\varphi_i(t)\rVert_{\inf (\partial D_i(t))}}+1,\quad \gamma\in (0,1). 
\end{align*}
To estimate $u$, we use the decomposition \eqref{eq:u-decompose}. The control of $u_i(t)$ on $\Omega_i(t)$ is the same as the single patch case, see \eqref{eq:u-patch-imp-est}, \eqref{eq:part-W-u-Hold} and \eqref{es:nab-u-tang}; while the control of $u_j(t)$ on $\Omega_i(t)$ follows from \eqref{eq:smooth-pertb}. We deduce
\begin{align*}
  \lVert u(t) \rVert_{C^1_{\mathrm{m}}(\Omega_i(t))} 
  &\leq C\big(1+\log \mathbf{\Delta}_{i,\gamma}(t) + \mathsf{d}_i(t)^{-3}\big),\\   \lVert \partial_{W_i} u\rVert_{C_{\mathrm{m}}^{\sigma}(\Omega_i(t))}
  &\leq C\big(1+\log \mathbf{\Delta}_{i,\gamma}(t) + \mathsf{d}_i(t)^{-3}\big) 
  \lVert W_i(t)\rVert_{C^{\sigma}(\Omega_i(t))},\quad \sigma\in(0,1),\\ 
  \big\lVert \nabla u\, \tfrac{W_i}{|W_i|} \cdot \tfrac{W_i}{|W_i|}\big\rVert_{L^\infty(\Omega_i(t))} 
  &\leq C\big(m(\mathbf{\Delta}_{i,\gamma}(t))+1\big)\big(1+\log \mathbf{\Delta}_{i,\gamma}(t)+\mathsf{d}_i(t)^{-3}\big).
\end{align*}
The rest of the proof follows from the same procedure in Section \ref{subsec:main-proof-1}, replacing $1+\log \mathbf{\Delta}_{\gamma}(\cdot )$ with $1+\log_{+}\mathbf{\Delta}_{i,\gamma}(\cdot)+ \mathsf{d}_i(\cdot)^{-3}$. 
In particular, we deduce the bound on $\mathbf{\Delta}_{i,\gamma}(t)$ analogous to \eqref{eq:Deltabound}: 
\begin{align*}
  \mathbf{\Delta}_{i,\mu-\varepsilon}(t)
  \leq C\exp\bigg(C\big(m(e^{V_{i,\mu-\varepsilon}(t)})+1\big)V_{i,\mu-\epsilon}(t) 
  + C \int_0^t \big(m(\mathbf{\Delta}_{i,\mu-\varepsilon}(s))+1\big) V'_{i,\mu-\varepsilon}(s)\, \dd s\bigg),
\end{align*}
where
\begin{align*}
  V_{i,\gamma}(t) \triangleq \int_0^t \Big(1+\log \mathbf{\Delta}_{i,\gamma}(s) + \mathsf{d}(s)^{-3} \Big)\dd s, \quad \gamma\in (0,1).  
\end{align*}
Define 
\begin{align*}
  f_i(t) \triangleq  V'_{i,\gamma}(t) = 1+\log \mathbf{\Delta}_{i,\mu-\epsilon}(t) + \mathsf{d}_i(t)^{-3},
\end{align*}
and $\mathcal{M}$ as in \eqref{def:fandM}. It follows from \eqref{eq:est-mini-dist-t} and the inequality above that for every $t\in [0,T]$ and $\epsilon>0$,
\begin{align*}
  f_i(t) \leq C\Big(\mathsf{g}(T)+\log^{\beta+\epsilon}(e+t)\int_0^t \mathcal{M}(f_i(s))\dd s\Big), 
\end{align*}
where $\mathsf{g}(\cdot)$ is  given by \eqref{def:gt}. 
The function
\begin{align*}
  R_i(t) \triangleq \frac{\mathsf{g}(T)}{\log^{\beta+\epsilon}(e+t)} + \int_0^t \mathcal{M}(f_i(s))\dd s 
\end{align*}
satisfies
\begin{align*}
  \frac{\dd R_i(t)}{\dd t} 
  & = -\frac{(\beta+\epsilon)\mathsf{g}(T)}{(e+t)\log^{\beta+\epsilon}(e+t)}  + \mathcal{M}(f_i(t))\\
  & \leq \mathcal{M}\big(C\log^{\beta+\epsilon}(e+t)R_i(t)\big)
  \leq  C\log^{\beta+2\epsilon}(e+t)\mathcal{M}(R_i(t)),
\end{align*}
which implies 
\begin{align*}
  \mathcal{H}(R_i(T)) - \mathcal{H}(R_i(0)) = \int_0^{T}\frac{R_i'(t) \dd t}{\mathcal{M}(R_i(t))} 
  \leq C(1+T)\log^{\beta+2\epsilon}(e+T),
\end{align*}
where $\mathcal{H}$ is defined in \eqref{def:Hr}.
Therefore, we conclude that 
\begin{align*}
  f_i(T) \leq C\log^{\beta+\epsilon}(e+T)R_i(T)\leq \log^{\beta+2\epsilon}(e+T) \mathcal{H}^{-1}\big(C(1+T)\log^{\beta+2\epsilon}(e+T)+ \mathcal{H}(\mathsf{g}(T))\big).
\end{align*}
This leads to \eqref{ineq:multi-est-1} and \eqref{ineq:multi-est-2},
and hence we complete the proof of Theorem \ref{thm:multi-patch} for the case $n=1$. 

Next, we address the case of higher regularity ($n \geq 2$) and aim to prove that $\partial D_i(t) \in C^{n,\mu-\varepsilon}$ for each $i \in \{1,\cdots,N\}$. The proof follows closely the single-patch case (Theorem~\ref{thm:propagation_regularity}), with additional arguments to account for interactions among patches.

To estimate $\partial_{W_i}^{k+1} u$, we use the decomposition \eqref{eq:u-decompose}
\begin{align*}
 \partial_{W_i}^{k+1} u(x,t) = 	\partial_{W_i}^{k+1} u_i(x,t) + \sum_{\ell\neq i}\partial_{W_i}^{k+1} u_\ell(x,t).
\end{align*}
The first term can be treated like the single-patch case. For the second term, since $D_\ell(t)\cap \Omega_i(t)=\emptyset$, we are still able to apply the expression
formula \eqref{eq:block_higher_tangential_derivative} and obtain
\begin{align*}
 \partial_{W_i}^{k+1} u_\ell(x)  &=\mathrm{p.v.} \int_{D_\ell(t)}
 \big(\partial_{W_i}^{k}W_i(x) -\partial_{W_i}^{k} W(y) \big) \cdot\nabla K(x-y) \dd y\\
 &\quad+a_\ell \sum_{j=2}^{k+1} \sum_{\substack{l_1\geq\cdots\geq l_j \geq 1 \\
 		l_1 + \cdots + l_j = k+1}}c^{k+1}_{l_1,\cdots,l_j}\int_{D_\ell(t)} \mathcal{D}_{W_i}^{l_1,\cdots,l_j}(x,y)\cdot \nabla^jK(x-y)\dd y\\
 & \triangleq \widetilde{\mathbf{G}}_{k+1}(x) + \Big(\partial_{W_i}^{k+1} u_\ell(x,t)-\widetilde{\mathbf{G}}_{k+1}(x)\Big),
 \end{align*}
 with 
 \begin{align*}
 	 \widetilde{\mathbf{G}}_{k+1}(x)=\mathrm{p.v.} \int_{D_\ell(t)}
 \big(\partial_{W_i}^{k}W_i(x) -\partial_{W_i}^{k} W(y) \big) \cdot\nabla K(x-y) \dd y. 
 \end{align*} 
 For $\widetilde{\mathbf{G}}_{k+1}$, we proceed similarly as in Lemma \ref{lem:higher_tangential_est}, and apply \eqref{eq:block_higher_tangential_derivative} to obtain
 \begin{align*}
 	\lVert \widetilde{\mathbf{G}}_{k+1}(t) \rVert_{C^\mu(\Omega_i(t))} \leq \lVert \partial_W^kW \rVert_{\dot{C}^\mu}\big(\mathsf{d}_i(t)^{-3}+1\big).
 \end{align*}
 For $\partial_{W_i}^{k+1} u_\ell(x,t)-\widetilde{\mathbf{G}}_{k+1}(x)$, the same proof of \eqref{eq:tangential_higher_est-2} yields
\begin{align*}
  \lVert \partial_{W_i}^{k+1} u_\ell(x,t)-\widetilde{\mathbf{G}}_{k+1}(x) \rVert_{C^\mu(\Omega_i(t))}
  \leq C\sum_{j=2}^{k+1} \sum_{\substack{l_1\geq\cdots\geq l_j \geq 1 \\
  l_1 + \cdots + l_j = k+1}} \Big(\prod_{s=1}^j 
  \|\partial_{W_i}^{l_s-1} W_i \|_{ \dot{C}^{\gamma_k}}\Big),
\end{align*}
Gathering the above estimates, we deduce the bound \eqref{eq:ukdiff}, where $1+\log \mathbf{\Delta}_{\gamma}(\cdot )$ is replaced by $1+\log\mathbf{\Delta}_{i,\gamma}(\cdot)+ \mathsf{d}_i(\cdot)^{-3}$.
 
The rest of the proof is identical to the single-patch case, with the additional treatment of the $\mathsf{d}_i(\cdot)^{-3}$ term same as the $n=1$ case in Theorem \ref{thm:multi-patch}. We conclude with the bound \eqref{ineq:multi-est-2}.
\end{proof}

\appendix

\section{Proof of Lemma \ref{lem:G-prop}}\label{sec:app}
Properties \eqref{eq:G-prop1} and \eqref{eq:G-prop1b} with $l=1,2$ have been proved in \cite[Lemma 2]{MTXX}. 
For \eqref{eq:G-prop1b} with general $l$ and \eqref{eq:G-prop2} can be proved analogously. Here we sketch the proof.

Starting from the explicit expression \eqref{eq:G-exp1} on $G$, we claim that for every $l\in \{1,2,\cdots,n+1\}$, 
the $l$-th order derivative $G^{(l)}$ can be expressed by
\begin{align}\label{def:G-l-deri}
  G^{(l)}(\rho) = \mathcal{G}_l(\rho) 
  + \frac{(-1)^l}{2\pi \rho^l} \int_0^\infty J_0(\rho r) M_l(r) \dd r,
\end{align}
where $\mathcal{G}_l(\rho)$ and $M_l(r)$ are iteratively defined as
\begin{align*}
  \mathcal{G}_l(\rho) = \mathcal{G}_{l-1}'(\rho) - \frac{l-1}{\rho} \big(G^{(l-1)}(\rho) - \mathcal{G}_{l-1}(\rho)\big), 
  \quad \mathcal{G}_0(\rho) = \mathcal{G}_1(\rho)= 0 , 
\end{align*}
\begin{align*}
  M_l(r) = M_{l-1}(r) + r M_{l-1}'(r),\quad M_0(r) = m'(r) ;
\end{align*}
or equivalently, 
\begin{align*}
  \forall\, l\in \{2,3,\cdots,n+1\},\quad \mathcal{G}_l(\rho) = \sum_{j=1}^{l-1} a_{j,l} \frac{G^{(j)}(\rho)}{\rho^{l-j}}, \quad 
  \textrm{$\{a_{j,l}\}_{1\leq j\leq l-1}$\;are iteratively defined}, 
\end{align*}
\begin{align*}
  \forall\, l\in \{1,2,\cdots,n+1\},\quad M_l(r) = \sum_{j=1}^{l+1} b_{j,l}\, r^{j-1} m^{(j)}(r), \quad b_{1,l} = b_{l+1,l}=1,
\end{align*}
and $ b_{j,l}=b_{j,l-1} + (j-1)b_{j,l-1} + b_{j-1,l-1}$ for every $j\in \{2,\cdots,l\}$.
Indeed, 
we calculate $G^{(l)}$ by differentiating $G^{(l-1)}$ and using the following integration by parts
\begin{align*}
  \int_0^\infty J_0'(\rho r) r M_{l-1}(r) \dd r & = \frac{J_0(\rho r) r M_{l-1}(r)}{\rho} \bigg|_{r=0}^{+\infty} 
  - \frac{1}{\rho} \int_0^\infty J_0(\rho r) \big(M_{l-1}(r) + r M_{l-1}'(r) \big) \dd r\\
  & =  - \frac{1}{\rho} \int_0^\infty J_0(\rho r) M_l(r) \dd r,
\end{align*}
where we have used the fact that 
\begin{align*}
  \lim_{r\rightarrow +\infty} \big|J_0(\rho r) r M_{l-1}(r)\big| 
  \leq C \lim_{r\rightarrow +\infty} r^{-\frac12}\cdot r \cdot m'(r) = 
  C\beta_1 \lim_{r\rightarrow +\infty} \frac{m(r)}{r^{1/2} \log r}=0,
\end{align*}
and
\begin{align*}
  \lim_{r\to 0^{+}}r^{l}m^{(l)}(r)=0,\quad\quad \lim_{r\rightarrow 0^+} r M_{l-1}(r) = 0, 
\end{align*}
applying the Mikhlin-H\"ormander condition in \eqref{H1}, \eqref{H2a}, \eqref{eq:m-prop2} and \eqref{eq:rm'(r)-limit}. 

Next, we prove \eqref{eq:G-prop1b} and \eqref{eq:G-prop2} by induction.
Note that \eqref{eq:G-prop1b}-\eqref{eq:G-prop2} with $l=1$ 
have already been proven in \cite[Lemma 2]{MTXX} (although it was only stated that
$|G'(\rho)|\leq C$ on $[\bar{c}_0,+\infty)$, it can easily be extended to
\eqref{eq:G-prop2} with $l=1$; see also below).
Let $k\in \{1,\cdots,n\}$. Suppose that \eqref{eq:G-prop1b} and \eqref{eq:G-prop2} 
with each $l\in \{1,\cdots,k\}$ hold, 
we intend to show that they also hold for the $l=k+1$ case.
Let $\chi(\xi)= \chi(|\xi|)\in C^\infty_c(\RR^2)$ be a smooth radial function such that
\begin{align*}
  \chi\equiv 1,\;\; \textrm{on}\;\{|\xi|\leq 1\},\qquad \chi\equiv 0,\;\; \textrm{on}\;\{|\xi|\geq 2\},\qquad
  0\leq \chi\leq 1.
\end{align*}
Thanks to identity \eqref{eq:G-exp1} with $|x|=\rho$ and formulas \eqref{def:G-l-deri}, we have
\begin{align*}
  G^{(k+1)}(\rho) & = \mathcal{G}_{k+1}(\rho) 
  + \frac{(-1)^{k+1}}{2\pi\rho^{k+1}} 
  \int_0^{\infty}J_0(\rho r)\chi(\rho r) M_{k+1}(r) \dd r\\
  & \quad + \frac{(-1)^{k+1}}{(2\pi)^2 \rho^{k+1}} 
  \int_{\RR^2} e^{ix\cdot \xi}(1-\chi(\rho |\xi|))
  \frac{M_{k+1}(|\xi|)}{|\xi|}\dd \xi \\
  & \triangleq \mathbf{I}_1+\mathbf{I}_2 + \mathbf{I}_3.
\end{align*}
For the first term, noting that
$\mathbf{I}_1 = \sum_{j=1}^k a_{j,k+1} \frac{G^{(j)}(\rho)}{\rho^{k+1-j}}$,
we immediately use \eqref{H1} and the induction assumptions to deduce that
\begin{align*}
  |\mathbf{I}_1|\leq C \frac{1 + m(\rho^{-1})}{\rho^{k+1}}.
\end{align*}
For the term $\mathbf{I}_2$, due to that the Bessel function $J_0(\cdot)$ satisfies that $0<J_0(r)\leq 1$ for $r\in [0,2]$,
and using \eqref{H1}, we find that
\begin{align*}
  |\mathbf{I}_2| & \leq \frac{1}{2\pi \rho^{k+1}} \int_0^\infty J_0(\rho r) \chi(\rho r)
  \Big(m'(r) + |b_{2,k+1}|\,r\,|m''(r)| +\cdots +|b_{k+2,k+1}|\,r^{k+1}\, |m^{(k+2)}(r)| \Big) \dd r \\
  & \leq \frac{C}{\rho^{k+1}} \int_0^{2\rho^{-1}} J_0(\rho r) m'(r) \dd r
  \leq \frac{C }{\rho^{k+1}} \Big(m(2\rho^{-1})  - m(0^+) \Big).
\end{align*}
For the last term $\mathbf{I}_3$, through the integration by parts and \eqref{H1}, observe that
\begin{align*}
  |\mathbf{I}_3| & = \frac{1}{(2\pi)^2 \rho^{k+3}}\Big|\sum_{j=1}^{k+2} b_{j,k+1}\int_{\RR^2}e^{ix\cdot \xi}\Delta_{\xi}
  \Big((1-\chi(\rho |\xi|))\frac{|\xi|^{j-1} m^{(j)}(|\xi|)}{|\xi|}\Big)\dd \xi \Big| \\
  & \leq \frac{C}{\rho^{k+3}} \sum_{j=1}^{k+2} \bigg(\int_{\rho^{-1}\leq |\xi|\leq 2\rho^{-1}} \big(\rho^2 + \rho |\xi|^{-1}\big)
  \frac{|\xi|^{j-1} m^{(j)}(|\xi|)}{|\xi|} \dd \xi 
  + \int_{|\xi|\geq \rho^{-1}} \Delta_\xi \Big(\frac{|\xi|^{j-1}m^{(j)}(|\xi|)}{|\xi|}\Big) \dd \xi \bigg) \\
  & \leq \frac{C}{\rho^{k+3}} \bigg(\int_{\rho^{-1}\leq |\xi|\leq 2\rho^{-1}} \big(\rho^2 + \rho |\xi|^{-1} \big)
  \frac{m'(|\xi|)}{|\xi|} \dd \xi 
  + \int_{|\xi|\geq \rho^{-1}} \frac{m'(|\xi|)}{|\xi|^3}\dd \xi \bigg).
\end{align*}
By virtue of the fact that $r\mapsto m'(r)$ is non-increasing and $r \,m'(r)\leq m(r)$ for $r>0$ large (from \eqref{H2a}-\eqref{eq:Osgood-cond}), 
we infer that for $\rho>0$ small enough, i.e., $\rho\in (0,\bar{c}_0]$,
\begin{align*}
  |\mathbf{I}_3| \leq \frac{C}{\rho^{k+2}} m'(\rho^{-1}) \leq \frac{C}{\rho^{k+1}} m(\rho^{-1});
\end{align*}
while for every $\rho \in [\bar{c}_0,\infty)$,
\begin{align*}
  |\mathbf{I}_3| & \leq \frac{C}{\rho^{k+3}} \bigg( \rho^2 \int_{\rho^{-1}}^{2\rho^{-1}} m'(r)\dd r 
  + \int_{\rho^{-1}}^{\bar{c}_0^{-1}} \frac{m'(r)}{r^2} \dd r + \int_{\bar{c}_0^{-1}}^\infty \frac{m'(r)}{r^2} \dd r \bigg) \\
  & \leq \frac{C}{\rho^{k+3}} \bigg( \rho^2 m(2 \rho^{-1})  
  + \bar{c}_0^2 m(\bar{c}_0^{-1}) + m(\bar{c}_0^{-1}) \int_{\rho^{-1}}^{\bar{c}_0^{-1}} \frac{1}{r^3} \dd r 
  + m'(\bar{c}_0^{-1}) \int_{\bar{c}_0^{-1}}^\infty \frac{1}{r^2} \dd r \bigg)  \leq \frac{C}{\rho^{k+1}}.
\end{align*}
Gathering the above estimates leads to the inequalities 
\eqref{eq:G-prop1b} and \eqref{eq:G-prop2} with $l=k+1$, as desired.
Hence, we complete the proof of \eqref{eq:G-prop1b} and \eqref{eq:G-prop2}.

\end{document}